\documentclass[runningheads]{article}
\usepackage[margin=1.8cm]{geometry}
\usepackage{anyfontsize}

\usepackage{graphicx} 
\usepackage{amsmath}
\usepackage{amssymb}
\usepackage{amsthm}
\usepackage[short]{optidef}
\usepackage{algorithm}
\usepackage{algpseudocode}
\usepackage{float}
\usepackage{rotating}
\usepackage{enumitem}
\usepackage{hyperref}
\usepackage{cleveref}
\usepackage{enumitem}
\usepackage{authblk}

\newtheorem{theorem}{Theorem}[section]
\newtheorem{lemma}[theorem]{Lemma}
\newtheorem{corollary}[theorem]{Corollary}

\newtheorem{conjecture}[theorem]{Conjecture}
\newtheorem{definition}{Definition}
\newtheorem{observation}{Observation}
\newtheorem{remark}{Remark}

\newtheoremstyle{named}{}{}{\itshape}{}{\bfseries}{.}{.5em}{}
\theoremstyle{named}
\newtheorem*{process}{Process}

\newtheoremstyle{concentrationinequality}{}{}{\itshape}{}{\bfseries}{.}{.5em}{\thmnote{#3's }#1}
\theoremstyle{concentrationinequality}
\newtheorem{inequality}{Inequality}

\newtheoremstyle{concentrationbound}{}{}{\itshape}{}{\bfseries}{.}{.5em}{\thmnote{#3 }#1}
\theoremstyle{concentrationbound}
\newtheorem{bound}{Bound}

\newtheoremstyle{oldtheorem}{}{}{\itshape}{}{\bfseries}{.}{.5em}{#1\thmnote{ #3}}
\theoremstyle{oldtheorem}
\newtheorem*{theorem*}{Theorem}
\newtheorem*{lemma*}{Lemma}
\newtheorem*{conjecture*}{Conjecture}
\newtheorem*{observation*}{Observation}

\newtheoremstyle{case_style}{}{}{}{}{\bfseries}{}{.5em}{\thmnote{#3 }#1}
\theoremstyle{case_style}
\newtheorem*{case}{Case}
\AtBeginEnvironment{case}{
  \pushQED{\qed}
}
\AtEndEnvironment{case}{\popQED}

\newtheoremstyle{claim_style}{}{}{}{}{\bfseries}{}{.5em}{\thmnote{#3 }#1}
\theoremstyle{claim_style}
\newtheorem*{claim}{Claim}
\AtBeginEnvironment{claim}{
  \pushQED{\qed}
}
\AtEndEnvironment{claim}{\popQED}

\usepackage{tikz}
\usetikzlibrary{patterns}
\usetikzlibrary{hobby,arrows,backgrounds,calc,trees}
\usetikzlibrary {arrows.meta}

\definecolor{my_green}{HTML}{99c04c}
\definecolor{my_red}{HTML}{ff8ea2}
\definecolor{my_blue}{HTML}{6c94b0}
\definecolor{my_grey}{HTML}{bdc3c7}

\begin{document}

\title{On the Critical Window for Adaptable 2-Colorability}

\setcounter{footnote}{1} 
\author{Thomas Snow\thanks{Department of Computer Science, University of Toronto, Toronto, Canada. Supported by the Ontario Graduate Scholarship (OGS). \\Current affiliation: School of Computer Science and Electrical Engineering, University of Ottawa, Ottawa, Canada. Email: \href{mailto:tsnow042@uottawa.ca}{\texttt{tsnow042@uottawa.ca}.}}}

\maketitle

\begin{abstract}
    We determine a sharp threshold for the adaptable 2-colorability of a random graph equipped 
    with a uniformly random, not necessarily proper, red/blue coloring of the edges. 
    To accomplish this, we characterize a family
    of subgraphs along with edge colorings whose inclusion or exclusion determines adaptable 
    $2$-colorability. We further show that above the threshold, a long path with alternating 
    edge colors is formed. We use this path to prove the existence of such a subgraph in the supercritical regime. 
    We then provide and prove symmetric bounds on the critical window for $2$-adaptable colorability. 
    Particularly, we prove bounds matching that of the critical windows for 
    the giant component in the Erdős–Rényi random graph model as well as the satisfiability of a 
    random $2$-SAT instance.
    Finally, we show that below the critical window, the solution space of 
    adaptable $2$-colorings remains connected, that is one can travel from one adaptable $2$-coloring to 
    another by a sequence of $2$-colorings which differ on $\mathcal{O}(\log{n})$ many vertices.
\end{abstract}

\section{Introduction}
For a graph $G=(V,E)$ with vertex set $V$ and edge set $E$, a \textit{proper} $k$-coloring of $G$ is an assignment of colors from $\{1, \dots, k\}$ to each vertex such that no two adjacent vertices receive the same color. 
In particular, each edge forbids $k$ pairs of colors of being assigned to its endpoints, namely the pairs $(1,1), (2,2), \dots, (k,k)$. We consider a similar model in which each edge $e$ is assigned a color $c(e)$ and instead of forbidding all $k$ pairs of colors, $e$ now only forbids the pair $(c(e), c(e))$ from being assigned to its endpoints. 

More formally, for a graph $G=(V,E)$ with vertex set $V$ and edge set $E$, the \textit{adaptable chromatic number}, $\chi_a(G)$ is defined to be the smallest $k$ such that for any, not necessarily proper, edge coloring $c: E \xrightarrow[]{} [k] \coloneqq \{1, \dots, k\}$, 
there exists a vertex coloring $\sigma : V \xrightarrow[]{} [k] \coloneqq \{1, \dots, k\}$ such that for every edge $uv \in E$ either $\sigma(u) \neq c(uv)$ or $\sigma(v) \neq c(uv)$, in which case we call $\sigma$ \textit{adapted} to $c$. In other words, for every assignment of colors to the edges, one can adapt the edge coloring into a vertex coloring such that no edge has both endpoints assigned the same color as the edge; that is each edge forbids one pair of colors from being assigned to its endpoints. We call this the \textit{adversarial model of adaptable k-colorability}. In this paper, we will primarily work with a separate but closely related model where we are given a graph $G = (V,E)$ along with an edge coloring $c: E \xrightarrow[]{} [k]$ and we wish to determine whether $c$ can be adapted to a valid vertex coloring, under the same conditions in the adversarial model. If this is the case, we say that $G$ is \textit{adaptably k-colorable under c}. Particularly, we are interested in determining whether a standard Erdős–Rényi random graph equipped with a uniformly random, not necessarily proper, $k$-edge coloring $c$ is adaptably $k$-colorable under $c$.

\begin{definition} \label{rand graph prob model}
    In $\mathcal{G}(n,p,k)$, each possible edge is included independently with probability $p$ and assigned a uniformly random color from $[k]$.
\end{definition}

This is precisely the same model as first sampling a graph from $\mathcal{G}(n,p)$ and then assigning each realized edge a color uniformly and independently at random from $k$. Particularly, for each unordered pair of vertices $(u,v)$ and $i \in [k]$ we have that, for $G = (V,E)$ sampled from $\mathcal{G}(n,p,k)$,

\begin{equation*}
    \mathbb{P}[uv \in E[G] \wedge uv \gets i] = \mathbb{P}[uv \gets i \mid uv \in E(G)] \mathbb{P}[uv \in E(G)] = \frac{p}{k} \,.
\end{equation*}

In particular, this is the same as if each unordered pair of vertices $(u,v)$ receives a 2-tuple $(X_{uv}, Y_{uv})$ of random variables, where $X_{uv} = 1$ with probability $p$ and $0$ otherwise and $Y_{uv}$ receives $i$ with probability $1/k$ for all $i \in [k]$. Where $X_{uv}$ denotes the random indicator variable for the inclusion of $uv$ in $E$ and $Y_{uv}$ denotes the color assigned to the pair $(u,v)$, which if $X_{uv} = 1$ denotes the color assigned to $uv$.                                                                                                                                                                                                 
\begin{definition} \label{rand graph edge model}
    In $\mathcal{G}(n,m,k)$, a graph is chosen uniformly at random from the set of all graphs on $n$ vertices with $m$ edges colored from $[k]$.
\end{definition}

As in the $\mathcal{G}(n,p,k)$ model, this is precisely the same as first sampling a graph from $\mathcal{G}(n,m)$ and then assigning each edge a uniformly random color from $[k]$.

We say a sequence of events $A_n$ occurs \textit{with high probability (w.h.p.)} if, 
\begin{equation*}
    \lim_{n \xrightarrow[]{} \infty} \mathbb{P}[A_n] = 1 \,.
\end{equation*}

We say an event $A$ is \textit{monotone} if the following implication holds: If a graph $H$ has $A$ and $H$ is a subgraph of $G$, $H \subseteq G$, then $G$ also has $A$.

We wish to define an analog of monotone events with respect to edge colorings. 
Given a graphs $H$ and $G$ along with edge colorings $c_H : E(H) \xrightarrow[]{} \{1, \dots, k\}$ and $c_G : E(G) \xrightarrow[]{} \{1, \dots, k\}$ 
we say the graph coloring pair $(H, c_H)$ is a subgraph of $(G, c_G)$, denoted $(H, c_H) \subseteq (G, c_G)$ if there exists an instance of $H$ in $G$ in which $c_G$ agrees with $c_H$ on all edges.  
We further say an event $A$ is \textit{monotone under edge colorings} if the following implication 
holds: If a graph coloring pair $(H, c_H)$ has $A$ and $(H, c_H) \subseteq (G, c_G)$ then 
$(G, c_G)$ has $A$.

Its natural to believe that there is a form of equivalence between the two models $\mathcal{G}(n,p,k)$ and 
$\mathcal{G}(n,m,k)$ with respect to monotone properties under edge colorings, similar to the relationship 
between $\mathcal{G}(n,p)$ and $\mathcal{G}(n,m)$ under monotone properties in the classical sense. We defer the reader 
to Appendix~\ref{appendix: two models} noting that we will primarily work in the $\mathcal{G}(n,p,k)$ model.

\subsection{Related Work} \label{sec: related work}

\subsubsection*{Adaptable Coloring}

In the work of Hell and Zhu \cite{adapt-cols}, they introduced the adversarial adaptable chromatic number of graphs. In particular, they give the following characterization of graphs with $\chi_a(G) \leq 2$. First, we need to introduce a couple of definitions. 

\begin{definition} \label{def: bicycle}
    A bicycle is a graph composed of a path $P = v_1 v_2 \cdots v_s$ along with edges $v_1v_i$ and $v_jv_s$ for $i,j \in \{1, \dots, s\}$. If $i \leq j$ then we call the bicycle proper. If $i > j$ then we call the bicycle improper. Moreover, if the cycles $C_1 = v_1v_2 \cdots v_iv_1$ and $C_2 = v_j v_{j+1} \cdots v_sv_j$ are both of odd length, we call the bicycle odd.
\end{definition}

If the cycles forming the faces of a planar embedding of a $K_4$-subdivision 
(including the outer face) are all odd, we call it an \textit{odd edge-$K_4$}. 
We are now ready to state the characterization given in \cite{adapt-cols}.

\begin{theorem}[Theorem 2.1 in \cite{adapt-cols}] \label{thm: adversary adpt 2 col}
    The following statements are equivalent for a connected graph $G$,
    \begin{enumerate}[label=\alph*.]
        \item $\chi_a(G) \leq 2$.
        \item G does not have a proper odd bicycle or an odd edge-$K_4$.
        \item There is an edge $e$ such that $G - e$ is bipartite. 
    \end{enumerate}
\end{theorem}

A consequence of Theorem~\ref{thm: adversary adpt 2 col} is that $p = \frac{1}{n}$ is a sharp threshold for the adversarial model of adaptable 2-colorability in $\mathcal{G}(n,p)$, as stated in the following Theorem. 

\begin{theorem}
    For any constant $\epsilon > 0$ the following holds w.h.p.,
    \begin{enumerate}[label=\alph*.]
        \item $\chi_a(\mathcal{G}(n,p = \frac{1 - \epsilon}{n})) \leq 2$.
        \item $\chi_a(\mathcal{G}(n,p = \frac{1 + \epsilon}{n})) > 2$.
    \end{enumerate}
\end{theorem}

This follows directly from the giant component threshold.
It is well known from Bollobás \cite{bolobas-rand-graph} and Erdős and Rényi \cite{ErdosRenyiRandGraphs} that below the giant component threshold, w.h.p., all connected components are either trees or graphs with precisely one cycle and hence by Theorem~\ref{thm: adversary adpt 2 col} $\chi_a \leq 2$. 
Furthermore, above the threshold, a giant component forms in which case, one may believe that almost surely an odd bicycle is formed and thus by Theorem~\ref{thm: adversary adpt 2 col} $\chi_a > 2$.
To see this, consider the following two step process. First, for $p = \frac{1 + \epsilon}{n}$ for some $\epsilon > 0$ it is well known (see \cite{longest-path-simple} and \cite{longest-path}) that $G$ has w.h.p. a path $P$ of length $\Theta(n)$.
We now consider increasing $p$ by $\epsilon/n$ and claim that doing so introduces w.h.p. edges along $P$ which form an odd bicycle. Notice that if edges already exist between non-adjacent vertices in $P$, this can only increase our probability. Therefore, 
we may assume no such edges exist. Standard calculations easily imply that such edges 
are introduced w.h.p. upon increasing $p$.

We will utilize this approach in Section~\ref{sec: a sharp threshold} with respect to the $\mathcal{G}(n,p,2)$ 
model. However, we will need to be a bit more careful as upon increasing $p$, we wish to introduce edges of a specific 
color. If an edge exists in the original graph, it may decrease the probability of an edge of the desired 
color in the final graph. 
Therefore, we can no longer make the assumption that certain edges do not exist prior to increasing $p$.

\subsubsection*{Correspondence Coloring}

Before introducing Correspondence Coloring, we first introduce the notion of List Coloring. 
We are given a graph $G = (V, E)$ as well as a list assignment $L$. That is $L$ is a function that maps each
vertex $v$ to a list $L(v)$ of colors which are available for $v$. We now ask the following question: given $G$ and $L$,
is it possible to assign each vertex a color from its respective list, such that no two adjacent vertices are assigned the same color? 
We call such an assignment of colors an $L$-coloring of $G$. Note that an $L$-coloring of $G$ is a proper coloring of $G$ on 
the colors $\bigcup_{v \in V} L(v)$. We call $G$ $k$-choosable if an $L$-coloring exists for every list assignment
$L$ satisfying $|L(v)| \geq k$ for all $v$.
The main question regarding list colorings is the following: 
What is the smallest $k$ such that $G$ is $k$-choosable? 
For more on list coloring we defer the reader to Chapter 5.4 in \cite{Diestel}.

Correspondence coloring (or DP-coloring) is a generalization of adaptable coloring and list coloring introduced by 
Dvo\v r\'ak and Postle \cite{Correspondence-Coloring}. 
Let $G = (V,E)$ be a graph and $L$ be a list assignment. 
For every edge $uv$ we add a list of constraints to the possible colors assigned to the 
pair $u$ and $v$; that is, each edge is given a partial matching that forbids certain combinations of 
colors being assigned to $u$ and $v$. More formally, we are given a function $C$ which maps each edge $e$ to a 
partial matching between $\{u\} \times L(u)$ and $\{v\} \times L(v)$. That is, for colors $c_1$ and $c_2$ if 
$(u, c_1)$ and $(v, c_2)$ are adjacent in $C(uv)$ then when assigning vertices colors from their lists, 
we cannot assign both $u$ the color $c_1$ and $v$ the color $c_2$. As in \cite{Correspondence-Coloring},
we call an $(L,C)$-coloring an assignment $\phi$ such that $\phi(v) \in L(v)$ for all $v$ and for each edge $uv$,
$(u, \phi(u))$ and $(v, \phi(v))$ are not adjacent in $C(uv)$. It is not hard to see that
this variation of colorings generalizes both list coloring as well as adaptable colorings.

Correspondence coloring has applications in list coloring on restricted classes of graphs, as seen in \cite{Correspondence-Coloring}.

\subsubsection*{Thresholds and Critical Windows}

The study of random graphs and their properties dates back to the work of Erdős and Rényi \cite{ErdosRenyiEarlyPaper}. 
Erdős and Rényi observe that the structure of $\mathcal{G}(n,m)$ undergoes a major change at 
$m \sim \frac{1}{2}n$, which corresponds to $p \sim 1/n$ in $\mathcal{G}(n,p)$. 
In fact, for constant $\epsilon > 0$ and $p = \frac{1 - \epsilon}{n}$ w.h.p. 
the graph consists ``simple" components, components with at most one cycle, the largest of which has size $\mathcal{O}(\log{n})$. However,
when $p = \frac{1 + \epsilon}{n}$ these components merge into a single component of linear size called the \textit{giant component}. The structure 
and evolution of the giant component has been greatly studied. Bollobás \cite{bolobas-rand-graph} gives a accurate depiction
of the emergence of the giant component. 

We call $p=1/n$ the \textit{threshold} for the giant component in fact, 
this threshold is \textit{sharp}. In general, for a monotone property $\mathcal{P}$, we call 
$p = p(n)$ a \textit{threshold} for $\mathcal{P}$ if 
\begin{equation*}
\lim_{n \rightarrow \infty} \mathbb{P}[\mathcal{G}(n, p') \ has \ \mathcal{P}] = 0 \ if \ p'/p \rightarrow 0 \qquad and \qquad
\lim_{n \rightarrow \infty} \mathbb{P}[\mathcal{G}(n, p') \ has \ \mathcal{P}] = 1 \ if \ p'/p \rightarrow \infty \,.
\end{equation*}
Furthermore, we call $p = p(n)$ a \textit{sharp threshold} for $\mathcal{P}$ if for every $\epsilon > 0$,
\begin{equation*}
\lim_{n \rightarrow \infty} \mathbb{P}[\mathcal{G}(n, p') \ has \ \mathcal{P}] = 0 \ if \ p'/p \leq 1 - \epsilon \qquad and \qquad
\lim_{n \rightarrow \infty} \mathbb{P}[\mathcal{G}(n, p') \ has \ \mathcal{P}] = 1 \ if \ p'/p \geq 1 + \epsilon \,.
\end{equation*}

What if we take $\epsilon = \epsilon(n) \rightarrow 0$? With respect to the giant component,
we have the following characterization due to Bollobás \cite{bolobas-rand-graph}: for $\lambda_n \rightarrow \infty$ arbitrarily slowly the largest component of 
$\mathcal{G}(n,p= \frac{1 - \lambda_n n^{-1/3}}{n})$ is w.h.p. of size $\Theta(n^{2/3}\lambda_n^{-2}\log{\lambda_n})$ and all components are simple. 
Whereas, for $\mathcal{G}(n,p= \frac{1 + \lambda_n n^{-1/3}}{n})$ the largest component is w.h.p of size $\Theta(\lambda_n n^{2/3})$ and is no longer simple and all 
other components are simple and of size $\mathcal{O}(n^{2/3}\lambda_n^{-2}\log{\lambda_n})$. 
We call this the \textit{critical window} for the giant component. Similar results have been shown for 
other classes of problems such as random instances of $2$-SAT. Bollobás et al. \cite{2-sat-critical-window} provide the 
critical window for $2$-SAT. In particular, Bollobás et al. show that for a random instance of $2$-SAT with
$n$ variables and $m$ clauses, if the ratio $m/n \leq 1 - \lambda_n n^{-1/3}$ then the instance is w.h.p. 
satisfiable; whereas, if $m/n \geq 1 + \lambda_n n^{-1/3}$ then the instance is w.h.p. not satisfiable.
In this paper, we prove a similar phenomenon with respect to adaptable $2$-colorability.

For more on random graphs and their properties we refer the reader to the following texts \cite{IntroRandGraphsFrieze}, \
\cite{Connected_component_rand_graphs}, and Chapter 11 in \cite{TheProbabilisticMethod}.

\subsection{Results}
In Section~\ref{sec: characterizing 2-adapt col}, we prove a similar condition to that in \cite{adapt-cols} for testing whether a 
graph edge-coloring pair $(G,c)$ is adaptably 2-colorable. 
We introduce a family of subgraphs which we call \textit{alternating bicycles} 
(formally defined in Definition~\ref{def: alt bicycle}), and show
that their inclusion or exclusion loosely determines adaptable 
$2$-colorability, see Lemmas~\ref{lma: non-adapt 2-col condition} and \ref{lma: improper alternating 
bicycle}. 
Our approach begins by considering a graph edge-coloring pair $(G,c)$ that is not adaptably $2$-colorable
as well as a non-proper coloring minimizing the number of violated edges. Exploiting this minimality, we 
explore the graph around a violated edge to construct an alternating bicycle. 

We then show that w.h.p. for $\epsilon > 0$ and $p = 2(1 + \epsilon)/n$, $(G,c)$ contains a path of linear 
length in which the edges alternate colors. This is achieved by adapting the Depth First Search approach 
introduced by 
Krivelevich and Sudakov \cite{longest-path-simple} for finding long paths in the classical $\mathcal{G}(n,p)$ model.
Utilizing these structural results, in Section~\ref{sec: a sharp threshold} we 
establish a sharp threshold for adaptable $2$-colorability.

\begin{theorem} \label{thm: a sharp threshold}
    For any constant $\epsilon > 0$:
    \begin{enumerate}[label=\alph*.]
        \item $\mathcal{G}(n, p = \frac{2 - \epsilon}{n}, 2)$ is w.h.p. adaptably 2-colorable.
        \item $\mathcal{G}(n, p = \frac{2 + \epsilon}{n}, 2)$ is w.h.p. not adaptably 2-colorable.
    \end{enumerate}
\end{theorem}

Using the first moment method, we show that when $p < 2/n$, $\mathcal{G}(n, p, 2)$ w.h.p. 
does not contain an alternating bicycle which by our characterization implies adaptable 
2-colorability. This argument is similar to the proof of a sharp threshold for 2-SAT by Chv\'atal and Reed in 
\cite{Mick-gets-some-2sat-thresh}. In contrast, when $p > 2/n$,  
$\mathcal{G}(n, p, 2)$ w.h.p. contains such a proper alternating odd bicycle 
whose inclusion implies the graph is not adaptably $2$-colorable. 
We first find a long path with alternating edge colors then by sprinkling a small 
fraction of additional edges, an alternating bicycle is formed along the path. 

Having established the sharp threshold, we turn to a more refined analysis of the 
subcritical and supercritical regimes. In particular, we establish the following bounds on the 
size of the scaling/critical window.

\begin{theorem} \label{thm: scaling window lower bound}
    For $\lambda_n \xrightarrow[]{} \infty$ arbitrarily slowly, $\mathcal{G}(n, p = 2(1 - \lambda_n n^{-1/3})/n, 2)$ is w.h.p. adaptably 2-colorable.
\end{theorem}

\begin{theorem} \label{thm: scaling window upper bound}
    For $\lambda_n \xrightarrow[]{} \infty$ arbitrarily slowly, $\mathcal{G}(n, p = 2(1 + \lambda_n n^{-1/5})/n, 2)$ is w.h.p. not adaptably 2-colorable.
\end{theorem}

We prove Theorems~\ref{thm: scaling window lower bound} and 
\ref{thm: scaling window upper bound} by refining the techniques used in 
Theorem~\ref{thm: a sharp threshold}. That is we consider 
$\epsilon = \epsilon(n) \xrightarrow[]{} 0$ and how fast it can converge.

In Section~\ref{sec: supercritical improvement} we prove a symmetric upper bound to that in 
Theorem~\ref{thm: scaling window lower bound}. Thus matching the critical windows of the giant component and $2$-SAT. 
That is we prove the following.

\begin{theorem} \label{thm: scaling window upper bound 2}
    For $\lambda_n \xrightarrow[]{} \infty$ arbitrarily slowly with $\lambda_n \ll n^{1/3}$, $\mathcal{G}(n, p = 2(1 + \lambda_n n^{-1/3})/n, 2)$ is w.h.p. not adaptably 2-colorable.
\end{theorem}

We introduce a family of subgraphs 
which act as ``seeds" for forming alternating odd bicycles.
We call these subgraphs \textit{hourglasses}. A similar family of subgraphs, were introduced 
by Bollobás et al. \cite{2-sat-critical-window} in the directed graph model. 
We adapt their constructions to the model $\mathcal{G}(n,p,2)$.
However, their techniques do not directly carry over to this model due to the dependence on the edge 
colors. Their work serves as inspiration and some of the techniques used follow 
from \cite{2-sat-critical-window}. Our approach builds upon their work and develops techniques 
to address this dependence.

In particular, we will show that
in the subcritical regime, $p = \frac{2(1 - \lambda_n n^{-1/3})}{n}$, w.h.p. many 
vertex disjoint hourglasses exist, which we identify through a systematic search procedure. 
As we increase $p$ to the 
supercritical regime, $p = \frac{2(1 + \lambda_n n^{-1/3})}{n}$, w.h.p. among the hourglasses 
and the newly added edges, an alternating odd bicycle is produced. 
To accomplish this, we will treat the hourglasses as vertices in an auxiliary  
graph with edges determined by those introduced upon increasing $p$. We proceed to show 
this auxiliary graph is distributed as $\mathcal{G}(N,p',2)$ for some $p' > 2/N$, allowing us to apply
Theorem~\ref{thm: a sharp threshold} to obtain an alternating bicycle which we then map back to the 
original graph. 
This differs from the $2$-SAT case in which Bollobás et al. show that 
the many hourglasses merge into a large hourglass then through another 
round of sprinkling, produce a certificate for unsatisfiability. 
Applying Theorem~\ref{thm: a sharp threshold} to the auxiliary graph simplifies the proof and removes the need 
for the additional round of sprinkling.

Finally in Section~\ref{sec: connectivity of sol space}, we consider the geometry of the solution space. 
Utilizing similar techniques developed in
Section~\ref{sec: characterizing 2-adapt col}, we show that below the critical window, 
the solution space forms a single connected component. 
Namely, for $p < 2/n$ the space of adapted colorings for $\mathcal{G}(n,p,2)$ 
remains connected right up to the lower bound of the critical window stated in 
Theorem~\ref{thm: scaling window lower bound}. Precisely, for any two vertex colorings 
$\sigma$ and $\tau$ adapted to $c$, there exists a sequence, or path, of vertex colorings 
adapted to $c$ that starts at $\sigma$ and ends at $\tau$ such that adjacent colorings in the 
path agree on almost all vertices. 
Precisely we prove the following Theorem.

\begin{theorem} \label{theorem: connected sol space}
	Let $\epsilon > 0$ be a constant. For $(G,c)$ sampled from $\mathcal{G}(n,p = \frac{2-\epsilon}{n},2)$ w.h.p., 
    any two vertex colorings $\sigma$ and $\tau$ adapted to $c$ there exists w.h.p. a sequence of colorings adapted to $c$, $\sigma = \sigma_1, \sigma_2, \sigma_3, \cdots, \sigma_l = \tau$ such that for all $i < l$: 
	\begin{equation*}
		\left| \{ v\in V \mid \sigma_i(v) \neq \sigma_{i+1}(v) \} \right| = \mathcal{O}(\log{n})
	\end{equation*}
\end{theorem}

Given an initial coloring $\rho$ and a target coloring $\tau$ we form an intermediate coloring 
by choosing a vertex on which they disagree and swap the color assigned to it in $\rho$.
Naturally, this can produce violations. We continue to swap the color of vertices in $\rho$ 
to correct these violations until we achieve a valid coloring adapted to $c$.
We show this procedure is contained within a graph branching process run on 
$\mathcal{G}(n,p/2)$ in the subcritical regime where w.h.p. all connected components are of size 
$\mathcal{O}(\log{n})$ (see Chapter 11 in \cite{TheProbabilisticMethod}).

For more on this notion of a clustering and the evolution of the solution space we refer the reader to \cite{sol-space-geo}, \cite{sol-space-rs}, and \cite{alg-barriers-phases}.
As we will see in Section~\ref{sec: connectivity of sol space}, the distance between colorings 
in the sequence defined in Theorem~\ref{theorem: connected sol space} is dominated by the distribution
of the size of connected components in $\mathcal{G}(n,p/2)$. 
This gives us an accurate depiction of the evolution of the solution space. 
Particularly, for constant
$\epsilon > 0$ and $p = (2 - \epsilon)/n$ we need only allow $\mathcal{O}(\log{n})$ 
distance between colorings in the sequence.

\section{Characterizing Two-Adaptable Colorability} \label{sec: characterizing 2-adapt col}

In this section, we prove a similar statement to Theorem~\ref{thm: adversary adpt 2 col} proved in \cite{adapt-cols}. In the following, we shall assume $G = (V, E)$ is a graph and $c: E \xrightarrow[]{} \{red, blue\}$ is an edge coloring and we denote the graph edge coloring pair by $(G, c)$. 
The following definitions and notation are with respect to the edge coloring $c$; however, when the edge coloring is obvious, we drop the ``under $c$". 

First, for $A \subseteq V$ we let $N(A)$ denote the \textit{neighborhood} of $A$, 
that is all vertices in $V \setminus A$ incident to a vertex in $A$. 
We also let $\delta(A)$ denote the \textit{cut} induced by $A$, that is all edges with precisely one endpoint in $A$.
Note that if $A$ is a singleton ${u}$, for ease of notation we let $N(u) \coloneqq N(\{u\})$ and $\delta(u) \coloneqq \delta(\{u\})$. We will also require analogs of $N(\cdot)$ with respect to specific edge colors and vertex colors, under a specified vertex coloring. Particularly, we will let $N_\gamma(A)$ denote the vertices in $V \setminus A$ which are incident to a vertex in $A$ by an edge of color $\gamma$. Moreover, for a coloring of the vertices $\sigma : V \xrightarrow[]{} \{red, blue\}$ we let $N_\gamma^\sigma (A)$ denote the vertices in $V \setminus A$ which are colored $\gamma$ under $\sigma$ and are incident to a vertex in $A$ by an edge of color $\gamma$. Formally we have, 

\begin{definition}
    Given a graph $G$ along with edge coloring $c: E \xrightarrow[]{} \{red, blue\}$ and vertex coloring $\sigma: V \xrightarrow[]{} \{red, blue\}$, for $\gamma \in \{red,blue\}$, and $A \subseteq V$ let us define the following notation,
    \begin{itemize}
        \item $N_\gamma(A) \coloneqq \{v \in V \setminus A \mid \exists u \in A \ s.t. \ uv \in E, c(uv) = \gamma\}$
        \item $N^\sigma_\gamma(A) \coloneqq \{v \in N_\gamma(A) \mid \sigma(v) = \gamma\}$
    \end{itemize}
\end{definition}

As above, when $A$ is a singleton, we drop $\{\cdot\}$ for ease of notation.

\begin{definition}[Alternating Path]
    A path $P = v_1 v_2 v_3 \cdots v_l$ is called alternating if for all $2 \leq i \leq l - 1$, $c(v_{i-1}v_i) \neq c(v_iv_{i+1})$. 
\end{definition}

We are now ready to define the following.

\begin{definition} \label{def: alt bicycle}    
    If a bicycle (see Definition~\ref{def: bicycle}) can be decomposed into an alternating path $P$ of the form $v_1 v_2 v_3 \cdots v_s$ and two edges $v_1v_i$ and $v_sv_j$ with $c(v_1 v_i) \neq c(v_1v_2)$ and $c(v_sv_j) \neq c(v_{s-1}v_s)$ then we call the bicycle alternating under $c$.
\end{definition}

\begin{figure}[h]
    \centering
        \begin{tikzpicture}[node distance={10mm}, main/.style = {draw, circle, fill=black}] 
            \begin{scope}[every node/.style={circle, thin, draw, minimum size=1mm}]
                \node[main] (v1) at (0,0) {};
                \node[main, draw=none, fill=none] (v1_name) at (0,-0.4) {$v_1$}; 

                \node[main] (v2) at (1.5,0) {};
                \node[main, draw=none, fill=none] (v2') at (2.25,0) {$\cdots$};
                \node[main, draw=none, fill=none] (C1) at (2.25,0.6) {$C_1$};
                \node[main] (v3) at (3,0) {};
                \node[main] (v4) at (4.5,0) {};
                \node[main, draw=none, fill=none] (vi_name) at (4.5,-0.4) {$v_i$}; 
                \node[main] (v5) at (6,0) {};
                \node[main, draw=none, fill=none] (v5') at (6.75,0) {$\cdots$};
                \node[main, draw=none, fill=none] (P) at (6.75,0.4) {$P$};
                \node[main] (v6) at (7.5,0) {};
                \node[main] (v7) at (9,0) {};
                \node[main, draw=none, fill=none] (vi_name) at (9,-0.4) {$v_j$}; 
                \node[main] (v8) at (10.5,0) {};
                \node[main, draw=none, fill=none] (v8') at (11.25,0) {$\cdots$};
                \node[main, draw=none, fill=none] (C2) at (11.25,0.6) {$C_2$};
                \node[main] (v9) at (12,0) {};
                \node[main] (v10) at (13.5,0) {};
                \node[main, draw=none, fill=none] (vs_name) at (13.5,-0.4) {$v_s$}; 

                \draw[red] (v1) -- (v2);
                \draw[blue, very thick] (v2) -- (v2');
                \draw[red] (v2') -- (v3);
                \draw[blue, very thick] (v3) -- (v4);

                \draw[red] (v4) -- (v5);
                \draw[blue, very thick] (v5) -- (v5');
                \draw[red] (v5') -- (v6);
                \draw[blue, very thick] (v6) -- (v7);
                
                \draw[red] (v8) -- (v7);
                \draw[blue, very thick] (v8') -- (v8);
                \draw[red] (v9) -- (v8');
                \draw[blue, very thick] (v10) -- (v9);

                \draw[blue, very thick] (v1) to[out=60,in=120] (v4);
                \draw[red] (v7) to[out=60,in=120] (v10);
                
            \end{scope}
    
        \end{tikzpicture}
        \caption{An example of a proper alternating odd bicycle with $s \equiv 1 \mod{2}$.}
        \label{fig: alternating odd proper bicycle}
    \end{figure}
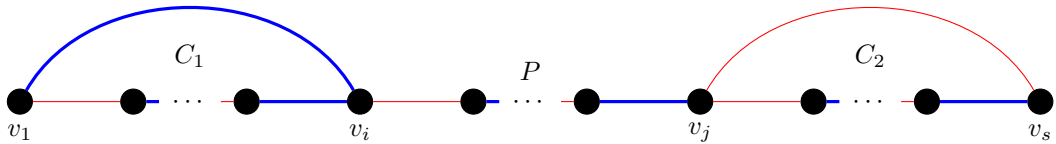
    
    \begin{figure}[h]
    \centering
        \begin{tikzpicture}[node distance={10mm}, main/.style = {draw, circle, fill=black}] 
            \begin{scope}[every node/.style={circle, thin, draw, minimum size=1mm}]
                \node[main, pattern=north east lines, pattern color=blue] (v1) at (0,0) {};
                \node[main, draw=none, fill=none] (v1_name) at (0,-0.4) {$v_1$}; 

                \node[main, fill = red] (v2) at (1.5,0) {};
                \node[main, draw=none, fill=none] (v2') at (2.25,0) {$\cdots$};
                \node[main, draw=none, fill=none] (C1) at (4.5,1) {$C_1$};
                \node[main, fill = red] (v3) at (3,0) {};
                \node[main, pattern=north east lines, pattern color=blue] (v4) at (4.5,0) {};
                \node[main, draw=none, fill=none] (vi_name) at (4.5,-0.4) {$v_i$}; 
                \node[main, fill = red] (v5) at (6,0) {};
                \node[main, draw=none, fill=none] (v5') at (6.75,0) {$\cdots$};
                \node[main, draw=none, fill=none] (P) at (6.75,0.4) {$P$};
                \node[main, pattern=north east lines, pattern color=blue] (v6) at (7.5,0) {};
                \node[main, fill=red] (v7) at (9,0) {};
                \node[main, draw=none, fill=none] (vi_name) at (9,-0.4) {$v_j$}; 
                \node[main, pattern=north east lines, pattern color=blue] (v8) at (10.5,0) {};
                \node[main, draw=none, fill=none] (v8') at (11.25,0) {$\cdots$};
                \node[main, draw=none, fill=none] (C2) at (9,-1) {$C_2$};
                \node[main, pattern=north east lines, pattern color=blue] (v9) at (12,0) {};
                \node[main, fill = red] (v10) at (13.5,0) {};
                \node[main, draw=none, fill=none] (vs_name) at (13.5,-0.4) {$v_s$}; 

                \draw[red] (v1) -- (v2);
                \draw[blue, very thick] (v2) -- (v2');
                \draw[red] (v2') -- (v3);
                \draw[blue, very thick] (v3) -- (v4);

                \draw[red] (v4) -- (v5);
                \draw[blue, very thick] (v5) -- (v5');
                \draw[red] (v5') -- (v6);
                \draw[blue, very thick] (v6) -- (v7);
                
                \draw[red] (v8) -- (v7);
                \draw[blue, very thick] (v8') -- (v8);
                \draw[red] (v9) -- (v8');
                \draw[blue, very thick] (v10) -- (v9);

                \draw[blue, very thick] (v1) to[out=40,in=140] (v7);
                \draw[red] (v4) to[out=-40,in=-140] (v10);
                
            \end{scope}
    
        \end{tikzpicture}
        \caption{An example of a improper alternating odd bicycle along with a 
        valid vertex coloring adapted to the edge coloring. We start by coloring $v_1$ blue, which forces
        the color of all vertices but $v_{i+1}, \cdots, v_{j-1}$. The vertex coloring on the 
        path $P$ from $v_i$ to $v_j$ is not alternating as $v_i$ is colored $blue$ and $v_j$ is colored $red$,
        we have some freedom on how we can color the vertices $v_{i+1}, \dots, v_{j-1}$. }
        \label{fig: alternating non-proper bicycle}
    \end{figure}
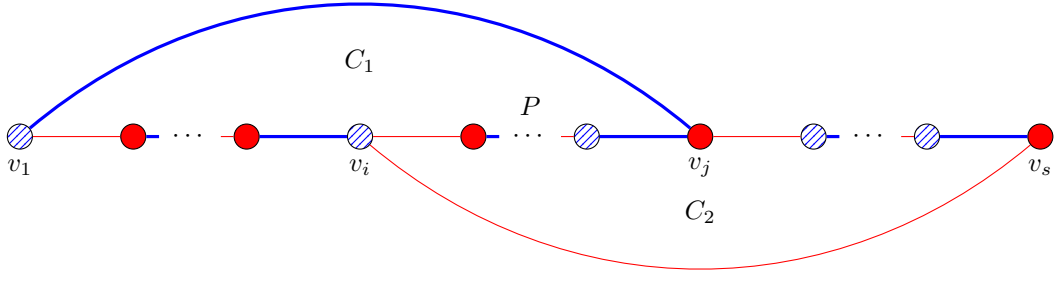

We further note that it is not hard to see that a sufficient and necessary condition for an alternating proper 
bicycle to be odd is that for $i,j$ as in Definition~\ref{def: alt bicycle} we have $i \equiv 1 \mod{2}$ and 
$j \equiv s \mod{2}$. We are now ready for the following lemma, which provides a sufficient condition for 
$(G,c)$ being not adaptably $2$-colorable.

\begin{lemma} \label{lma: non-adapt 2-col condition}
	\mbox{}
	\begin{enumerate}[label=\alph*.]
    	\item Every proper alternating odd bicycle is not adaptably 2-colorable. 
		\item Every improper alternating bicycle is adaptably 2-colorable.
	\end{enumerate}
\end{lemma}

We defer the proof of Lemma~\ref{lma: non-adapt 2-col condition} to 
Appendix~\ref{appendix: adapt 2-col proofs}; however, we note that this is shown in \cite{adapt-cols} 
in the context of adversarial adaptable colorings.

We now provide a specific family $\mathcal{H}$ of graph edge coloring pairs whose exclusion as a subgraph guarantees adaptable $2$-colorability. First, we introduce the notion of an alternating unicycle which can be viewed as ``half" of an alternating bicycle. 

\begin{definition} \label{def: alt unicycle}
    An alternating unicycle consists of an alternating path $P = v_1 \cdots v_s$ along with an edge 
    $v_jv_s$ with $c(v_jv_s) \neq c(v_{s-1}v_s)$ for some $j \in \{1, \dots, s-2\}$. Furthermore, 
    if $j \equiv s \mod{2}$ (that is the cycle $C = v_j v_{j+1} \cdots v_s v_j$ is odd) then we call the 
    alternating unicycle odd. We call the path $v_1 \cdots v_j$ the handle of the alternating unicycle.
\end{definition}

Note that by Lemma~\ref{lma: non-adapt 2-col condition}(b), the converse of the 
following Lemma does not hold. However, the lemma is sufficient for proving 
Theorems~\ref{thm: a sharp threshold} and \ref{thm: scaling window lower bound}.

\begin{lemma} \label{lma: improper alternating bicycle}
    If $G$ is connected and $(G,c)$ is not adaptably 2-colorable, then $(G,c)$ contains an alternating bicycle. 
\end{lemma}

\begin{proof}
    Given a candidate vertex coloring $\sigma : V \xrightarrow[]{} \{red, blue\}$, call an edge $e = uv$ violated under $\sigma$ if $\sigma(u) = \sigma(v) = c(uv)$. Let $\sigma : V \xrightarrow[]{} \{red, blue\}$ be a candidate vertex coloring with a minimum number of violated edges. So as $G$ is not adaptably 2-colorable, $\sigma$ must contain at least one violated edge, let $e = uv$ be such an edge. Without loss of generality, assume $c(e) = \sigma(u) = \sigma(v) = red$.
    
    We consider a two-phase process, which uses the minimality of $\sigma$ to search for an alternating bicycle. The two phases are near identical but with a slight alteration in the second phase. Each phase can be viewed as performing a Breadth First Search (BFS) using the minimality of $\sigma$ to ``explore" the graph around $uv$. The difference in the two phases will be determined by the ``starting vertex" of the process as well as the termination conditions. 
    
    We start by describing the first phase. The purpose of the first phase is to find an alternating odd unicycle. This will serve as ``half" of our alternating bicycle. That is we find an alternating path, with an extra edge of the appropriate color forming a cycle. 
 
    \textbf{Phase 1:} Start at $u$ and swap the color of $u$ from $red$ to $blue$. Call this new vertex coloring $\sigma_1$. Minimality of $\sigma$ implies that doing so must introduce at least one violated edge under $\sigma_1$, as $e = uv$ is no longer violated. So consider all such violated edges, that is the edges incident to $u$ which are colored $blue$ and whose other endpoint is colored blue under $\sigma_1$; this set of vertices is precisely $N_{blue}^{\sigma_1}(u)$ which we call $S_1$. Next, swap the color of every vertex in $S_1$ from $blue$ to $red$; call this new vertex coloring $\sigma_2$. Doing so must introduce violations, again by minimality; namely, edges incident to a vertex in $S_1$ which are colored $red$ and whose other endpoint is also colored $red$ under $\sigma_2$. These are precisely the vertices in $\cup_{x \in S_1} N_{red}^{\sigma_2}(x)$ which we call $S_2$. 
    
    We get a recursive definition for $S_i$. Given a ``layer" $S_i$ along with a vertex coloring $\sigma_i$, which is constant over $S_i$, 
    we construct $S_{i+1}$ and $\sigma_{i+1}$ as follows: we first swap the color of every vertex in $S_i$ to get the vertex coloring $\sigma_{i+1}$. We then consider the violations induced by the swap. As $\sigma_i$ is constant across $S_i$, so is $\sigma_{i+1}$. Let $\gamma$ denote the color that $\sigma_{i+1}$ assigns each vertex in $S_i$. Thus, the violations induced by the swap are the edges incident to a vertex in $S_i$ of color $\gamma$ whose other endpoint is also colored $\gamma$ under $\sigma_{i+1}$. This vertex set is precisely $S_{i+1} \coloneqq \cup_{x \in S_i} N_{\gamma}^{\sigma_{i + 1}}(x)$.

    It's worth noting that each vertex in each layer $S_i$ is reachable from $u$ by an alternating path beginning with $blue$. We will explicitly construct such paths after providing the formal definition of the process.
    Formally, we have the following process:
    
    \begin{process} \label{proc: bicycle search phase 1}
        \mbox{}
        \begin{enumerate}
            \item Initialize $S_0 \coloneqq \{u\}$ and $\sigma_0 = \sigma$.
            \item For $i \geq 1$: 
                \begin{enumerate}[label=\alph*., ref=\theenumi\alph*]
                    \item $\sigma_i = \sigma_{i-1}$ everywhere except for on $S_{i-1}$. For all $x \in S_{i-1}$, $\sigma_i(x) \neq \sigma_{i-1}(x)$, that is swap the color assigned to each vertex in $S_{i-1}$.
                    \item $S_i \coloneqq \cup_{x \in S_{i-1}} N_{\sigma_{i}(x)}^{\sigma_{i}}(x)$, that is $S_i$ denotes the endpoints of the violated edges induced by swapping the colors of $S_{i-1}$.
                \end{enumerate}
        \end{enumerate}   
    \end{process}
    
    Furthermore, for each vertex $x$ reached by the process, we define an alternating $u,x-$path 
    $P_x$ produced by the process at the first iteration in which we visit $x$. 
    This can be done recursively. That is, let $P_u = u$, the trivial path. 
    Now given $S_i$ and a $u,v$-path $P_v$ for all $v \in S_i$ we wish to construct $P_x$ 
    for all $x \in S_{i+1}$ which have not previously been revealed by the process.
    By construction for each vertex $x \in S_{i + 1}$ there exists a $v \in S_i$ 
    such that $x \in N_{\sigma_{i}(v)}^{\sigma_i}(v)$,
    that is $x$ is incident to $v$ and $c(vx) = \sigma_i(v)$.
    Note the choice of $v \in S_i$ is not necessarily unique; in the case when multiple 
    choices exists choose one arbitrarily. Define the $u,x$-path $P_x$ to be $P_v x$, 
    the path $P_v$ with $x$ added to the end, which by the structure of the process ensures 
    $P_x$ is alternating. 
    
    Constructing $P_x$ this way ensures that for $i$ minimal such that $x \in S_i$, $P_x$ is \\
    $u = x_0 x_1 x_2 \cdots x_{i-1}x_i = x$ where $x_j \in S_j$ for each $j$ and $x_j \in N_{\sigma_j(x_{j-1})}^{\sigma_j}(x_{j-1})$.

    At iteration $i \geq 1$, we say we \textit{revisit} a node $z \in S_i$ if $z \in S_{j}$ 
    for some $j < i$, that is $z$ was already ``revealed" by the process, see \ref{revisit cond 1}.
    Also, if $z = v \in S_i$, we also call $z$ revisited, see \ref{revisit cond 3}. We include this case,
    as it produces a cycle with the required alternating structure.
    \begin{enumerate}[label=Revisit Condition \arabic*), wide=0pt, align=left]
        \item \label{revisit cond 1} For $z \in S_i$ if $z \in \bigcup_{j=0}^{i-1} S_j $ we say $z$ is revisited.
        \item \label{revisit cond 3} If $v \in S_i$ we say $v$ is revisited.
    \end{enumerate}

    Let $t$ be the first iteration in which we revisit a vertex. 
    Let $y \in S_t$ be such a revisited vertex. 
    Therefore, there exists a $x \in S_{t-1}$ such that $y \in N_{\sigma_t(x)}^{\sigma_t}(x)$. 
    Let $f_1 = xy$. Different possibilities correspond to the dashed edges in Figure~\ref{fig: 2-adapt-col process pic}. 
    Note that $x$ may also be a revisited vertex, in such a case $x$ would be in $S_t$ as well, 
    as in the dashed edge between two vertices in $S_3$ in Figure~\ref{fig: 2-adapt-col process pic}.

    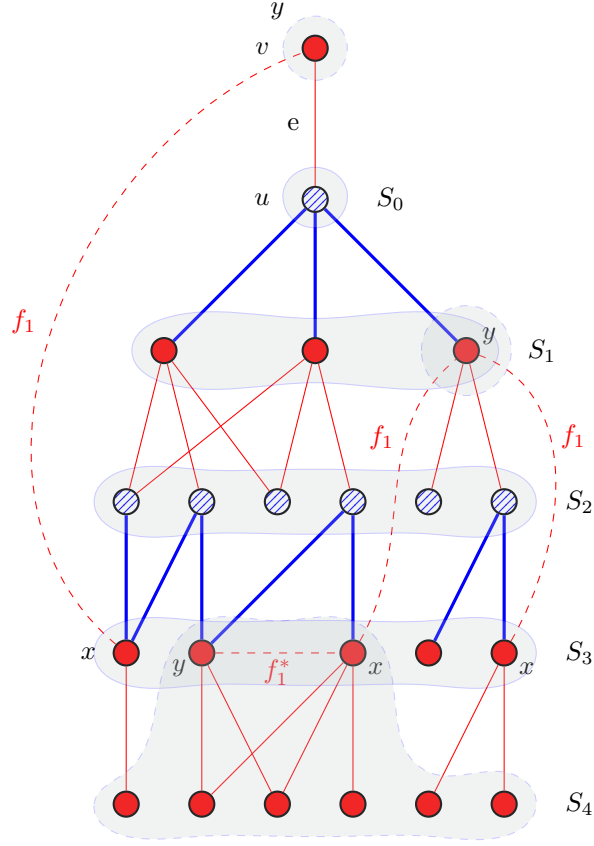
\begin{figure}[h]
    \centering
        \begin{tikzpicture}[node distance={15mm}, main/.style = {draw, circle}]
            \begin{scope}[every node/.style={circle, thick, draw, minimum size=3mm}]
                \node[main, pattern=north east lines, pattern color=blue] (u) at (0,0) {}; 
                \node[draw = none] at (u) [name=fake_u,outer sep=5pt,inner sep=5pt]{};
                \node[draw = none] at (1,0) {$S_0$};
                \node[main, draw=none, fill=none] (u_name) at (-0.7,0) {$u$}; 
                \node[main, fill=red] (v) at (0,2) {};
                \node[draw = none] at (v) [name=fake_v,outer sep=5pt,inner sep=5pt]{};
                \node[main, draw=none, fill=none] (v_name) at (-0.7,2) {$v$}; 
                 \node[main, draw=none, fill=none] (y_1) at (-0.5,2.5) {$y$}; 

                \node[main, fill=red] (s1_2) at (0,-2) {}; 
                \node[draw = none] at (s1_2) [name=fake_s1_2,outer sep=5pt,inner sep=5pt]{};
                \node[main, fill=red] (s1_1) at (-2,-2) {}; 
                \node[draw = none] at (s1_1) [name=fake_s1_1,outer sep=5pt,inner sep=5pt]{};
                \node[main, fill=red] (s1_3) at (2,-2) {}; 
                \node[main, draw=none, fill=none] (u_name) at (2.3,-1.8) {$y$}; 
                \node[draw = none] at (s1_3) [name=fake_s1_3,outer sep=5pt,inner sep=5pt]{};
                \node[draw = none] at (s1_3) [name=fake_s1_3_2,outer sep=7pt,inner sep=7pt]{};
                \node[draw = none] at (3, -2) {$S_1$};
                
                \node[main, pattern=north east lines, pattern color=blue] (s2_1) at (-2.5,-4) {}; 
                \node[draw = none] at (s2_1) [name=fake_s2_1,outer sep=5pt,inner sep=5pt]{};
                \node[main, pattern=north east lines, pattern color=blue] (s2_2) at (-1.5,-4) {}; 
                \node[draw = none] at (s2_2) [name=fake_s2_2,outer sep=5pt,inner sep=5pt]{};
                \node[main, pattern=north east lines, pattern color=blue] (s2_3) at (-0.5,-4) {};
                \node[draw = none] at (s2_3) [name=fake_s2_3,outer sep=5pt,inner sep=5pt]{};
                \node[main, pattern=north east lines, pattern color=blue] (s2_4) at (0.5,-4) {};   
                \node[draw = none] at (s2_4) [name=fake_s2_4,outer sep=5pt,inner sep=5pt]{};
                \node[main, pattern=north east lines, pattern color=blue] (s2_5) at (1.5,-4) {};    
                \node[draw = none] at (s2_5) [name=fake_s2_5,outer sep=5pt,inner sep=5pt]{};
                \node[main, pattern=north east lines, pattern color=blue] (s2_6) at (2.5,-4) {};    
                \node[draw = none] at (s2_6) [name=fake_s2_6,outer sep=5pt,inner sep=5pt]{};
                \node[draw = none] at (3.5, -4) {$S_2$};

                \node[main, fill=red] (s3_1) at (-2.5,-6) {}; 
                \node[main, draw=none, fill=none] (x_1) at (-3,-6) {$x$}; 
                \node[draw = none] at (s3_1) [name=fake_s3_1,outer sep=5pt,inner sep=5pt]{};
                \node[main, fill=red] (s3_2) at (-1.5,-6) {}; 
                \node[main, draw=none, fill=none] (y_2) at (-1.8,-6.2) {$y$}; 
                \node[draw = none] at (s3_2) [name=fake_s3_2,outer sep=5pt,inner sep=5pt]{};
                \node[draw = none] at (-0.5,-6) [name=fake_s3_3,outer sep=5pt,inner sep=5pt]{};
                \node[main, fill=red] (s3_4) at (0.5,-6) {};   
                 \node[main, draw=none, fill=none] (x_2) at (0.8, -6.2) {$x$}; 
                \node[draw = none] at (s3_4) [name=fake_s3_4,outer sep=5pt,inner sep=5pt]{};
                \node[main, fill=red] (s3_5) at (1.5,-6) {};    
                \node[draw = none] at (s3_5) [name=fake_s3_5,outer sep=5pt,inner sep=5pt]{};
                \node[main, fill=red] (s3_6) at (2.5,-6) {};    
                \node[main, draw=none, fill=none] (u_name) at (2.8,-6.2) {$x$}; 
                \node[draw = none] at (s3_6) [name=fake_s3_6,outer sep=5pt,inner sep=5pt]{};
                \node[draw = none] at (3.5, -6) {$S_3$};

                \node[main, fill=red] (s4_1) at (-2.5,-8) {}; 
                \node[draw = none] at (s4_1) [name=fake_s4_1,outer sep=5pt,inner sep=5pt]{};
                \node[main, fill=red] (s4_2) at (-1.5,-8) {}; 
                \node[draw = none] at (s4_2) [name=fake_s4_2,outer sep=5pt,inner sep=5pt]{};
                \node[main, fill=red] (s4_3) at (-0.5,-8) {};
                \node[draw = none] at (s4_3) [name=fake_s4_3,outer sep=5pt,inner sep=5pt]{};
                \node[main, fill=red] (s4_4) at (0.5,-8) {};   
                \node[draw = none] at (s4_4) [name=fake_s4_4,outer sep=5pt,inner sep=5pt]{};
                \node[main, fill=red] (s4_5) at (1.5,-8) {};    
                \node[draw = none] at (s4_5) [name=fake_s4_5,outer sep=5pt,inner sep=5pt]{};
                \node[main, fill=red] (s4_6) at (2.5,-8) {};    
                \node[draw = none] at (s4_6) [name=fake_s4_6,outer sep=5pt,inner sep=5pt]{};
                \node[draw = none] at (3.5, -8) {$S_4$};

                \draw[red] (u) -- node[draw = none, midway, left, pos=0.5, text=black] {e} (v);

                \draw[blue, very thick] (u) --  (s1_1);
                \draw[blue, very thick] (u) --  (s1_2);
                \draw[blue, very thick] (u) --  (s1_3);

                \draw[red] (s1_1) --  (s2_1);
                \draw[red] (s1_1) --  (s2_2);
                \draw[red] (s1_2) --  (s2_3);
                \draw[red] (s1_2) --  (s2_4);
                \draw[red] (s1_3) --  (s2_5);
                \draw[red] (s1_3) --  (s2_6);

                \draw[red] (s1_1) -- (s2_3);
                \draw[red] (s1_2) -- (s2_1);

                \draw[blue, very thick] (s2_1) -- (s3_1);
                \draw[blue, very thick] (s2_2) -- (s3_1);
                \draw[blue, very thick] (s2_2) -- (s3_2);
                \draw[blue, very thick] (s2_4) -- (s3_2);
                \draw[blue, very thick] (s2_4) -- (s3_4);
                \draw[blue, very thick] (s2_6) -- (s3_5);
                \draw[blue, very thick] (s2_6) -- (s3_6);

                \draw[red] (s3_1) -- (s4_1);
                \draw[red] (s3_2) -- (s4_2);
                \draw[red] (s3_2) -- (s4_3);
                \draw[red] (s3_4) -- (s4_2);
                \draw[red] (s3_4) -- (s4_3);
                \draw[red] (s3_4) -- (s4_4);
                \draw[red] (s3_6) -- (s4_5);
                \draw[red] (s3_6) -- (s4_6);

                \draw[red, dashed] (s3_2) -- node [draw = none, midway, below, yshift = 5] {$f_1^*$} (s3_4);
                \draw[red, dashed] (s3_1) to[out=140,in=200] node [draw = none, midway, left] {$f_1$} (v);
                \draw[red, dashed] (s3_4) to[out=60,in=220] node [draw = none, midway, left, yshift = 20, xshift = 6] {$f_1$} (s1_3);
                \draw[red, dashed] (s3_6) to[out=60,in=340] node [draw = none, midway, right, yshift = 15, xshift = -5] {$f_1$} (s1_3);

                \draw[blue,fill=my_grey,opacity=0.2](fake_u.west) 
                to[closed,curve through={
                (fake_u.north) ..
                (fake_u.east) ..
                }] (fake_u.west);

                \draw[blue,fill=my_grey,opacity=0.2](fake_s1_1.west) 
                to[closed,curve through={
                (fake_s1_1.north) ..
                (fake_s1_2.north) ..
                (fake_s1_3.north) ..
                (fake_s1_3.east) ..
                (fake_s1_3.south) ..
                (fake_s1_2.south) ..
                (fake_s1_1.south) ..
                }] (fake_s1_1.west);

                \draw[blue,fill=my_grey,opacity=0.2](fake_s2_1.west) 
                to[closed,curve through={
                (fake_s2_1.north) ..
                (fake_s2_2.north) ..
                (fake_s2_3.north) ..
                (fake_s2_4.north) ..
                (fake_s2_5.north) ..
                (fake_s2_6.north) ..
                (fake_s2_6.east) ..
                (fake_s2_6.south) ..
                (fake_s2_5.south) ..
                (fake_s2_4.south) ..
                (fake_s2_3.south) ..
                (fake_s2_2.south) ..
                (fake_s2_1.south) ..
                }] (fake_s2_1.west);

                \draw[blue,fill=my_grey,opacity=0.2](fake_s3_1.west) 
                to[closed,curve through={
                (fake_s3_1.north) ..
                (fake_s3_2.north) ..
                (fake_s3_3.north) ..
                (fake_s3_4.north) ..
                (fake_s3_5.north) ..
                (fake_s3_6.north) ..
                (fake_s3_6.east) ..
                (fake_s3_6.south) ..
                (fake_s3_5.south) ..
                (fake_s3_4.south) ..
                (fake_s3_3.south) ..
                (fake_s3_2.south) ..
                (fake_s3_1.south) ..
                }] (fake_s3_1.west);

                \draw[blue, dashed ,fill=my_grey,opacity=0.2](fake_s4_1.west) 
                to[closed,curve through={
                (fake_s4_1.north) ..
                (fake_s3_2.west) ..
                (fake_s3_2.north) ..
                (fake_s3_3.north) ..
                (fake_s3_4.north) ..
                (fake_s3_4.east) ..
                (fake_s4_5.north) ..
                (fake_s4_6.north) ..
                (fake_s4_6.east) ..
                (fake_s4_6.south) ..
                (fake_s4_5.south) ..
                (fake_s4_4.south) ..
                (fake_s4_3.south) ..
                (fake_s4_2.south) ..
                (fake_s4_1.south) ..
                }] (fake_s4_1.west);

                \draw[blue, dashed ,fill=my_grey,opacity=0.2](fake_s1_3_2.west) 
                to[closed,curve through={
                (fake_s1_3_2.north) ..
                (fake_s1_3_2.east) ..
                (fake_s1_3_2.south) ..
                }] (fake_s1_3_2.west);

                \draw[blue, dashed ,fill=my_grey,opacity=0.2](fake_v.west) 
                to[closed,curve through={
                (fake_v.north) ..
                (fake_v.east) ..
                (fake_v.south) ..
                }] (fake_v.west);
    
            \end{scope}
    
        \end{tikzpicture}
        \caption{An example where the first occurrence of a revisited vertex occurred at iteration $t = 4$. 
        The set of vertices enclosed by a dashed bubble form $S_4$. 
        Furthermore vertices labelled $y$ are \textit{revisited} during iteration $t = 4$ and 
        vertices labelled $x$ denote their respective neighbours in $S_3$. 
        These are the dashed edges which we label these edges with $f_1$ (or $f_1^*$).
        Note that the endpoints in the dashed edge labeled $f_1^*$ can be swapped, 
        that is both $x$ and $y$ in this case are \textit{revisited}.
        }
        \label{fig: 2-adapt-col process pic}
    \end{figure}
    
    It's not hard to see that $f_1$ forms an alternating odd unicycle consisting of a cycle and a (potentially empty) path from $u$ to the cycle, as seen in Figure~\ref{fig: 2-adapt-col first phase}. 
    We will prove this by considering two cases based on whether $f_1$ corresponds to \ref{revisit cond 1} or \ref{revisit cond 3}.

    \begin{case}{$y \neq v$.} 
    
    Let $a \in V(P_x \cap P_y)$ be the last vertex in which the two paths intersect. That is, for the $a,x$-path $P_{a,x} \subseteq P_x$ and the $a,y$-path $P_{a,y} \subseteq P_y$ we have $V(P_{a,x}) \cap V(P_{a,y}) = \{a\}$. Now, let $C_1$ be the cycle formed by $P_{a,x}$, $P_{a,y}$, and $f_1 = xy$ and let $P_1$ be the $u,a$-path $P_1 \subseteq P_x$. Thus, the two edges in $C_1$ incident to $a$ have the same color, and the remaining edges in $C_1$ are alternating as seen in Figure~\ref{fig: 2-adapt-col first phase}. Thus we have an alternating unicycle, as seen in Figure~\ref{fig: 2-adapt-col first phase}.

    \end{case}

    \begin{case}{$y = v$.} 
        
        Let $C_1$ be the cycle formed by $P_x$, $e = uv$, and $f_1 = xv$ and let $P_1 = \{u\}$ for completeness. So, let $a = v$ and hence the edges $f_1, e \in \delta(a)\cap C_1$ are both colored red and the remaining edges in $C_1$ are alternating. Thus we have an alternating unicycle as seen in Figure~\ref{fig: 2-adapt-col first phase} where $a = v$.
    
    \end{case}

    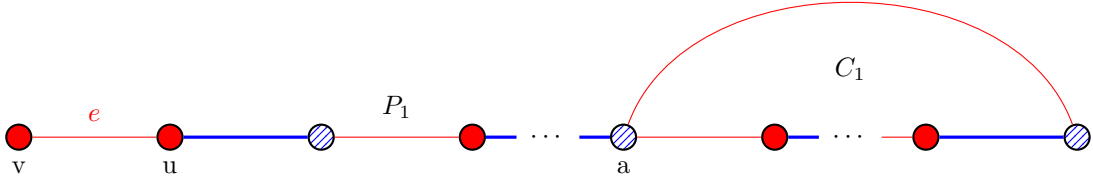
\begin{figure}[h]
    \centering
        \begin{tikzpicture}[node distance={10mm}, main/.style = {draw, circle}] 
            \begin{scope}[every node/.style={circle, thick, draw, minimum size=3mm}]
                \node[main, fill=red] (u) at (0,0) {};
                \node[main, draw=none] (u_0) at (0,-0.4) {u}; 
                
                \node[main, fill=red] (v) at (-2,0) {};
                \node[main, draw=none] (v_0) at (-2,-0.4) {v}; 

                \node[main, pattern=north east lines, pattern color=blue] (x1) at (2,0) {}; 
                \node[main, draw = none] (P) at (3, 0.4) {$P_1$};
                \node[main, fill=red] (x2) at (4,0) {}; 
                \node[main, draw = none] (x23) at (5,0) {$\cdots$};
                \node[main, pattern=north east lines, pattern color=blue] (x3) at (6,0) {}; 
                \node[main, draw = none] (x3') at (6,-0.4) {a};
                \node[main, fill=red] (x4) at (8,0) {}; 
                \node[main, draw = none] (x45) at (9,0) {$\cdots$};
                \node[main, draw = none] (C) at (9,0.9) {$C_1$};
                \node[main, fill=red] (x5) at (10,0) {}; 
                \node[main, pattern=north east lines, pattern color=blue] (x6) at (12,0) {}; 
                \node[main, draw = none] (x6') at (12,-0.4) {};

                \draw[red] (v) -- node [draw = none, midway, above] {$e$} (u);
                \draw[blue, very thick] (u) --  (x1);
                \draw[red] (x1) --  (x2);
                \draw[blue, very thick] (x2) --  (x23);
                \draw[blue, very thick] (x23) --  (x3);
                \draw[red] (x3) --  (x4);
                \draw[blue, very thick] (x4) --  (x45);
                \draw[red] (x45) --  (x5);
                \draw[blue, very thick] (x5) --  (x6);

                \draw[red] (x6) to[out=110,in=70] node [draw = none, midway, above] {} (x3);
            \end{scope}
    
        \end{tikzpicture}
        \caption{The path $P_1$ and cycle $C_1$ produced by the first run of the process, along with the original vertex coloring $\sigma$. Note that $P_1$ may be empty, and $e = uv$ may be contained in $C_1$, that is either $a = u$ or $a = v$.}
        \label{fig: 2-adapt-col first phase}
    \end{figure}

    \textbf{Phase 2:} We now consider restarting the process, this time starting at $v$ instead of $u$; 
    that is, $S_0 = \{v\}$ and \ref{revisit cond 3} changes accordingly. However, this time we 
    terminate the process when we either:
    \begin{enumerate}[label=Termination Condition \arabic*), wide=0pt, align=left]
        \item \label{phase 2 term cond 2} Visit a vertex in $P_1 \cup C_1 \setminus \{v\}$, or
        \item \label{phase 2 term cond 1} Revisit a vertex (in this new process)
    \end{enumerate}

    Let $z$ be a vertex which caused termination, that is either $z$ is revisited, 
    or $z \in P_1 \cup C_1 \setminus \{v\}$. Therefore, there exists $f_2 = wz$ with $w$ being in the 
    previous layer. Note that the choice of $z$ may not be unique: that is there may be multiple causes for termination.
    
    It's worth noting that the events \ref{phase 2 term cond 1} and \ref{phase 2 term cond 2} are not disjoint:   
    if $z = u$, then $z \in P_1 \cup C_1$ and $z$ is a revisited vertex. 
    We first consider \ref{phase 2 term cond 2},
    then we consider \ref{phase 2 term cond 1} assuming \ref{phase 2 term cond 2} does not 
    hold.

    \begin{case}{\ref{phase 2 term cond 2} was a cause of termination.} 
        Therefore assume, $z \in P_1 \cup C_1$ and $w \notin P_1 \cup C_1$. 
        Let $P_w$ be an alternating path from $v$ to $w$ as constructed by the process. Recall the definition of 
        $a \in C_1$ from the first phase and let $f_a = aa_0$ be an edge of $C_1$. 
        So, $P = P_1 \cup C_1 \cup P_w \setminus \{f_a\}$ is an alternating path from $a_0$ to $w$ 
        beginning with the color which is not $c(f_a)$ and ending with the color which is not $c(f_2)$. Thus, by Definition~\ref{def: alt bicycle}, $P \cup \{f_a, f_2\}$ forms an improper alternating bicycle as required. 
    \end{case}

    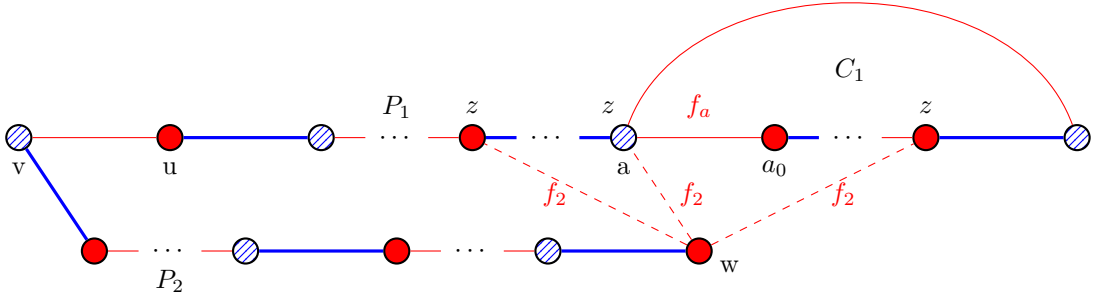
\begin{figure}[h]
    \centering
        \begin{tikzpicture}[node distance={10mm}, main/.style = {draw, circle}] 
            \begin{scope}[every node/.style={circle, thick, draw, minimum size=3mm}]
                \node[main, fill=red] (u) at (0,0) {};
                \node[main, draw=none] (u_0) at (0,-0.4) {u}; 
                
                \node[main, pattern=north east lines, pattern color=blue] (v) at (-2,0) {};
                \node[main, draw=none] (v_0) at (-2,-0.4) {v}; 

                \node[main, pattern=north east lines, pattern color=blue] (x1) at (2,0) {}; 
                \node[main, draw = none] (P) at (3, 0.4) {$P_1$};
                \node[main, draw = none] (x12) at (3,0) {$\cdots$};
                \node[main, fill=red] (x2) at (4,0) {}; 
                \node[main, draw = none] (x2') at (4,0.4) {$z$};
                \node[main, draw = none] (x23) at (5,0) {$\cdots$};
                \node[main, pattern=north east lines, pattern color=blue] (x3) at (6,0) {}; 
                \node[main, draw = none] (x3'') at (5.8,0.4) {$z$};
                \node[main, draw = none] (x3') at (6,-0.4) {a};
                \node[main, fill=red] (x4) at (8,0) {}; 
                \node[main, draw = none] (x4') at (8,-0.4) {$a_0$};
                \node[main, draw = none] (x45) at (9,0) {$\cdots$};
                \node[main, draw = none] (C) at (9,0.9) {$C_1$};
                \node[main, fill=red] (x5) at (10,0) {}; 
                \node[main, draw = none] (x5') at (10,0.4) {$z$};
                \node[main, pattern=north east lines, pattern color=blue] (x6) at (12,0) {}; 
                \node[main, draw = none] (x6') at (12,-0.4) {};

                \node[main, fill = red] (y1) at (-1,-1.5) {};
                \node[main, draw = none] (y12) at (0,-1.5) {$\cdots$};
                \node[main, draw = none] (P2) at (0,-1.9) {$P_2$};
                \node[main, pattern=north east lines, pattern color=blue] (y2) at (1,-1.5) {};
                \node[main, fill = red] (y3) at (3,-1.5) {};
                \node[main, draw = none] (z) at (3,-1.1) {};
                \node[main, draw = none] (y34) at (4,-1.5) {$\cdots$};
                \node[main, pattern=north east lines, pattern color=blue] (y4) at (5,-1.5) {};
                \node[main, fill = red] (y5) at (7,-1.5) {};
                \node[main, draw = none] (w) at (7.4,-1.7) {w};

                \draw[red] (v) --  (u);
                \draw[blue, very thick] (u) --  (x1);
                \draw[red] (x1) --  (x12);
                \draw[red] (x12) --  (x2);
                \draw[blue, very thick] (x2) --  (x23);
                \draw[blue, very thick] (x23) --  (x3);
                \draw[red] (x3) --  node [draw = none, midway, above] {$f_a$} (x4);
                \draw[blue, very thick] (x4) --  (x45);
                \draw[red] (x45) --  (x5);
                \draw[blue, very thick] (x5) --  (x6);

                \draw[blue, very thick] (v) -- (y1);
                \draw[red] (y1) -- (y12);
                \draw[red] (y12) -- (y2);
                \draw[blue, very thick] (y2) -- (y3);
                \draw[red] (y3) -- (y34);
                \draw[red] (y34) -- (y4);
                \draw[blue, very thick] (y4) -- (y5);

                \draw[red, dashed] (y5) -- node [draw = none, midway, left] {$f_2$} (x2);

                \draw[red, dashed] (y5) -- node [draw = none, midway, right] {$f_2$} (x5);
                
                \draw[red, dashed] (y5) -- node [draw = none, midway, right] {$f_2$} (x3);

                \draw[red] (x6) to[out=110,in=70] node [draw = none, midway, above] {} (x3);
            \end{scope}
    
        \end{tikzpicture}
        \caption{The process of Phase 2, assuming \ref{phase 2 term cond 2} was the cause of termination. Particularly, each dashed edge labeled $f_2$ corresponds to a possible cause of termination. The vertex coloring shown denotes that at the moment of termination.}
        \label{fig: 2-adapt-col second phase}
    \end{figure}
    
    We now consider the case where \ref{phase 2 term cond 1} was the cause of termination and \ref{phase 2 term cond 2} did not occur.
    
    \begin{case}{\ref{phase 2 term cond 2} was not a cause of termination.}
    	Therefore $z$ must satisfy \ref{phase 2 term cond 1}. Following the same proof as in Phase 1, we get an alternating path $P_2$ along with a cycle $C_2$ as in Figure~\ref{fig: 2-adapt-col first phase}, but starting from $v$. Moreover, as \ref{phase 2 term cond 2} did not occur, 
        it follows that $P_2 \cup C_2$ is edge disjoint from $P_1 \cup C_1$; infact, they are vertex-disjoint except potentially at $v$. 
        Therefore, upon joining them with the edge $uv$, they form an alternating bicycle. 
    \end{case}
    
\end{proof}

Recall, the condition for adaptable 2-colorability in Lemma~\ref{lma: improper alternating bicycle} is not sufficient, that is the converse does not hold. In particular, all improper alternating bicycles are adaptably $2$-colorable.

We conjecture the following sufficient and necessary condition similar to Theorem~\ref{thm: adversary adpt 2 col} in the adversarial model.

\begin{conjecture} \label{conjecture: iff cond}
    There exists a unique edge coloring (up to relabeling) of an odd edge-$K_4$ such that $(G,c)$ is adaptably 2-colorable if and only if $(G,c)$ does not contain a proper alternating odd bicycle or an odd edge-$K_4$ with the specified edge coloring.
\end{conjecture}

Consider the case when Phase 2 terminates due to \ref{phase 2 term cond 2} and we 
achieve an improper alternating bicycle. 
Thus, by Lemma~\ref{lma: non-adapt 2-col condition} (b) we have that the improper 
bicycle is adaptably $2$-colorable, let $\phi$ be such a vertex coloring. 
Consider the original vertex coloring $\sigma$
and swap the coloring on the improper bicycle to $\phi$. By minimality of $\sigma$ this
must introduce some new violations. We believe that continuing this process, will produce 
the required subgraph as in Conjecture~\ref{conjecture: iff cond}.

We also note the following consequence of Lemma~\ref{lma: improper alternating bicycle} regarding \textit{frozen} vertices. 
In $(G,c)$ we call a vertex $u \in V$ \textit{frozen} if $\sigma(u)$ is the same color in every adaptable $2$-coloring $\sigma$ of $(G,c)$.

We get the following corollary,

\begin{corollary} \label{cor: frozen char}
    Given a graph edge coloring pair $(G,c)$, a vertex $u$ is frozen if and only if 
    it is a vertex in the handle of an alternating odd 
    unicycle as in Definition~\ref{def: alt unicycle}.
\end{corollary}

\begin{proof}{(sketch)}

If a vertex $u$ is frozen, say $\sigma(u) = red$ for all adaptable $2$-colorings $\sigma$ of $(G,c)$, then swapping the 
color of $u$ from $red$ to $blue$ must introduce some violations. This allows us to 
follow the same arguments from Phase 1 in Lemma~\ref{lma: improper alternating bicycle} to show that
there exists an alternating odd unicycle.

If a vertex is in an such a path $P_{i,j}$, then it is easy to see that vertex is forced to be a certain color. The argument is in the proof of Lemma~\ref{lma: non-adapt 2-col condition}. 
\end{proof}

\section{Probability Preliminaries}

\subsection{Notation}

We make note on some notation which we adopt. First, when dealing with asymptotic equivalence, 
we say two functions $f,g: \mathbb{R} \xrightarrow[]{} \mathbb{R}$ are 
\textit{asymptotically equivalent}, denoted $f \sim g$ if 
$\lim_{n \xrightarrow[]{} \infty} \frac{f(n)}{g(n)} = 1$. Secondly, given a probability distribution $\mathcal{D}$, $d \sim \mathcal{D}$ 
denotes that $d$ is sampled from $\mathcal{D}$. Differentiating between these two uses will be clear from context.

\subsection{Inequalities and Concentration Bounds}

Before proceeding, let us introduce some concentration bounds and Inequalities which we shall 
use throughout the paper, all of which are standard within the literature. We begin with Markov's Inequality.

\begin{inequality}[Markov]{(Theorem 2.17 in \cite{Connected_component_rand_graphs})} \label{Markov's Inequality}
    Let $X$ be a non-negative random variable with finite expectation. Then,
    \begin{equation*}
        \mathbb{P}[X \geq a] \leq \frac{\mathbb{E}[X]}{a} \,.
    \end{equation*}
\end{inequality}

As standard when dealing with random graphs we will also require 
concentration bounds on the binomial distribution. 
We denote the binomial distribution with $n$ trials and
success probability $p$ by $\mathcal{B}in(n,p)$. We will make use of the following Binomial Chernoff Bound.

\begin{bound}[Binomial Chernoff]{(Chapter 5 in \cite{GCProbMethod})} \label{Chernoff Bound}
    For $0 \leq t \leq np$:
    \begin{equation*}
        \mathbb{P}[\left| \mathcal{B}in(n,p) - np \right| > t] < 2e^{-t^2/3np} \,.
    \end{equation*}
\end{bound}

\subsection{Random Variables}

Throughout this paper, we shall make use of the following random variables. 
For vertices $u$ and $v$ as well as color $col \in \{red, blue\}$ define the indicator variable
$X_{uv}^{col}$ on a graph edge-coloring pair $(G,c)$ as follows: 
\begin{equation*}
    X_{uv}^{col} = \begin{cases} 
      1 & uv \in E \ \wedge \ c(uv) = col \\
      0 & otherwise 
   \end{cases} \,.
\end{equation*}
So, when considering $\mathcal{G}(n,p,2)$,
\begin{equation*}
    \mathbb{P}[X^{red}_{uv} = 1] = \mathbb{P}[X^{blue}_{uv} = 1] = p/2 \,.
\end{equation*}
However, the moment we reveal say $X^{red}_{uv}$ the distribution of $X^{blue}_{uv}$ changes. 
In particular, say we reveal $X^{red}_{uv}$ and it equates to zero; that is, we reveal that there is 
no $red$ edge between $u$ and $v$. What happens to the probability that a $blue$ edge exists? Or equivalently, 
what happens to the probability that an edge exists? We note these events are identical; however, 
we state them seperately to distinguish between the two use cases.  
This is captured in the following lemma. 

\begin{lemma} \label{lma: prob change}
    Let $c_1, c_2 \in \{red,blue\}$ with $c_1 \neq c_2$. If $X^{c_2}_{xy}$ was revealed and equals 0 then the following hold:
    \begin{itemize}
        \item the probability that $X^{c_1}_{xy} = 1$ given $X^{c_2}_{xy} = 0$ is $\frac{p/2}{1-p/2} > p/2$.
        \item the probability $xy \in E$ given $X^{c_2}_{xy} = 0$ is $\frac{p/2}{1-p/2} < p$.
    \end{itemize}
\end{lemma}

\begin{proof}
    Without loss of generality, let us assume $c_1 = red$ and $c_2 = blue$.
    Since all that has been revealed is $X^{blue}_{xy} = 0$, the only information which we gain
    is that there is not a blue edge between $x$ and $y$. So the set of possible outcomes becomes, 
    no edge exists or a red edge exists. Therefore, given $X^{c_2}_{xy} = 0$ the events, 
    $X^{c_1}_{xy} = 1$ and $xy \in E$ are precisely the same. In fact by considering the truncated probability space we get,
    \begin{equation*}
        \mathbb{P}[X^{red}_{xy} = 1 \ given \ X^{blue}_{xy} = 0] = \frac{p/2}{p/2 + (1 - p)} = \frac{p/2}{1-p/2} > p/2
    \end{equation*}
    Similarly,
    \begin{equation*}
        \mathbb{P}[xy \in E \ given \ X^{blue}_{xy} = 0] = \frac{p/2}{p/2 + (1 - p)} = \frac{p/2}{1-p/2} < p
    \end{equation*}
\end{proof}

\subsection{The Sprinkling Method} \label{sec: sprinkling method}

We will make use of the ``sprinkling" method, a standard technique in the study of random graphs. 
We consider the case where $k = 2$ for ease of notation; however, this technique is not unique to this case. 
First, consider $\mathcal{G}(n,m,2)$ as the technique is most natural in this model. 
Say we wish to sample a graph from $\mathcal{G}(n,m = m_1 + m_2,2)$. We can do this as follows: first sample
a graph $G_1$ from $\mathcal{G}(n, m_1, 2)$ then "sprinkle" $m_2$ additional edges; i.e. not in $G_1$. Note the term "sprinkle" is used 
as usually $m_2 \ll m_1$. The resulting graph is distributed precisely as $\mathcal{G}(n, m, 2)$.

We now consider an analogous technique to sample from $\mathcal{G}(n,p,2)$ where $p = \Theta(n^{-1})$ as follows:
First we sample
$(G_1, c_1) \sim \mathcal{G}(n,p_1,2)$ for some edge probability $p_1 < p$. We increase 
the probability to $p$ as follows: For every edge which is not in $G_1$, we add it with probability $p_2$, where $p_2$ is chosen so the resulting probability of an edge existing in the graph is $p$. We then color these "sprinkled" edges $red$ or $blue$ uniformly at random. Therefore, $p_2$ is the solution to $(1 - p_1)(1 - p_2) = 1 - p$ and thus, $p_2 = \frac{p-p_1}{1 - p_1} \sim p - p_1$ if $p_1 = \mathcal{O}(n^{-1})$. 
The distribution of the resulting graph is precisely the distribution of 
$\mathcal{G}(n,p,2)$.

For more on the method of ``sprinkling" we refer the reader to Chapter 11.9 in \cite{TheProbabilisticMethod}.

\section{Long Alternating Paths} \label{sec: long alternating path}

In this section we prove that for $\epsilon = \epsilon(n) > 0$ the random graph $\mathcal{G}(n,p = (2 + \epsilon)/n,2)$ contains w.h.p. an alternating path of length $\Theta(\epsilon^2n)$.  Particularly, the goal of this section is to prove the following lemma. 

\begin{lemma} \label{lma: alt path existence sub linear}
    Let $\epsilon = \epsilon(n)> 0$ be such that $\epsilon \gg n^{-1/3} \log^{1/3}n$, then w.h.p. $\mathcal{G}(n,p= 2(1 + \epsilon)/n, 2)$ contains an alternating path of length at least $\epsilon^2 n / 5$.
\end{lemma}

Krivelevich and Sudakov in \cite{longest-path-simple} prove that in the standard model $\mathcal{G}(n,p)$ there exists a path of length $\Theta(\epsilon^2 n)$ when $p = (1 + \epsilon)/n$. Particularly they prove the following,

\begin{theorem} [Theorem 1 in \cite{longest-path-simple}]
 	For $\epsilon > 0$ then w.h.p. $\mathcal{G}(n,p)$ contains a path of length at least $\frac{\epsilon^2 n}{5}$.
\end{theorem}

To accomplish this, Krivelevich and Sudakov utilize a standard Depth First Search (DFS) 
run on the random input $\mathcal{G}(n,p)$, they proceed to utilize standard Chernoff Bounds 
for Binomial random variables to show that w.h.p., DFS finds such a long path. 

To study $\mathcal{G}(n,p,2)$, we follow their approach adapted to alternating paths.  
If we have an alternating path $P$ with an end $u$ and wish to extend the 
alternating path from $u$, then naturally we can only consider edges of a certain color 
incident to $u$. The probability an edge both exists 
and gets assigned the ``correct" color is precisely $p/2 = (1+\epsilon)/n$. This is precisely the branching probability in the supercritical regime in $\mathcal{G}(n,p)$. So its natural to think their approach will carry over to the setting 
of $\mathcal{G}(n,p,2)$. 

Let us now provide the DFS algorithm, which will produce an alternating path. 
As in \cite{longest-path-simple}, let $S$ denote the set of vertices which have been fully 
explored, let $T$ denote the set of unvisited vertices. 
We will use a stack $U$ to track the vertices of our current candidate for such a path. 
Furthermore, for $u \in U$ we will let $\psi(u)$ denote the edge color required to extend the alternating path at $u$.
Therefore, when searching for potential ways to augment the path, we check $T \cap N_{\psi(u)}(u)$. 
We defer the formal definition of the process to Appendix~\ref{appendix: long alternating paths}; however, we state the following 
observations regarding the information revealed during the formulation of the alternating path.

\begin{observation} \label{ob: long path vars considered}
    For each pair of vertices $u, v \in V$ at most one of $X_{uv}^{red}$ and $X_{uv}^{blue}$ is ever revealed. 
\end{observation}

\begin{observation} \label{ob: alt path prob change}
    For two non-adjacent vertices $x$ and $y$ in $U$, with $x$ before $y$, the following hold: 
    \begin{enumerate}[label=\alph*.]
        \item If $X_{xy}^{\psi(x)}$ was revealed, then it must equate to $0$. 
        \item For $col \neq \psi(x)$, $X^{col}_{xy}$ was not revealed. 
    \end{enumerate}
\end{observation}

We defer the proof of Lemma~\ref{lma: alt path existence sub linear} and 
Observations~\ref{ob: long path vars considered} and \ref{ob: alt path prob change} to Appendix~\ref{appendix: long alternating paths}.

\section{A Sharp Threshold and Initial Bounds on the Critical Window} \label{sec: a sharp threshold}

In this section we prove Theorem~\ref{thm: a sharp threshold}. 

\begin{theorem*} [\ref{thm: a sharp threshold}]
    For any constant $\epsilon > 0$:
    \begin{enumerate}[label=\alph*.]
        \item $\mathcal{G}(n, p = \frac{2 - \epsilon}{n}, 2)$ is w.h.p. adaptably 2-colorable.
        \item $\mathcal{G}(n, p = \frac{2 + \epsilon}{n}, 2)$ is w.h.p. not adaptably 2-colorable.
    \end{enumerate}
\end{theorem*}
 
Furthermore, using the same techniques we achieve the following bounds on the critical window.

\begin{theorem*} [\ref{thm: scaling window lower bound}]
    For $\lambda_n \xrightarrow[]{} \infty$ arbitrarily slowly, $\mathcal{G}(n, p = 2(1 - \lambda_n n^{-1/3})/n, 2)$ is w.h.p. adaptably 2-colorable.
\end{theorem*}

\begin{theorem*} [\ref{thm: scaling window upper bound}]
    For $\lambda_n \xrightarrow[]{} \infty$ arbitrarily slowly, $\mathcal{G}(n, p = 2(1 + \lambda_n n^{-1/5})/n, 2)$ is w.h.p. not adaptably 2-colorable.
\end{theorem*}

We defer the proof of Theorem~\ref{thm: scaling window lower bound} and Theorem~\ref{thm: scaling window upper bound} to 
Appendix~\ref{appendix: Critical Window Bounds}. Note that both proofs follow precisely the proofs of Theorem~\ref{thm: a sharp threshold} (a) and 
Theorem~\ref{thm: a sharp threshold} (b) respectively. 
Furthermore, the technique used in Theorem~\ref{thm: a sharp threshold} (b) cannot be pushed past 
$p = 2(1 + \lambda_n n^{-1/5})/n$. That is this method cannot achieve $p = 2(1 + \lambda_n n^{-1/3})/n$, which would be symmetric to   
Theorem~\ref{thm: scaling window lower bound}.

\subsection{The Subcritical Regime}

We first consider the subcritical regime.

\begin{proof}[Proof of Theorem~\ref{thm: a sharp threshold} (a)]
    Let $X$ denote the number of alternating bicycles in $\mathcal{G}(n,p,2)$ for $p = 2(1 - \epsilon)/n$. Consider a fixed sequence of distinct vertices $S = (v_1, v_2, \dots, v_s)$ and let $X_S = 1$ if $S$ forms an alternating bicycle and zero otherwise. Thus, $X = \sum_{S} X_S$.
    Note that, given the color of any edge in an alternating bicycle, the colors of the remaining edges become forced. Thus, there are two choices for the coloring and $(s-2)^2$ choices for the edges $v_1v_i$ and $v_jv_s$. Therefore,
    \begin{align*}
        \mathbb{E}[X_S] = 2(s-2)^2 \left(\frac{p}{2}\right)^{s+1} = (s-2)^2p^{s+1}2^{-s} \,.
    \end{align*}
    So, as $\mathbb{E}[X_S]$ depends only on the number of vertices in $S$,
    \begin{equation*}
        \mathbb{E}[X] \leq \sum_{s = 1}^n n^s s^2 p^{s+1} 2^{-s} = \frac{2(1 - \epsilon)}{n} \sum_{s=1}^n s^2 \left( \frac{2(1 - \epsilon)}{2} \right)^s = \frac{2(1 - \epsilon)}{n} \sum_{s=1}^n s^2 \left( 1 - \epsilon \right)^s \,.
    \end{equation*}
    By properties of power series we have that $\sum_{s=1}^\infty s^2 \left( 1 - \epsilon \right)^s$ converges if $1 - \epsilon < 1$. Thus, as $\epsilon > 0$ we have,
    \begin{equation*}
        \mathbb{E}[X] \leq \frac{2(1 - \epsilon)}{n} \sum_{s=1}^n s^2 \left( 1 - \epsilon \right)^s = \frac{2(1 - \epsilon)}{n} \mathcal{O}(1) = \mathcal{O}\left(\frac{1}{n}\right) = o(1) \,.
    \end{equation*}
    Thus, by Markov's Inequality, for $p < 2(1 - \epsilon)/n$ w.h.p. $G(n,p,2)$ does not contain an alternating bicycle and thus by the contrapositive of Lemma~\ref{lma: improper alternating bicycle} we have that w.h.p. $G(n,p,2)$ is adaptably 2-colorable. 
\end{proof}

\subsection{The Supercritical Regime} \label{sec: supercrit regime}

The proof of of Theorem~\ref{thm: a sharp threshold} (b) relies on the method of 
``deferred decisions" as well as the ``sprinkling" method. 
We reveal the graph in a very structured manner while searching for a long 
alternating path, making sure that we keep track of what information was revealed 
as well; particularly, between vertices in the long alternating path.
This is handled by Observation~\ref{ob: alt path prob change} regarding the search procedure in 
Section~\ref{sec: long alternating path}. We then ``sprinkle" edges and show that among the ``sprinkled" 
edges, an alternating odd bicycle is formed along the alternating path. 

In $\mathcal{G}(n,p)$,
we would be able to assume there are no edges between such vertices before 
``sprinkling" the extra edges, as this only decreases the probability of success.
However, as we require edges of a specific color, we can no longer make this 
assumption. To remedy this, we will show that the number of edges between such pairs of 
vertices is small before ``sprinkling" the extra edges. To do this, we need to be careful regarding 
what has been exposed during the construction of the long alternating path. 

\begin{proof}[Proof of Theorem~\ref{thm: a sharp threshold} (b)]
    We utilize a ``sprinkling" argument. We first consider $G_1 \sim \mathcal{G}(n, p = \frac{2(1 + \epsilon/2)}{n}, 2)$. 
    By Lemma~\ref{lma: alt path existence sub linear} w.h.p. there exists an alternating path of length 
    $\alpha n$ for $\alpha = (\epsilon/2)^2 / 5$. 
    Let $P = v_1, \dots, v_{\alpha n}$ be such an alternating path. We will increase $p$ to $2(1 + \epsilon)/n$; that is,
    we sprinkle edges with probability $\sim \epsilon / n$. We will show w.h.p. these sprinkled edges form an alternating odd bicycle along $P$.
    
    To simplify calculations significantly, we only consider the odd indexed vertices of $P$. Noting that 
    asymptotically this only results in a constant factor loss. 
    We partition the set of odd indices into two sets $I_1$ and $I_2$ of roughly equal size,
    \begin{equation*}
        I_1 = \left\{i \mid i \ odd \wedge i \leq \frac{\alpha n}{2}\right\} \qquad and \qquad I_2 = \left\{i \mid i \ odd \wedge i > \frac{\alpha n}{2}\right\} \,.
    \end{equation*}
    So, $|I_1| \sim |I_2| \sim \alpha n / 4$.

    Assume without loss of generality that the first edge $v_1 v_2$ in $P$ is colored $red$. Therefore, 
    if upon increasing $p$ we introduce a $blue$ edge between vertices in $I_1$ and a $red$ edge 
    between vertices in $I_2$ then it's not hard to see an odd alternating bicycle is formed. 

    Before we consider increasing $p$, let us consider the number of edges in $G_1$ between vertices 
    in $I_1$ and $I_2$. We first note that by Observation~\ref{ob: alt path prob change} and Lemma~\ref{lma: prob change}, the information
    revealed during the construction of $P$ does not increase the probability of an edge existing between two vertices in $I_1$ or $I_2$. 
    Therefore, if $B_1$ denotes the number of existing edges between vertices in $I_1$ and $B_2$ denotes the number of existing edges between vertices in $I_2$, 
    it follows that 
    $B_1$ and $B_2$ are stochastically dominated by a binomial distribution with ${\alpha n / 4 \choose 2}$ 
    trials and 
    probability $\frac{2(1+\epsilon/2)}{n}$. Therefore for $i \in \{1,2\}$, 
    $\mathbb{E}[B_i] \leq {\alpha n / 4 \choose 2} \frac{2(1+\epsilon/2)}{n} \sim n \alpha^2 (1 + \epsilon/2)/16$.
    So by the \hyperref[Chernoff Bound]{Binomial Chernoff Bound},
    \begin{equation*}
        \mathbb{P}[B_i > \mathbb{E}[B_i] + n^{2/3}] \leq 2e^{\frac{-16n^{4/3}}{3 n \alpha^2 (1 + \epsilon/2)}} = 2e^{-\Theta(n^{1/3})} = o(1) \,.
    \end{equation*}
    Particularly, w.h.p. $B_i = \mathcal{O}(n)$. 
    Let $\widetilde{I}_1$ denote the pairs of vertices in $I_1$ for which there does not already exist an edge. 
    Similarly, let $\widetilde{I}_2$ denote the pairs of vertices in $I_2$ for which there does not already exist an edge
    Therefore, w.h.p. for $i \in \{1,2\}$:
    \begin{equation*}
        |\widetilde{I}_i| = {|I_i| \choose 2} - B = {|I_i| \choose 2} (1 - o(1)) \,.
    \end{equation*}
    We now consider sprinkling edges with probability $\sim \frac{\epsilon}{n}$. 
    Let $X_1$ denote the number of $blue$ edges between pairs of vertices in $\widetilde{I}_1$ which are introduced among the sprinkled edges. Similarly, 
    let $X_2$ denote the number of $red$ edges between pairs of vertices in $\widetilde{I}_2$ which are introduced among the sprinkled edges. 
    Therefore, $X_i$ follows a binomial distribution with $|\widetilde{I}_i|$ trials and probability 
    $\sim \frac{\epsilon}{2n}$. It suffices to show that w.h.p. $X_1 \geq 1$ and $X_2 \geq 1$,
    \begin{equation*}
        \mathbb{P}[X_i = 0] \sim \left(1 - \frac{\epsilon}{2n}\right)^{{|I_i| \choose 2} (1 - o(1))}
        \sim e^{-\frac{\epsilon}{2n}\frac{\alpha^2 n^2}{32}(1-o(1))} = \Theta(e^{-n}) = o(1) \,.
    \end{equation*}
    So, w.h.p. $X_1 \geq 1$ and $X_2 \geq 1$ and thus an alternating odd bicycle is formed.
\end{proof}

Before proceeding, we note the following corollary of 
Theorem~\ref{thm: a sharp threshold} (b) regarding the structure of 
the proper alternating odd bicycle produced. 

\begin{corollary} \label{cor: disjoint cycle bicycle}
    $\mathcal{G}(n,p = \frac{2 + \epsilon}{n}, 2)$ w.h.p. contains a proper alternating odd bicycle with vertex disjoint cycles.
\end{corollary}

\section{Improving our Supercritical bound on the Critical Window} \label{sec: supercritical improvement}

The purpose of this section is address Theorem~\ref{thm: scaling window upper bound 2}, which we restate:

\begin{theorem*}[\ref{thm: scaling window upper bound 2}]
    For $\lambda_n \xrightarrow[]{} \infty$ arbitrarily slowly, $\mathcal{G}(n, p = 2(1 + \lambda_n n^{-1/3})/n, 2)$ is w.h.p. not adaptably 2-colorable.
\end{theorem*}

That is, we consider matching the $\epsilon$ term given in Theorem~\ref{thm: scaling window lower bound}, consequently matching the size of the 
critical windows for the giant component in $\mathcal{G}(n,p)$ and for $2$-SAT.

This work follows closely the work of Bollobás et al. \cite{2-sat-critical-window} and many of the following statements are direct adaptations of their work. 
Bollobás et al. \cite{2-sat-critical-window}
use the standard method of representing instances of $2$-SAT as directed graphs (see Section 2 in \cite{2-sat-critical-window}). 
They proceed to show that satisfiability of a $2$-SAT formula is equivalent to determining whether the directed
graph representation contains a specific type of cycle, which they call a \textit{contradictory cycle} (see Lemma 2.1 in \cite{2-sat-critical-window}).

The main challenge with adaptable colorability in comparison to the directed graph representation of a $2$-SAT instance (see. \cite{2-sat-critical-window}) is the lack of independence of certain properties. For instance, consider performing a search for all vertices reachable from a given vertex $u$ by an alternating path beginning with the color $red$ and compare this to searching for all vertices reachable by a directed path in a directed graph. Say we know that there exists an alternating path from $u$ to some vertex $v$, and we wish to extend this path. However, this depends on our previous choices; namely, the color of the edge incident to $v$ in this path. Whereas, if we consider a directed path from $x$ to $y$ to extend this path, we simply look at all the edges ``leaving" $y$, which is independent of our previous choices. The main technical contribution in this section is handling this dependence.

Before proceeding we provide an outline of the proof. We begin in the subcritical regime, 
$p_1 = \frac{2(1- \lambda_n n^{-1/3})}{n}$. We will show that in the subcritical regime there exists a 
lot of ``small" subgraphs which serve as ``seeds" or ``building blocks" for alternating bicycles. 
We then increase our probability by $p_2 = M\lambda_n / n^{4/3}$, for some large constant $M$, pushing us into the supercritical 
regime. We will show that by doing so, among a large subset of the small subgraphs, a proper odd alternating bicycle is formed w.h.p.. Thus, by 
Lemma~\ref{lma: non-adapt 2-col condition}, 
$\mathcal{G}(n,p = \frac{2(1 + \lambda_n n^{-1/3})}{n},2)$ is not adaptably $2$-colorable.

In the following subsection, we consider a search process in the subcritical regime, 
i.e. $p = \frac{2(1 - \lambda_n n^{-1/3})}{n}$, which will be a key component in constructing the ``building blocks".

\subsection{A Search Process}

In this subsection we shall assume we are in the subcritical regime, $p = \frac{2(1 - \lambda_n n^{-1/3})}{n}$.

Given a vertex $u$ and a color $\gamma \in \{red, blue\}$ we consider a Breath First Search 
process that explores $\mathcal{G}(n, p = \frac{2(1 - \lambda_n n^{-1/3})}{n}, 2)$ for vertices 
which are reachable from $u$ by an alternating path beginning with the color $\gamma$. 
First, recall the following random variable for each potential edge, i.e. each unordered pair of vertices $x, y$ and color $\kappa \in \{red, blue\}$, 
\[ X_{xy}^\kappa = \begin{cases} 
      1 & xy \in E \ and \  c(xy) = \kappa \\
      0 & otherwise
    \end{cases} \,.
\]
When revealing the random variables $X_{xy}^\kappa$ we are simply revealing whether there exists an edge of color $\kappa$ 
between the vertices $x$ and $y$. In particular, in the case that $X_{xy}^\kappa = 0$ we reveal no other information
other than there is not an edge of color $\kappa$.

We now define the following process on $\mathcal{G}(n,p,2)$ with vertex set $V$ and an arbitrary ordering 
$\tau$ of $V$.
\begin{process}{\textbf{\textit{Search\_Neighborhood}}(u, $\gamma$)} \label{process: construct G tilde}
    \begin{enumerate}
        \item Initialize: $F \coloneqq \{u\}$, $\widetilde{V} \coloneqq \{u\}$, $\widetilde{E} \coloneqq \emptyset$. 
        \item Define a function $\psi_u: \widetilde{V} \xrightarrow{} \{red, blue\}$ with $\psi_u(u) = \gamma$, which assigns a 
        color to each vertex in $x \in \widetilde{V}$ that can be used to extend the alternating path from $u$ to $x$.
        \item Pick minimal $v \in F$ according to $\tau$ (until $F = \emptyset$). \label{G tilde: step 2}
        \begin{enumerate}[label=\alph*., ref=\theenumi\alph*]
            \item For all $x \in V \setminus \widetilde{V}$ (in order according to $\tau$): \label{G tilde: exploratory step}
                \begin{itemize}
                        \item Reveal $X^{\psi_u(v)}_{vx}$.
                            \begin{itemize}
                                \item If $X^{\psi_u(v)}_{vx} = 1$: add $x$ to $\widetilde{V}$ and $F$, $vx$ to $\widetilde{E}$, set $\psi_u(x)$ to the color which is not $\psi_u(v)$.
                            \end{itemize}
                \end{itemize}
            \item Remove $v$ from $F$. \label{G tilde: step 4}
            \item For all pairs $(w,z)$ where either: i) both $w$ and $z$ were just added to $F$ with $z < w$, or ii) $w$ was just added to $F$ and $z \in F$ but not newly added: \label{G tilde: step 5}
                \begin{itemize}
                        \item Reveal $X^{\psi_u(f)}_{zw}$.
                            \begin{itemize}
                                \item If $X^{\psi_u(f)}_{zw} = 1$: add $zw$ to $\widetilde{E}$.
                            \end{itemize}
                \end{itemize}
        \end{enumerate}
    \end{enumerate}
\end{process}   

We call step~\ref{G tilde: exploratory step} the \textit{exploratory} 
step and step~\ref{G tilde: step 5} the \textit{backtracking} step. The purpose of the 
backtracking step is to allow us to compare the distribution of the final graph 
obtained by \hyperref[process: construct G tilde]{\textit{Search\_Neighborhood}} when run on $\mathcal{G}(n,p,2)$
to the distribution of a connected component of $\mathcal{G}(n,p/2)$, as we will see in 
Lemma~\ref{lma: G tilde distribution}. 

\begin{definition}
    Let $\widetilde{G}_{\gamma, n,p}(u)$ denote the distribution of the graph obtained by 
    running \\ \hyperref[process: construct G tilde]{\textit{Search\_Neighborhood}}($u, \gamma$) on 
    $\mathcal{G}(n,p,2)$. When the parameters $n$ and $p$ are clear from context we denote 
    $\widetilde{G}_{\gamma, n,p}(u)$ as $\widetilde{G}_\gamma(u)$.
\end{definition}

First, we make the following observation which holds at the end of 
\hyperref[process: construct G tilde]{\textit{Search\_Neighborhood}}.

\begin{observation} \label{ob: alt search process}
    Let $\widetilde{V}$ denote the vertex set of $\widetilde{G}_\gamma(u)$ and $\psi_u : \widetilde{V} \rightarrow \{red, blue\}$ 
    be the function defined in \hyperref[process: construct G tilde]{\textit{Search\_Neighborhood}}(u, $\gamma$). 
    \begin{enumerate}[label=\alph*.]
        \item For $x \in \widetilde{V}$, $y \notin \widetilde{V}$, and $col \neq \psi_u(x)$ the random variable $X^{col}_{xy}$ was not revealed.
        \item For $x \in \widetilde{V}$ and $y \notin \widetilde{V}$ the random variable $X^{\psi_u(x)}_{xy}$ was revealed and equals 0. 
        \item For $x, y \in \widetilde{V}$ precisely one of $X_{xy}^{red}$ or $X_{xy}^{blue}$ was revealed.
        \item For $x, y \notin \widetilde{V}$ neither $X_{xy}^{red}$ or $X_{xy}^{blue}$ was revealed.
    \end{enumerate}
\end{observation}

Let us now consider the underlying distribution of $\widetilde{G}_{\gamma, n, p}(u)$. 

\begin{lemma} \label{lma: G tilde distribution}
    Given $u \in V$ and $\gamma \in \{red, blue\}$, the uncolored projection of 
    $\widetilde{G}_\gamma(u) \coloneqq \widetilde{G}_{\gamma, n, p}(u)$ has the same distribution as 
    $C_{n, p/2}(u)$, the connected component of $u$ in $\mathcal{G}(n,p/2)$.
\end{lemma}

\begin{proof}
    Observation~\ref{ob: alt search process} (c) implies that while running \hyperref[process: construct G tilde]{\textit{Search\_Neighborhood}}(u, $\gamma$),
    $\mathbb{P}[X^{red}_{xy}] = \mathbb{P}[X^{blue}_{xy}] = p/2$ for each random variable considered. This implies
    a natural coupling between \\ \hyperref[process: construct G tilde]{\textit{Search\_Neighborhood}}(u, $\gamma$) when run on $\mathcal{G}(n,p,2)$
    and the following process which we run in parallel but on $\mathcal{G}(n,p/2)$.

    \begin{process}{\textbf{\textit{Find\_Component}}(u)} \label{process: find component}
        \mbox{}
        \begin{enumerate}
        \item Initialize: $F \coloneqq \{u\}$, $\widetilde{V} \coloneqq \{u\}$, $\widetilde{E} \coloneqq \emptyset$. 
        \item Pick minimal $v \in F$ according to $\tau$ (until $F = \emptyset$).
        \begin{enumerate}[label=\alph*., ref=\theenumi\alph*]
            \item For all $x \in V \setminus \widetilde{V}$ (in order according to $\tau$):
                \begin{itemize}
                        \item Reveal whether $vx$ is an edge in $\mathcal{G}(n,p/2)$.
                            \begin{itemize}
                                \item If Yes: add $x$ to $\widetilde{V}$ and $F$, $vx$ to $\widetilde{E}$.
                            \end{itemize}
                \end{itemize}
            \item Remove $v$ from $F$.
            \item For all pairs $(w,f)$ where either: i) both $w$ and $f$ were just added to $F$ with $f < w$, or ii) $w$ was just added to $F$ and $f \in F$ but not newly added:
                \begin{itemize}
                        \item Reveal whether $fw$ is an edge in $\mathcal{G}(n,p/2)$. 
                            \begin{itemize}
                                \item If Yes: add $fw$ to $\widetilde{E}$.
                            \end{itemize}
                \end{itemize}
        \end{enumerate}
    \end{enumerate}
    \end{process}

    In fact, this is precisely Breadth First Search but with an extra backtracking step, 
    which ensures each pair of vertices in the component is checked for an edge. 
    Therefore, it is not hard to see that \hyperref[process: find component]{\textit{Find\_Component}}
    explores $\mathcal{G}(n,p/2)$ for the connected component of $u$. 
    Thus implying that $|V(\widetilde{G}_{\gamma, n, p})|$ follows the same distribution as $|C_{n,p/2}(u)|$,
    the size of the connected component containing $u$ in $\mathcal{G}(n,p/2)$. 
    Therefore, the backtracking step implies that $V(\widetilde{G}_{\gamma, n, p})$ follows precisely the same distribution
    as $C_{n,p/2}(u)$. 
\end{proof}

We get the following corollary of Lemma~\ref{lma: G tilde distribution}. 
It is well know that in the subcritical regime of the giant component, i.e. $p = \frac{1 - \epsilon}{n}$ 
for $\epsilon \gg n^{-1/3}$, that the  
expected size of a connected component scales as $1/\epsilon$ 
(see Chapter 11 of \cite{TheProbabilisticMethod} or Chapter 5 of 
\cite{Connected_component_rand_graphs}). Specifically, applying Lemma~\ref{lma: G tilde distribution} we have:

\begin{corollary} \label{cor: expected component size}
    Let $p = \frac{2(1 - \lambda_n n^{-1/3})}{n}$. For $x \in V$ and $\gamma \in \{red, blue\}$ we have,
    \begin{equation*}
        \mathbb{E}[|V(\widetilde{G}_{\gamma, n, p}(x))|] = \frac{n^{1/3}}{\lambda_n}(1 + o(1)) =  \frac{2}{p\lambda_n n^{2/3}} (1 + o(1)) \,.
    \end{equation*}
\end{corollary}

We now wish to determine the probability that $|V(\widetilde{G}_\gamma(x))| = k$ and that $\widetilde{G}_\gamma(x)$ is a tree. Define the following,
\begin{equation*}
    R_{n,p}(x,k) \coloneqq \mathbb{P}[|V(\widetilde{G}_{\gamma, n,p}(x))| = k \ \wedge \ \widetilde{G}_{\gamma, n, p}(x) \ is \ a \ tree] \,.
\end{equation*}

Before preceding, we make the following observation which follows from Lemma~\ref{lma: G tilde distribution}.

\begin{observation} \label{ob: uniform tree}
    Given $\widetilde{G}_\gamma(x)$ is a tree, it is uniformly distributed among all trees on $|V(\widetilde{G}_\gamma(x))|$ vertices.
\end{observation}

This follows as for a fixed number of edges $l$ and vertices $k$, $C_{n, p/2}(x)$ is uniform over all
connected graphs with $k$ vertices and $l$ edges. 

We now provide an expression for $R_{n,p}(k)$, the proof of which we defer to 
Appendix~\ref{appendix: R(k)}. The calculations to derive Lemma~\ref{lma: R(k) expression} 
are standard and contain only small deviations from those found in \cite{bolobas-rand-graph} 
and Chapter 11.10 in \cite{TheProbabilisticMethod}.

\begin{lemma} \label{lma: R(k) expression}
    If $\lambda_n \xrightarrow[]{} \infty$, $\lambda_n \ll n^{1/3}$, and $k/n^{2/3} \xrightarrow[]{} 0$ then for $p = \frac{2(1 - \lambda_n n^{-1/3})}{n}$:
    \begin{equation*}
        R_{n,p}(k) = \frac{2}{pn\sqrt{2\pi}k^{3/2}} e^{-\frac{k}{2}(\lambda_n n^{-1/3})^2\left(1 \pm o(1)\right)}
    \end{equation*}
\end{lemma}

We will often run \hyperref[process: construct G tilde]{\textit{Search\_Neighborhood}}
on a subset of the vertices of size $n' = n - \lambda_n n^{2/3}$. The following corollary of Lemma~\ref{lma: R(k) expression} will be useful:

\begin{corollary} \label{cor: G tilde tree}
    There exists constant $\phi > 0$ such that for $x \in V$, $\gamma \in \{red, blue\}$, $\lambda_n \ll n^{1/3}$, and $p = \frac{2(1 - \lambda_n n^{-1/3})}{n}$ and $n' = n - \lambda_nn^{2/3}$ the probability that $\widetilde{G}_{\gamma, n', p}(x)$ is a tree of size between $2n^{2/3}/\lambda_n^2$ and $4n^{2/3}/\lambda_n^2$ is $(\phi + o(1))\frac{\lambda_n}{n^{1/3}}$.
\end{corollary}

\begin{proof}
    Choose $\lambda_n'$ such that the following holds,
    \begin{equation*}
        \frac{2(1 - \lambda_nn^{-1/3})}{n} = p =  \frac{2(1 - \lambda_n'(n')^{-1/3})}{n'} \,,
    \end{equation*}
    straightforward calculations show $\lambda_n' = \lambda_n(2 - o(1))$.
    From Lemma~\ref{lma: R(k) expression} we have,
    \begin{align*}
        \sum_{k = 2n^{2/3}/\lambda_n^2}^{4n^{2/3}/\lambda_n^2} R_{n',p}(k) &= \frac{2}{pn'\sqrt{2\pi}} \sum_{k = 2n^{2/3}/\lambda_n^2}^{4n^{2/3}/\lambda_n^2} \frac{1}{k^{3/2}}e^{-\frac{k}{2} (\lambda_n' (n')^{-1/3})^2 (1 \pm o(1))} \\
        &= \frac{\mathcal{O}(1)}{1 - \lambda_n'(n')^{-1/3}} \frac{\lambda_n^3}{n} 2\frac{n^{2/3}}{\lambda_n^2} \\
        &= (\phi + o(1))\frac{\lambda_n}{n^{1/3}} \,, \qquad \text{for some constant $\phi > 0$} \,.
    \end{align*}
    The second equality follows as $k = \Theta(n^{2/3}/\lambda_n^2)$ for each term in the sum. 
\end{proof}

\subsection{Hourglasses}

We now define a subgraph that will serve as a ``building block" for finding alternating bicycles. 
This subgraph is precisely an edge colored variant of the hourglass used in \cite{2-sat-critical-window}.

\begin{definition} \label{def: hourglass}
    An Hourglass is defined as a triple $(u,R,B)$ where $R, B, \{u\}$ are all pair-wise disjoint and for all $x \in R$, $x$ is reachable from $u$ by an alternating path $P_1 = u v_1 v_2 \cdots v_k x$ where $c(uv_1) = red$ and $V(P_1) \subseteq R \cup \{u\}$. Similarly, for all $y \in B$, $y$ is reachable from $u$ by an alternating path $P_2 = u w_1 w_2 \cdots w_l y$ where $c(uw_1) = blue$ and $V(P_2) \subseteq B \cup \{u\}$. 
    Furthermore, we call $u$ the center of the hourglass and the sets $R$ and $B$ the components. 
\end{definition}

We shall show that for $p = 2(1 - \lambda_n n^{-1/3})/n$, there exists w.h.p. a 
lot of ``small" hourglasses. These will serve as ``building blocks" for alternating bicycles.

\begin{lemma} \label{lma: hourglass}
    Suppose $\lambda_n \xrightarrow[]{} \infty$ arbitrarily slowly. There exists a constant $\psi > 0$ such that w.h.p. 
    $\mathcal{G}(n, \frac{2(1 - \lambda_nn^{-1/3})}{n}, 2)$ 
    has at least $\psi \lambda_n^3$ vertex disjoint hourglasses with 
    respective sets $R$ and $B$ of size at least $n^{2/3} / \lambda_n^2$.
\end{lemma}

\subsubsection{Proof of Lemma~\ref{lma: hourglass}}

This proof is a direct adaptation of the work by Bollobás et al. \cite{2-sat-critical-window} where they prove the existence of a similar structure but in the directed graph model.

The proof relies heavily on the method of ``deferred decisions". 
We reveal the graph in a very structured manner while searching for hourglasses, 
making sure that we keep track of what information was revealed and when as well as 
which vertices were considered in each step.

First, recall Lemma~\ref{lma: prob change} regarding the random variables $X^{red}_{xy}$ and $X^{blue}_{xy}$. 

\begin{lemma*}[\ref{lma: prob change}] 
    Let $c_1, c_2 \in \{red,blue\}$ with $c_1 \neq c_2$. If $X^{c_2}_{xy}$ was revealed and equals 0 then the following hold:
    \begin{itemize}
        \item the probability that $X^{c_1}_{xy} = 1$ given $X^{c_2}_{xy} = 0$ is $\frac{p/2}{1-p/2} > p/2$.
        \item the probability $xy \in E$ given $X^{c_2}_{xy} = 0$ is $\frac{p/2}{1-p/2} < p$.
    \end{itemize}
\end{lemma*}

\begin{proof}[Proof of Lemma~\ref{lma: hourglass}]

   Before proceeding we provide a roadmap of the proof. First, we shall design a search process 
   \hyperref[process: hourglass]{Hourglass\_Process} which attempts to find an hourglass of the 
   desired size by starting at a single vertex. As one would expect the probability this succeeds is 
   quite small; in fact it goes to zero with $n$.
   So, we construct \hyperref[process: batched hourglass]{Batched\_Hourglass\_Process}
   which runs \hyperref[process: hourglass]{Hourglass\_Process} many times to boost the probability of success 
   to a constant. Each run of \hyperref[process: hourglass]{Hourglass\_Process} will 
   be on a subset of the vertices which were not used during previous runs. As a result, each run of \hyperref[process: hourglass]{Hourglass\_Process}
   will be independent.

   To accomplish this, we set aside a subset $S$ of $\lambda_n n^{2/3}$ 
   vertices from which we will replenish the used vertices. 
   We begin by only searching on $V \setminus S$. 
   At certain points, we will discard some previously used vertices and replace them 
   with vertices from $S$. Let $n' = n - \lambda_n n^{2/3}$.

   At a high level we search for a single hourglass as follows: 
   We pick an unused vertex $v$ and specify a color, say $red$. 
   We search for the set of vertices reachable from $v$ by alternating 
   paths beginning with $red$. From the subgraph obtained by this search, we obtain a long alternating path.
   The middle vertex $u$ of the path will serve as the center of our hourglass, and all vertices 
   which reach $v$ by first going through $u$ will form a component of the hourglass, say $B_u$
   (see the set of vertices enclosed by a blue dashed circle in Figure~\ref{fig: hourglass process}). 
   We then observe that each vertex in the path $P_{v,u}$ from $v$ to $u$ is reachable from $u$ by 
   an alternating path beginning with $red$. Thus, we can form the other component of the hourglass, $R_u$ 
   by searching for vertices which are reachable from each 
   vertex of $P_{v,u}$ by an alternating path beginning with the appropriate color.  
   
   \begin{figure}[h]
    \centering
        \begin{tikzpicture}[scale = 0.9,node distance={15mm}, main/.style = {draw, circle, thick, fill=black} ]
                \node[main, scale = 0.75] (u) at (0,0) {}; 
                \node[main, draw=none, fill=none] (u_name) at (0,-0.3) {$v$}; 
                \node[draw = none] at (u) [name=fake_u,outer sep=5pt,inner sep=5pt]{};
                
                \node[main, scale =0.75] (w) at (14.25, 0) {};
                \node[main, draw=none, fill=none] (w_name) at (14.25,-0.3) {$w$}; 
                \node[draw = none] at (w) [name=fake_w,outer sep=5pt,inner sep=5pt]{};

                \node[main, scale = 0.75] (v) at (7.5,0) {};
                \node[draw = none] at (v) [name=fake_v,outer sep=5pt,inner sep=5pt]{};
                \node[main, draw=none, fill=none] (v_name) at (7.5,-0.3) {$u$};

                \draw (8,0) ellipse (8.5cm and 2.3cm);

                \draw[blue, dashed, very thick] (11.85,0) ellipse (4.2cm and 1.8cm);
                

                \node[main, scale = 0.5] (u_1) at (2, 0) {};
                \node[draw = none] at (u_1) [name=fake_x_1,outer sep=5pt,inner sep=5pt]{};
                \node[main, draw=none, fill=none] (x_1) at (2,-0.3) {$x_1$}; 

                \node[main, scale = 0.5] (u_2) at (4,0) {};
                \node[draw = none] at (u_2) [name=fake_x_2,outer sep=5pt,inner sep=5pt]{};
                \node[main, draw=none, fill=none] (x_2) at (4,-0.3) {$x_2$}; 
                
                \node[main, draw=none, fill=none] (u_dots) at (5.25, 0) {$\cdots$};
                \node[main, draw=none, fill=none] (Puv) at (5.25,0.4) {$P_{v,u}$}; 
                
                \node[main, scale = 0.5] (u_3) at (6.5, 0) {};

                \draw[red] (u)  --  (u_1);
                \draw[blue, very thick] (u_1) -- (u_2);
                 
                \draw[red] (u_2) -- (u_dots);
                \draw[blue, very thick] (u_dots) -- (u_3);
        		\draw[red] (u_3) -- (v);

                \node[main, scale = 0.5] (v_1) at (8.75, 0) {};
                \node[main, scale = 0.5] (v_2) at (10, 0) {};
                \node[main, scale = 0.5] (v_3) at (11.25, 0) {};
                \node[main, draw=none, fill=none] (v_4) at (12.125, 0) {$\cdots$};
                \node[main, scale = 0.5] (v_5) at (13, 0) {};
                
                \draw[blue, very thick] (v) -- (v_1);
                \draw[red] (v_1) -- (v_2);
                \draw[blue, very thick] (v_2) -- (v_3);
                \draw[red] (v_3) -- (v_4);
                \draw[blue, very thick] (v_4) -- (v_5);
                \draw[red] (v_5) -- (w);

                
                \node[main, scale = 0.25] (b_1) at (8.75, 1) {};
                \node[main, draw = none, fill = none, rotate=17] (b_1_1) at (9.5, 1.25) {$\cdots$};
                \node[main, draw = none, fill = none, rotate=-17] (b_1_2) at (9.5, 0.75) {$\cdots$};
                
                \node[main, scale = 0.25] (b_2) at (8.75, -1) {};
                \node[main, draw = none, fill = none, rotate=-17] (b_2_1) at (9.5, -1.25) {$\cdots$};
                \node[main, draw = none, fill = none, rotate=17] (b_2_2) at (9.5, -0.75) {$\cdots$};
                
                \node[main, draw = none, fill = none, rotate=38] (b_3) at (11, 0.75) {$\cdots$};
                \node[main, draw = none, fill = none, rotate=-38] (b_4) at (12.25, -0.75) {$\cdots$};

                \node[main, draw = none, fill = none, rotate=38] (b_5) at (14, 0.75) {$\cdots$};
                \node[main, draw = none, fill = none, rotate=-38] (b_6) at (14, -0.75) {$\cdots$};

                \node[main, draw = none, fill = none, rotate=38] (b_w_1) at (15.25, 0.75) {$\cdots$};
                \node[main, draw = none, fill = none, rotate=-38] (b_w_2) at (15.25, -0.75) {$\cdots$};

                \draw[blue, very thick] (v) -- (b_1);
                \draw[red] (b_1) -- (b_1_1);
                \draw[red] (b_1) -- (b_1_2);
                
                \draw[blue, very thick] (v) -- (b_2);
                \draw[red] (b_2) -- (b_2_1);
                \draw[red] (b_2) -- (b_2_2);
                
                \draw[blue, very thick] (v_2) -- (b_3);
                \draw[red] (v_3) -- (b_4);

                \draw[red] (v_5) -- (b_5);
                \draw[red] (v_5) -- (b_6);

                \draw[blue, very thick] (w) -- (b_w_1);
                \draw[blue, very thick] (w) -- (b_w_2);


                \node[main, draw = none, fill = none] (x_0_1) at (-0.9, 3) {};
                \node[main, draw = none, fill = none] (x_0_2) at (0.9, 3) {};

                \draw[blue, very thick] (u) -- (x_0_1);
                \draw[blue, very thick] (u) -- (x_0_2);

                \draw[blue, very thick] (0,3) ellipse (0.9cm and 0.5cm);
                \node[main, draw = none, fill = none] (x_red_graph) at (0, 3) {$\widetilde{G}_{blue}(x_0)$};


                \node[main, draw = none, fill = none] (x_1_1) at (1.1, 3) {};
                \node[main, draw = none, fill = none] (x_1_2) at (2.9, 3) {};

                \draw[red] (u_1) -- (x_1_1);
                \draw[red] (u_1) -- (x_1_2);

                \draw[red] (2,3) ellipse (0.9cm and 0.5cm);
                \node[main, draw = none, fill = none] (x_red_graph) at (2, 3) {$\widetilde{G}_{red}(x_1)$};
                

                \node[main, draw = none, fill = none] (x_2_1) at (3.1, 3) {};
                \node[main, draw = none, fill = none] (x_2_2) at (4.9, 3) {};

                \draw[blue, very thick] (u_2) -- (x_2_1); 
                \draw[blue, very thick] (u_2) -- (x_2_2);

                \draw[blue, very thick] (4,3) ellipse (0.9cm and 0.5cm);
                \node[main, draw = none, fill = none] (x_blue_graph) at (4, 3) {$\widetilde{G}_{blue}(x_2)$};
                
                \node[main, draw = none, fill = none] (G_tilde_dots) at (5.4,3) {$\cdots$};
    
        \end{tikzpicture}
        \caption{A depiction of an iteration of \hyperref[process: hourglass]{Hourglass\_Process} up to Step~\ref{Hourglass proc: step 3d}, 
        assuming all previous steps succeeded. The set of vertices enclosed by a blue dashed oval denotes
        candidate set $B_v$ for the hourglass, that is a set of vertices reachable from $v$ by an alternating path
        beginning with $blue$. Furthermore, for each $x_i$ along the $v,u$ path $P_{v,u}$, 
        the vertex set of $\widetilde{G}_{col}(x_i)$ is reachable from $u$ 
        by an alternating path beginning with $red$, by first taking the path from $u$ to $x_i$. }
        \label{fig: hourglass process}
    \end{figure}
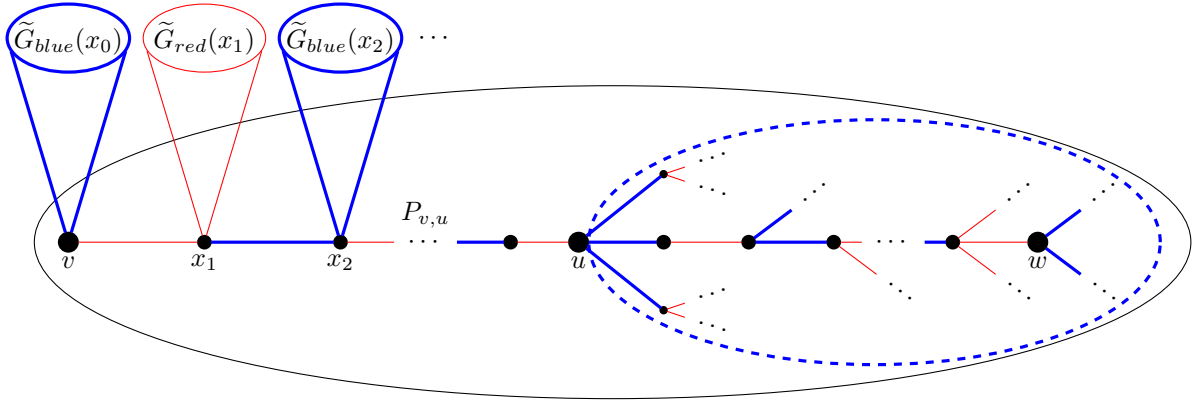
   
   We now formally define \hyperref[process: hourglass]{Hourglass\_Process}. 
   We will make use of a ``trimmed" version of 
   \hyperref[process: construct G tilde]{\textit{Search\_Neighborhood}}, which we define formally later. 
   We use ``large" and ``long" as placeholders for concrete values which will be provided in \ref{large tree def} and \ref{long path in tree def} later in the analysis. 

   \begin{process}{\textbf{\textit{Hourglass\_Process}}(V', S)} \label{process: hourglass}
    \begin{itemize}
        \item If at any point $S$ becomes empty, report \textit{fail}.
    \end{itemize}
    \begin{enumerate}
        \item \label{Hourglass proc: step 1} Pick a vertex $v \in V'$. 
        \item \label{Hourglass proc: step 2} Run \hyperref[process: construct G tilde]{\textit{Search\_Neighborhood}($v, red$)} on the 
        vertex set $V'$. Let $F_{red}(v) = (\widetilde{V}_v, \widetilde{E}_v)$ be the resulting graph.
        \item \label{Hourglass proc: step remove used from v} Remove the vertices in $F_{red}(v)$ from $V'$ and replace them with vertices from $S$.
        \item \label{Hourglass proc: step check large tree} Check: If $|V(F_{red}(v))|$ is a ``large" tree: continue; Otherwise, report \textit{fail}.
        \item \label{Hourglass proc: sample w} Sample a vertex $w$ from $V(F_{red}(v))$ uniformly at random.
        \item Check: If the $v,w$-path $P_{v,w}$ is ``long": continue; Otherwise, report \textit{fail}.
        \item Let $u$ denote the middle vertex in $P_{v,w}$.
        \item \label{hourglass proc: step B_u formed} Check: If at least half of the vertices in $F_{red}(v)$ connect to $v$ through $u$: continue; Otherwise, report \textit{fail}.
        \item \label{Hourglass proc: step 3d} For $x$ in the $v,u$-path $P_{v,u}$:
            \begin{enumerate}[label=\alph*., ref=\theenumi\alph*]
                \item \label{Hourglass proc: step 3a} Run \hyperref[proc: trimmed search]{\textit{Trimmed\_Search\_Neighborhood}}($x, \gamma_x$) on $V'$ for the appropriate color $\gamma_x$. Let $F_{\gamma_x}(x) = (\widetilde{V}_x, \widetilde{E}_x)$ be the resulting graph.
                \item \label{Hourglass proc: remove used path search} Remove the vertices in $F_{\gamma_x}(x)$ from $V'$ and replace them with vertices from $S$.
                \item \label{Hourglass proc: step success} Check: If $F_{\gamma_x}(x)$ is a ``large" tree: report \textit{success}; Otherwise, continue.
            \end{enumerate}
        \item Report \textit{fail}.
        \end{enumerate}
   \end{process}

   We first note that, if at any point \hyperref[process: hourglass]{Hourglass\_Process} 
   outputs \textit{fail} or \textit{success} we immediately abort the process without considering 
   future steps. 

   As one might expect, the probability \hyperref[process: hourglass]{Hourglass\_Process} finds an hourglass
   is $o(1)$. However, we will show:
   \begin{equation}
        \mathbb{P}[\hyperref[process: hourglass]{Hourglass\_Process} \ succeeds] \geq (1 + o(1))\frac{\phi(1 - e^{-\phi})}{2e}\frac{\lambda_n}{n^{1/3}} \,.
   \end{equation}
   We will run \hyperref[process: hourglass]{Hourglass\_Process} many times.
   We make the following observations regarding each run:
   
    \begin{observation} \label{ob: hourglass process}
        \mbox{}
        \begin{enumerate}[label=\alph*.]
            \item \label{hourglass process: ob 1} Before each search step (i.e. step~\ref{Hourglass proc: step 2} or \ref{Hourglass proc: step 3a}), for all $a, b \in V'$ neither $X^{red}_{ab}$ or $X^{blue}_{ab}$ have been revealed.
            \item \label{hourglass process: ob 2} In step~\ref{Hourglass proc: step 3d}, for $x \in P_{v,u}$ and $v' \in V'$ the following hold before step~\ref{Hourglass proc: step 3a}:
                \begin{enumerate}[label=\roman*.]
                    \item $X^{\gamma_{x}}_{xv'}$ has not been revealed.
                    \item For $col_x \neq \gamma_x$ if $X^{col_x}_{xv'}$ has been revealed it equates to $0$.
                \end{enumerate}
        \end{enumerate}
    \end{observation}

    These observations follow directly from Observation~\ref{ob: alt search process} 
    regarding the search process \\
    \hyperref[process: construct G tilde]{\textit{Search\_Neighborhood}} as 
    well as the replenishing steps~\ref{Hourglass proc: step remove used from v} and \ref{Hourglass proc: remove used path search}.
    
    Observation~\ref{ob: hourglass process} \ref{hourglass process: ob 1} implies that $F_{red}(v)$ is sampled directly from $\widetilde{G}_{red, n', p}(v)$. 

    We wish to use Corollary~\ref{cor: G tilde tree} and Corollary~\ref{cor: expected component size} 
    on each $F_{\gamma_x}(x)$ considered
    in step~\ref{Hourglass proc: step 3d} to upper bound the expected size and lower bound the probability $F_{\gamma_x}(x)$ is a ``large" tree.
    However, during the search procedure in step~\ref{Hourglass proc: step 2}, 
    we may have revealed some information regarding the neighborhood of $x$ in $V'$. Particularly, 
    for $y \in V'$, we may have revealed that there is no edge of color $col_x \neq \gamma_{x}$ 
    between $x$ and $y$.
    
    For instance, consider $x_2$ in Figure~\ref{fig: hourglass process}.
    We have $\gamma_{x_2} = blue$ and $col_x = red$. 
    During the search in step~\ref{Hourglass proc: step 2}
    we may have revealed $X^{red}_{x_2y}$ for some $y \in V'$. 
    It must 
    be the case that $X^{red}_{x_2y} = 0$, as $y$ was not added to $F_{red}(v)$. 
    I.e. there is no red edge between $x_2$ and $y$.

    By Lemma~\ref{lma: prob change} this only increases the probability of an edge 
    of color $\gamma_{x_i}$, i.e. the edge we want to find. Specifically, this exposure implies that for $x \in P_{v,u}$ and $y \in V'$ the probability of an edge of color $\gamma_x$ between $x$ and $y$ is now either $p/2$ or $\frac{p/2}{1 - p/2}$.

    Consider $x \in P_{v,u}$. In step~\ref{Hourglass proc: step 3a} it would be natural to 
    form $F_{\gamma_x}(x)$ by running \\ \hyperref[process: construct G tilde]{\textit{Search\_Neighborhood}}($x, \gamma_x$). However, when searching for neighbours of $x$ 
    we might query some $X^{\gamma_x}_{xy}$ which is $1$ with probability $\frac{p/2}{1 - p/2}$ rather than $p/2$.
    As a result, the graph formed from \hyperref[process: construct G tilde]{\textit{Search\_Neighborhood}}($x, \gamma_x$) is not distributed as $\widetilde{G}_{\gamma_x, n', p}(x)$.
    To shift the distribution back to $\widetilde{G}_{\gamma_x, n', p}(x)$ we introduce the following 
    modification:

    \smallskip
    \textbf{Trimming Alteration($x, \gamma_x$):} \label{proc: trimming alteration} \\
    \mbox{}
    \indent Let $col \neq \gamma_x$. When considering the potential edge $xy$: 
    \begin{itemize}
        \item If $X^{col}_e$ was already revealed:
            \begin{itemize}
                \item With probability $p/2$, don't query $X^{\gamma_x}_e$ and skip the edge $e$.
                \item Otherwise, query $X^{\gamma_x}_e$.
            \end{itemize}
        \item If $X^{col}_e$ has not already been revealed, query $X^{\gamma_x}_e$.
    \end{itemize}
    
    We call this a \textit{trimming} alteration, as we are effectively trimming off branches in the 
    graph branching process, i.e. there may be some edges $xy$ of color $\gamma_x$ which are not included in $F_{\gamma_x}(x)$. We define the following process:
    
    \begin{process} {\textbf{\textit{Trimmed\_Search\_Neighborhood}}$(x, \gamma_x)$} \label{proc: trimmed search} 
    	\mbox{}
	\begin{itemize}
		\item Run \hyperref[process: construct G tilde]{\textit{Search\_Neighborhood}}($x, \gamma_x$) with the \hyperref[proc: trimming alteration]{Trimming Alteration}.
	\end{itemize}
    \end{process}

    By Observation~\ref{ob: hourglass process} \ref{hourglass process: ob 1}, 
    the only edges with information 
    revealed have $x$ as an endpoint. So we need only apply the Trimming Alteration to edges with $x$ 
    as an endpoint.

    Therefore, the probability each considered edge is added to $\widetilde{E}_x$ is now $p/2$. 
    Thus, the resulting distribution of the $F_{\gamma_x}(x)$ produced by \hyperref[proc: trimmed search]{\textit{Trimmed\_Search\_Neighborhood}}$(x, \gamma_x)$
    is precisely $\widetilde{G}_{\gamma_x, n', p}(x)$.

    To recap:
    \begin{itemize}
        \item $F_{red}(v)$ is sampled directly from $\widetilde{G}_{red, n', p}(v)$.
        \item For $x \in P_{v,u}$, $F_{\gamma_x}(x)$ is sampled directly from $\widetilde{G}_{\gamma_x, n', p}(x)$.
    \end{itemize}

    We now determine the probability \hyperref[process: hourglass]{Hourglass\_Process}
    returns \textit{success} as well as the expected number of vertices used. 
    
    We start by analyzing the probability steps~\ref{Hourglass proc: step check large tree} - \ref{hourglass proc: step B_u formed} succeed.
    
    In step~\ref{Hourglass proc: step 2} we construct $F_{red}(v)$. 
    By Lemma~\ref{lma: G tilde distribution}, the uncolored projection of $F_{red}(v)$ is distributed as $C_{n',p/2}(v)$, the connected component of $v$ in $\mathcal{G}(n', p/2)$. We define $\lambda_n'$ such that,
    \begin{equation*}
        \frac{2(1 - \lambda_n n^{-1/3})}{n} = p = \frac{2(1 - \lambda_{n}' (n')^{-1/3})}{n'} \,.
    \end{equation*}
    A simple calculation shows $\lambda_{n'} = \lambda_n(2 - o(1))$. 
    
    Consider $T = F_{red}(v)$. By Corollary~\ref{cor: G tilde tree}, there exists a constant $\phi > 0$ such that with probability $(\phi + o(1))\lambda_n/n^{1/3}$, $T$ is a tree of size between $2n^{2/3}/ \lambda_n^2$ and $4n^{2/3} / \lambda_n^2$; in this case call $T$ a \textit{large} tree. 
    \begin{equation} \label{large tree def}
    	\text{Call a tree T \textit{large} if } |T| \in [2n^{2/3} / \lambda_n^2, 4 n^{2/3} / \lambda_n^2] \,.
    \end{equation}
    Furthermore, if $T$ is a tree Observation~\ref{ob: uniform tree} implies $T$ is a uniformly random tree on $|T|$ vertices.
    
    In step~\ref{Hourglass proc: sample w} we pick a random vertex $w \neq u$ in $T$. Let $P_{v,w}$ denote the path from $v$ to $w$ in $T$. 
    By properties of uniformly random trees (see \cite{Unif-rand-trees} and 
    \cite{2-sat-critical-window}), with probability $(1 - o(1))e^{-1}$ $P_{v,w}$ has length at least $\sqrt{2|T|}$; if this is the case call $P_{v,w}$ \textit{long}. 
    \begin{equation} \label{long path in tree def}
    	\text{Call a path } P \text{ in a tree } T \text{ \textit{long} if the length of } P \text{ is at least } \sqrt{2|T|} \,.
    \end{equation}
    Let $u$ denote the middle vertex in $P_{v,w}$; in case of a tie, choose the vertex closer to $v$. 
    Let $P_{v,u} \subset P_{v,w}$ be the path from $v$ to $u$ and let $P_{u,w} \subset P_{v,w}$ be the 
    path from $u$ to $w$. As $T$ is a tree, for each vertex $z \in T$ it follows that the path from $z$ to $u$ either intersects $P_{u,w} - u$, $P_{v,u} - u$, or neither (i.e. it connects directly to $u$).
    As $T$ is a uniformly random spanning tree, $v$ can also be viewed as a uniformly random vertex in $T$. By symmetry,
    with probability at least $1/2$, at least half of the vertices in $T$ will be connected to $u$ 
    via $P_{u,w} - u$ or directly to $u$, call this set $A_u$.
     
    Call $u$ \textit{promising} if $P_{v,u}$ is \textit{long} and $|A_u| \geq \frac{1}{2}|T| \geq n^{2/3}/\lambda_n^2$. Thus,    \begin{equation} \label{eq: u promising}
        \mathbb{P}[F_{red}(v) \ is \ a \ large \ tree \ \bigwedge  u \ is \ promising ] \geq (1 + o(1)) \frac{\phi}{2e} \frac{\lambda_n}{n^{1/3}} \,.
    \end{equation}
    In the event that $u$ is promising, assume without loss of generality that $red$ 
    is the color of the edge incident to $u$ in $P_{v,u}$, noting 
    that the parity of the position of $u$ in $P_{v,w}$ has no influence on the analysis.
    
    We now analyze the probability step~\ref{Hourglass proc: step 3d} succeeds.

    We consider the first $\frac{n^{1/3}}{\lambda_n}$ vertices in $P_{v,u}$ in order 
    beginning with the vertex $v$, call this set $\Lambda$. By construction $P_{v,u}$ is an 
    alternating path. For each vertex $x \in \Lambda$, 
    let $\gamma_x$ be the color of the edge incident to $x$ in $P_{v,u}$ which is 
    closest to $v$; in the case that $x = v$ let $\gamma_x = blue$. 
    
    In step~\ref{Hourglass proc: step 3a} of \hyperref[process: hourglass]{Hourglass\_Process}, 
    we construct $F_{\gamma_x}(x)$. Note that, by steps~\ref{Hourglass proc: step remove used from v} and 
    \ref{Hourglass proc: remove used path search}, when constructing $F_{\gamma_x}(x)$ 
    the search is restricted to $n'$ unused vertices. \\
    \hyperref[proc: trimmed search]{\textit{Trimmed\_Search\_Neighborhood}}$(x, \gamma_x)$ produces $F_{\gamma_x}(x)$ with distribution $\widetilde{G}_{\gamma_x, n', p}(x)$. So, as described before (\ref{large tree def}) (i.e. by Corollary~\ref{cor: G tilde tree}), $F_{\gamma_x}(x)$ is a \textit{large} tree with probability at least $(\phi + o(1)) \lambda_n/n^{1/3}$. 
    Thus, the probability that $F_{\gamma_x}(x)$ is a large tree for at least one $x \in \Lambda$ is at least,
    \begin{equation} \label{eq: graph on path good}
        1 - \left(1 - (\phi + o(1)) \lambda_n/n^{1/3}\right)^{n^{1/3} / \lambda_n} = 1 - e^{-\phi} + o(1) \,.
    \end{equation} 
    
    If we find that $F_{\gamma_x}(x)$ is a \textit{large} tree, we stop our search. We have now found an hourglass with center vertex $u$ and components $R_u \coloneqq F_{\gamma_x}(x) \cup P_{x,u}$ and $B_u = A_u$, as seen in Figure~\ref{fig: hourglass process final hourglass}. Both $R_u$ and $B_u$ have size at least $n^{2/3}/\lambda_n^2$. 
    
    \begin{figure}[h]
    \centering
        \begin{tikzpicture}[node distance={15mm}, main/.style = {draw, circle, thick, fill=black} ]
                
                \node[main, scale =0.75] (w) at (14.25, 0) {};
                \node[main, draw=none, fill=none] (w_name) at (14.25,-0.3) {$w$}; 
                \node[draw = none] at (w) [name=fake_w,outer sep=5pt,inner sep=5pt]{};

                \node[main, scale = 0.75] (v) at (7.5,0) {};
                \node[draw = none] at (v) [name=fake_v,outer sep=5pt,inner sep=5pt]{};
                \node[main, draw=none, fill=none] (v_name) at (7.5,-0.3) {$u$};

                \draw[blue, dashed, very thick] (12,0) ellipse (4cm and 2.5cm);
                \node[main, draw=none, fill=none] (Bu) at (12.125, 1.5) {$B_u$}; 
                
                \draw[red, rotate around={45:(4.5,1.5)}] (4.5,1.5) ellipse (2cm and 3cm);
                \node[main, draw=none, fill=none] (Ru) at (6,1.8) {$R_u$}; 
                

                \node[main, scale = 0.5] (u_2) at (4,0) {};
                \node[draw = none] at (u_2) [name=fake_x_2,outer sep=5pt,inner sep=5pt]{};
                \node[main, draw=none, fill=none] (x_2) at (4,-0.3) {$x_i$}; 
                
                \node[main, draw=none, fill=none] (u_dots) at (5.25, 0) {$\cdots$};
                \node[main, draw=none, fill=none] (Puv) at (5.25,0.4) {$P_{x_i,u}$}; 
                
                \node[main, scale = 0.5] (u_3) at (6.5, 0) {};

                \draw[red] (u_2) -- (u_dots);
                \draw[blue, very thick] (u_dots) -- (u_3);
        		\draw[red] (u_3) -- (v);

                \node[main, scale = 0.5] (v_1) at (8.75, 0) {};
                \node[main, scale = 0.5] (v_2) at (10, 0) {};
                \node[main, scale = 0.5] (v_3) at (11.25, 0) {};
                \node[main, draw=none, fill=none] (v_4) at (12.125, 0) {$\cdots$};
                \node[main, scale = 0.5] (v_5) at (13, 0) {};
                
                \draw[blue, very thick] (v) -- (v_1);
                \draw[red] (v_1) -- (v_2);
                \draw[blue, very thick] (v_2) -- (v_3);
                \draw[red] (v_3) -- (v_4);
                \draw[blue, very thick] (v_4) -- (v_5);
                \draw[red] (v_5) -- (w);

                
                \node[main, scale = 0.25] (b_1) at (8.75, 1) {};
                \node[main, draw = none, fill = none, rotate=17] (b_1_1) at (9.5, 1.25) {$\cdots$};
                \node[main, draw = none, fill = none, rotate=-17] (b_1_2) at (9.5, 0.75) {$\cdots$};
                
                \node[main, scale = 0.25] (b_2) at (8.75, -1) {};
                \node[main, draw = none, fill = none, rotate=-17] (b_2_1) at (9.5, -1.25) {$\cdots$};
                \node[main, draw = none, fill = none, rotate=17] (b_2_2) at (9.5, -0.75) {$\cdots$};
                
                \node[main, draw = none, fill = none, rotate=38] (b_3) at (11, 0.75) {$\cdots$};
                \node[main, draw = none, fill = none, rotate=-38] (b_4) at (12.25, -0.75) {$\cdots$};

                \node[main, draw = none, fill = none, rotate=38] (b_5) at (14, 0.75) {$\cdots$};
                \node[main, draw = none, fill = none, rotate=-38] (b_6) at (14, -0.75) {$\cdots$};

                \node[main, draw = none, fill = none, rotate=38] (b_w_1) at (15.25, 0.75) {$\cdots$};
                \node[main, draw = none, fill = none, rotate=-38] (b_w_2) at (15.25, -0.75) {$\cdots$};

                \draw[blue, very thick] (v) -- (b_1);
                \draw[red] (b_1) -- (b_1_1);
                \draw[red] (b_1) -- (b_1_2);
                
                \draw[blue, very thick] (v) -- (b_2);
                \draw[red] (b_2) -- (b_2_1);
                \draw[red] (b_2) -- (b_2_2);
                
                \draw[blue, very thick] (v_2) -- (b_3);
                \draw[red] (v_3) -- (b_4);

                \draw[red] (v_5) -- (b_5);
                \draw[red] (v_5) -- (b_6);

                \draw[blue, very thick] (w) -- (b_w_1);
                \draw[blue, very thick] (w) -- (b_w_2);


                \node[main, draw = none, fill = none] (x_2_1) at (3.1, 3) {};
                \node[main, draw = none, fill = none] (x_2_2) at (4.9, 3) {};

                \draw[blue, very thick] (u_2) -- (x_2_1); 
                \draw[blue, very thick] (u_2) -- (x_2_2);

                \draw[blue, very thick] (4,3) ellipse (0.9cm and 0.5cm);
                \node[main, draw = none, fill = none] (x_blue_graph) at (4, 3) {$F_{blue}(x_i)$};
    
        \end{tikzpicture}
        \caption{The hourglass produced by \hyperref[process: hourglass]{Hourglass\_Process} assuming it returns \textit{success} on $x_i$ in step~\ref{Hourglass proc: step success}.
        The vertex $u$ is the center of the hourglass. The set enclosed by a blue dashed oval forms 
        the component $B_u$ and the set enclosed by the red oval forms the component $R_u$. 
        Furthermore, the path $x_i,u$-path denoted $P_{x_i,u}$ is the corresponding subpath of $P_{v,u}$.}
        \label{fig: hourglass process final hourglass}
    \end{figure}
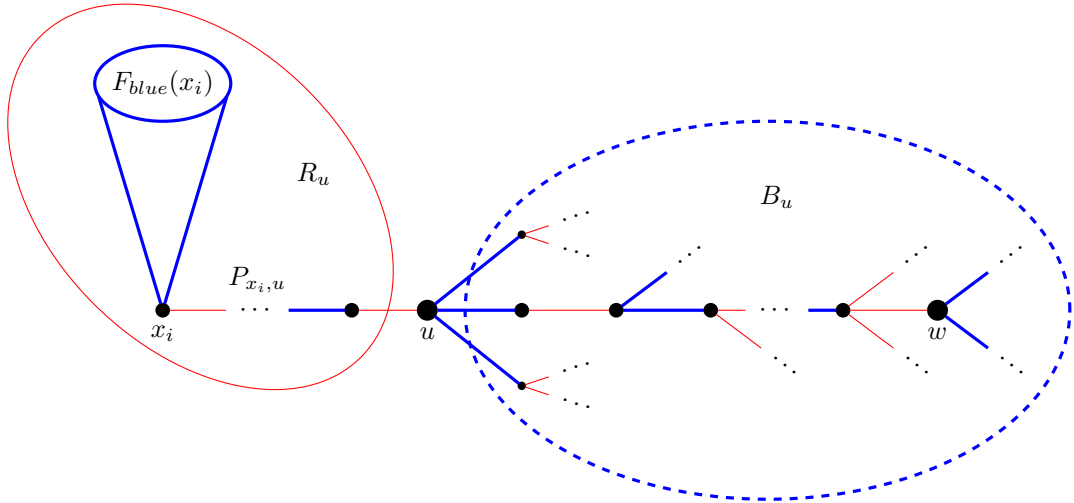

    By (\ref{eq: u promising}) and (\ref{eq: graph on path good}) we have,
    \begin{equation} \label{eq: hourglass process success prob}
        \mathbb{P}[\hyperref[process: hourglass]{Hourglass\_Process} \ succeeds] \geq (1 + o(1))\frac{\phi(1 - e^{-\phi})}{2e}\frac{\lambda_n}{n^{1/3}} \,.
    \end{equation}
    Let us now consider the expected number of vertices used during a run of \textit{Hourglass\_Process}~\ref{process: hourglass}. 

    For each $F_{\cdot}(\cdot)$ we have by 
    Corollary~\ref{cor: expected component size} that,
    \begin{equation*}
        \mathbb{E}[|V(F_{\cdot}(\cdot))|] = \frac{2}{pn'|\lambda_n'|(n')^{-1/3}}(1 + o(1)) = \frac{(n')^{1/3}}{\lambda_n'}(1 + o(1)) = \frac{n^{1/3}}{2 \lambda_n}(1 + o(1)) \,.
    \end{equation*}

    As we construct $F_{red}(v)$ with probability one and construct $n^{1/3} / \lambda_n$ 
    many other such $F_{\cdot}(\cdot)$ with probability 
    $(1 + o(1)) \frac{\phi}{2e} \frac{\lambda_n}{n^{1/3}}$, the expected number of 
    vertices used in \hyperref[process: hourglass]{Hourglass\_Process} is at most,

    \begin{equation} \label{eq: HP number of vertices used}
        \frac{n^{1/3}}{2 \lambda_n} (1 + o(1)) + \frac{n^{1/3}}{\lambda_n} \frac{n^{1/3}}{2 \lambda_n} \frac{\phi}{2e} \frac{\lambda_n}{n^{1/3}} (1 + o(1)) = \frac{n^{1/3}}{\lambda_n} \left[1 + \frac{\phi}{2e}\right] (1 + o(1)) = \Theta\left(\frac{n^{1/3}}{\lambda_n}\right) \,.
    \end{equation}

   We now define \hyperref[process: batched hourglass]{Batched\_Hourglass\_Process}, 
   which runs \hyperref[process: hourglass]{Hourglass\_Process} 
   multiple times to increase the probability of finding an hourglass of the desired size to 
   at least $1/2$ while directly controlling the total 
   number of vertices used.
   
   We will ensure that the number of vertices used in total is at most 
   the size of $S$. Thus Observation~\ref{ob: hourglass process} \ref{hourglass process: ob 1} implies
   each run is independent, as before each run, no information was revealed between vertices in $V'$. In what follows $\kappa$ and $\rho$ are positive constants to be determined later. 

   \begin{process}{\textbf{\textit{Batched\_Hourglass\_Process}}$(V', S)$} \label{process: batched hourglass}
   	\mbox{}
	\begin{itemize}
		\item If at any point more than $\rho n^{2/3}  / \lambda_n^2$ vertices have been used (cummulative): report \textit{fail}
	\end{itemize}
        \begin{enumerate}
            \item Run \hyperref[process: hourglass]{Hourglass\_Process} $\kappa n^{1/3}/\lambda_n$ times:
                \begin{enumerate}[label=\alph*., ref=\theenumi\alph*]
                    \item If \hyperref[process: hourglass]{Hourglass\_Process} returned \textit{success}: return \textit{success}
                \end{enumerate}
            \item If all runs of \hyperref[process: hourglass]{Hourglass\_Process} returned \textit{fail}: return \textit{fail}
        \end{enumerate}
   \end{process}
   
   As \hyperref[process: batched hourglass]{Batched\_Hourglass\_Process} terminates if more than $\rho n^{2/3}  / \lambda_n^2$ many vertices were used, we can always replenish $V'$. Moreover, each run of \hyperref[process: batched hourglass]{Batched\_Hourglass\_Process}$(V', S)$ will be on an unused set of vertices of size $n'$.

   We will prove that \hyperref[process: batched hourglass]{Batched\_Hourglass\_Process} finds a desired hourglass with probability at least $1/2$.
   
    Let $\kappa = \frac{4e}{\phi(1 - e^{-\phi})}$. If we run \hyperref[process: hourglass]{Hourglass\_Process}  $\kappa n^{1/3} / \lambda_n$ times, then by (\ref{eq: hourglass process success prob}) the probability that all runs \textit{fail} is at most,
    \begin{equation*}
        \left( 1 - (1 + o(1)) \frac{\phi(1 - e^{-\phi})}{2e} \frac{\lambda_n}{n^{1/3}} \right)^{\frac{4e}{\phi(1-e^{-\phi})} \frac{n^{1/3}}{\lambda_n}} = e^{-2(1 + o(1))} = e^{-2} + o(1) \,.
    \end{equation*}
    By (\ref{eq: HP number of vertices used}), the expected number of vertices used in $\kappa n^{1/3} / \lambda_n$ many runs of  \hyperref[process: hourglass]{Hourglass\_Process}  is at most $\Theta(n^{2/3} / \lambda_n^2)$.

    Therefore, by Markov's inequality there exists a constant $\rho > 0$ such that with probability at least $\frac{2}{3}$ we use less than $\rho n^{2/3} / \lambda_n^2$ vertices.
    Hence, by the union bound we have,
    \begin{equation*}
        \mathbb{P}[Batched\_Hourglass\_Process \ succeeds] \geq 1 - e^{-2} - \frac{1}{3} - o(1) > \frac{1}{2}, \qquad  \text{for large enough } n\,.
    \end{equation*}

    Recall from the statement of Lemma~\ref{lma: hourglass}, we wish to 
    find $\Theta(\lambda_n^3)$ vertex disjoint hourglasses. To do this, we run \hyperref[process: batched hourglass]{Batched\_Hourglass\_Process}
    $\lambda_n^3 / \rho$ times. We start with an arbitrary set $S$ of $\lambda_n n^{2/3}$ 
    vertices and $V' = V \setminus S$. Each run of \hyperref[process: batched hourglass]{Batched\_Hourglass\_Process}
    will replace the used vertices in $V'$ with vertices in $S$ 
    (this is handled implicitly by \hyperref[process: hourglass]{Hourglass\_Process}).
    Thus, each run is independent. This is possible since the total number of vertices used 
    is at most $\lambda_n^3 / \rho \cdot \rho n^{2/3} / \lambda_n^2 = \lambda_n n^{2/3}$. 
    As $|S| = \lambda_n n^{2/3}$, we never run out of ``reserve" vertices.

    Let $X$ denotes the number of hourglasses found during the $\lambda_n^3 / \rho$ runs 
    of \\ \hyperref[process: batched hourglass]{Batched\_Hourglass\_Process}. So, $X$ 
    stochastically dominates the binomial distribution \\
    $\mathcal{BIN}\left( \lambda_n^3 / \rho, 1/2 \right)$.
    Therefore by \hyperref[Chernoff Bound]{Binomial Chernoff Bounds}, w.h.p. 
    $X \geq \frac{1}{4 \rho} \lambda_n^3 = \Theta(\lambda_n^3)$ as required.
\end{proof}

\subsubsection{Information Revealed}

\begin{definition} \label{def: small hourglass}
    Call a hourglass produced by \hyperref[process: hourglass]{Hourglass\_Process} in Lemma~\ref{lma: hourglass} ``small". 
\end{definition}

Let $H = (u, R_u, B_u)$ be a small hourglass. It follows from 
steps~\ref{Hourglass proc: step check large tree} 
and \ref{Hourglass proc: step success} of \hyperref[process: hourglass]{Hourglass\_Process} that $H$ is a tree; note this does not mean the subgraph induced by $V(H)$ 
in the underlying graph is a tree. So, for each vertex $x \in H$, there is a unique alternating path $P_{u, x} \subseteq H$ 
from the center vertex $u$ to $x$. We define $\psi_{H} : V(H) \setminus \{u\} \rightarrow \{red, blue\}$, 
such that $\psi_H(x)$ denotes the color required for an edge to extend the alternating path $P_{u,x}$. 

We now specify what has been exposed about edges between two small hourglasses in the following lemma,

\begin{lemma} \label{lma: info between hourglasses}
    Let $H_u$ and $H_w$ be hourglasses constructed by seperate runs of \\
    \hyperref[process: batched hourglass]{Batched\_Hourglass\_Process}.
    Let $z \in H_u$ and $y \in H_w$.
    If either of $X^{red}_{zy}$ or $X^{blue}_{zy}$ has been revealed it equates to $0$.
\end{lemma}

\begin{proof}
    Assume without loss of generality that $H_u$ was constructed before $H_w$. 
This implies that before constructing $H_w$, $V(H_u)$ was discarded. 
Therefore, the only point in which the random variables associated with the
vertex pair $z, y$ are potentially revealed is during the construction of $H_u$. 

First consider the search \hyperref[process: construct G tilde]{\textit{Search\_Neighborhood}($v, red$)} performed in 
step~\ref{Hourglass proc: step 2} of \hyperref[process: hourglass]{Hourglass\_Process}. Let $F_{red}(v)$ be the resulting graph.
As $y \notin F_{red}(v)$, it follows by Observation~\ref{ob: alt search process} regarding \\
\hyperref[process: construct G tilde]{\textit{Search\_Neighborhood}}
that if either of $X^{red}_{zy}$ or $X^{blue}_{zy}$ has been revealed it equates to $0$.

So, it remains to consider the potential exposure during the searches performed by \\
\hyperref[proc: trimmed search]{\textit{Trimmed\_Search\_Neighborhood}} in 
step~\ref{Hourglass proc: step 3a} of \hyperref[process: hourglass]{Hourglass\_Process}.
Let $x_i$ denote the vertex which caused \hyperref[process: hourglass]{Hourglass\_Process}($V', S$) to report \textit{success}
in step~\ref{Hourglass proc: step 3d}. 
Let $F_{\gamma_{x_i}}(x_i)$ be the resulting graph and let $P_{x_i, u}$ be the path as in 
Figure~\ref{fig: hourglass process final hourglass}.

Note that the runs of \hyperref[proc: trimmed search]{\textit{Trimmed\_Search\_Neighborhood}}
performed on vertices before $x_i$ on the path $P_{v,u}$ were discarded and not included in $H_u$. 
Moreover, as the search in step~\ref{Hourglass proc: step 3d} stopped after $x_i$, 
no new information was exposed regarding vertices after $x_i$ on the path $P_{v,u}$. 
Therefore, the only new information revealed is regarding vertices in 
$F_{\gamma_{x_i}}(x_i)$.

The case where $z \neq x_i$ follows a similar argument as with the construction of 
$F_{red}(v)$. By Observation~\ref{ob: alt search process}, if either of $X^{red}_{zy}$ or 
$X^{blue}_{zy}$ has been revealed it 
equates to $0$.

It remains to consider the case where $z = x_i$, let $col \neq \gamma_{x_i}$. 
First we consider exposure regarding $X^{col}_{zy}$. It follows by 
Observation~\ref{ob: hourglass process}, that if $X^{col}_{x_iy}$ was revealed during 
the construction of $F_{red}(v)$, it equates to 0.
Furthermore, by Observation~\ref{ob: alt search process}, $X^{col}_{x_iy}$ is not 
revealed during the construction of $F_{\gamma_{x_i}}(x_i)$. 

We now consider the exposure regarding $X^{\gamma_{x_i}}_{x_iy}$. Regarding the construction of 
$F_{red}(v)$, it follows by Observation~\ref{ob: hourglass process} that $X^{\gamma_{x_i}}_{x_iy}$ was not revealed. 
During the construction of $F_{\gamma_{x_i}}(x_i)$ it follows by Observation~\ref{ob: alt search process} that
if $X^{\gamma_{x_i}}_{x_iy}$ was revealed it equates to $0$. Note that during the 
\hyperref[proc: trimming alteration]{Trimming Alteration}, it may have been decided to not reveal
$X^{\gamma_{x_i}}_{x_iy}$.
\end{proof}

\subsection{An Alternating Bicycle is Formed}
 
In this subsection we prove Theorem~\ref{thm: scaling window upper bound 2} which we restate below. 

\begin{theorem*}[\ref{thm: scaling window upper bound 2}]
    For $\lambda_n \xrightarrow[]{} \infty$ arbitrarily slowly with $\lambda_n \ll n^{1/3}$, $\mathcal{G}(n, p = 2(1 + \lambda_n n^{-1/3})/n, 2)$ is w.h.p. adaptably 2-colorable.
\end{theorem*}

To prove Theorem~\ref{thm: scaling window upper bound 2} we utilize the ``sprinkling" method describing in 
Section~\ref{sec: sprinkling method}. Particularly, in the subcritical regime, the small hourglasses in 
$\mathcal{G}(n,2(1 - \lambda_nn^{-1/3})/n,2)$ from Lemma~\ref{lma: hourglass} will serve as 
"building blocks" for forming a proper alternating odd bicycle in 
$\mathcal{G}(n,2(1 + \lambda_nn^{-1/3})/n,2)$ after ``sprinkling" the additional edges. 
In fact, we will view these small hourglasses as vertices in an auxiliary graph. We will add $red$ and 
$blue$ edges to this graph as follows: if among the ``sprinkled" edges we introduce an alternating path 
between the centre vertices of two hourglasses beginning and ending with the same color, we add an edge of 
that color between the respective vertices in the auxiliary graph. We will show that this reduces the 
problem to finding an alternating bicycle in the auxiliary graph. Furthermore, we couple the distribution of 
the auxiliary graph to $\mathcal{G}(n,p,2)$ for $p > 2/n$, allowing us to use our results regarding the 
supercritical regime from Section~\ref{sec: supercrit regime}; particularly, 
Corollary~\ref{cor: disjoint cycle bicycle}.

We note that this approach differs from that in \cite{2-sat-critical-window} and their method does not naturally cary over to the case of adaptable colorings, due to the differences between the edge color model and the directed graph model highlighted in the beginning of Section~\ref{sec: supercritical improvement}. However, their approach did serve as inspiration. 

\begin{proof}[Proof of Theorem~\ref{thm: scaling window upper bound 2}]
    We start with $G_1 = \mathcal{G}(n, p = \frac{2(1 - (M/2 - 1)\lambda_n n^{-1/3})}{n}, 2)$ 
    for some large enough constant $M$ to be determined shortly. 
    By Lemma~\ref{lma: hourglass}, there exists a constant $\phi > 0$ such that w.h.p. $G_1$ has at 
    least $\phi (M/2 - 1)^3 \lambda_n^3$ vertex disjoint hourglasses with 
    respective components of size at least $n^{2/3}/\lambda_n^2$; 
    denote the set of them by $\mathcal{H}$. Choose $M$ to be large enough such that $M(M/2 - 1)^3 > 72 / \phi$,
    the reason behind this will be made clear later in the proof. 
    We then consider increasing $p$ 
    to $\frac{2(1 + \lambda_n n^{-1/3})}{n}$. We do this by ``sprinkling" edges with 
    probability $\sim M\lambda_n/n^{4/3}$. We claim that upon 
    "sprinkling" the edges, a subset of the hourglasses in $\mathcal{H}$ and ``sprinkled" 
    edges will form a proper alternating odd bicycle.

   For each hourglass $H_u = (u, R_u, B_u) \in \mathcal{H}$, we say 
   $H_u$ is $(c_r, c_b)$-dominated for $c_r, c_b \in \{red, blue\}$ if 
   at least half of the vertices in $R_u$ are assigned $c_r$ under $\psi_{H_u}$ and similarly, at least
   half of the vertices $B_u$ are assigned $c_b$ under $\psi_{H_u}$. That is,
   \begin{equation*}
   	|\psi_{H_u}^{-1}(c_r) \cap R_u| \geq \frac{1}{2} |R_u| \qquad and \qquad  |\psi_{H_u}^{-1}(c_b) \cap B_u| \geq \frac{1}{2} |B_u| \,.
   \end{equation*}
   Note that every hourglass is $(c_r, c_b)$-dominated for some $c_r, c_b \in \{red, blue\}$;
   however, this choice of colors may not be unique, in this case choose $c_r$ or $c_b$ to be $red$. 
   We now partition $\mathcal{H}$ as follows: for $c_r, c_b \in \{red, blue\}$ let
    \begin{equation*}
    	\mathcal{H}_{(c_r, c_b)} \coloneqq \{H \in \mathcal{H} \mid H  \ is \ (c_r, c_b)-dominated\} \,.
    \end{equation*}

   By the Pigeonhole Principle, there exists $(c_r, c_b) \in \{red, blue\}$ such that $|\mathcal{H}_{(c_r, c_b)}| \geq \frac{1}{4} |\mathcal{H}| \geq \frac{1}{4}\phi (M/2 - 1) \lambda_n^3$. 
   
	We will now prove that upon ``sprinkling" the edges, an alternating bicycle is formed 
    among the hourglasses and these sprinkled edges. To do so, we construct a graph $F$ 
    with vertex set $\mathcal{H}_{(c_r, c_b)}$ and edge coloring $c_F : E(F) \rightarrow \{red, blue\}$. 
    For hourglasses $H_u, H_v \in \mathcal{H}_{(c_r,c_b)}$, we connect $H_u$ to $H_v$ in 
    $F$ by a $red$ edge if one of the ``sprinkled" edges $e = u'v'$ has $u' \in R_u$, 
    $v' \in R_v$ and $\psi_{H_u}(u') = \psi_{H_v}(v') = c(e)$; i.e. $e$ has the color required to extend both the alternating path from $u$ to $u'$ in $H_u$ and the alternating path from $v$ to $v'$ in $H_v$. Denote this indicator 
    variable for this event by $E_{uv}^{red}$. Similarly, we connect $H_u$ to $H_v$ by a 
    $blue$ edge if one of the ``sprinkled" edges $e = u'v'$ has $u' \in B_u$ and 
    $v' \in B_v$ such that $\psi_{H_u}(u') = \psi_{H_v}(v') = c(e)$. Denote this event by $E_{uv}^{blue}$. 
    This is seen in Figure~\ref{Hourglass: colored graph}.
    
    We proceed by defining the sets $A^{red}_{uv}$ and $A^{blue}_{uv}$. The set $A^{red}_{uv}$ denotes the pairs of vertices which can produce an alternating path from $u$ to $v$ beginning and ending with $red$. Similarly, the set $A^{blue}_{uv}$ denotes the pairs of vertices which can produce an alternating path from $u$ to $v$ beginning and ending with $blue$. Formally, 
    \begin{equation*}
        A_{uv}^{red} \coloneqq \left\{(u', v') \in R_u \times R_v \mid \psi_{H_u}(u') = \psi_{H_v}(v')\right\} \,,
    \end{equation*}
    and 
    \begin{equation*}
    	A_{uv}^{blue} \coloneqq \left\{(u', v') \in B_u \times B_v \mid \psi_{H_u}(u') = \psi_{H_v}(v')\right\} \,.
    \end{equation*}
    By construction of $\mathcal{H}_{(c_r,c_b)}$, $|A_{uv}^{red}| \geq \frac{1}{4}|R_u||R_v| \geq \frac{n^{4/3}}{4 \lambda_n^4}$ and $|A_{uv}^{blue}| \geq \frac{1}{4}|B_u||B_v| \geq \frac{n^{4/3}}{4 \lambda_n^4}$.

    \begin{figure}[h]
    \centering
        \begin{tikzpicture}[node distance={15mm}, main/.style = {draw, circle, thick, fill=black} ]
            

            \node[main, scale = 0.75] (u) at (0, 0) {};
            \node[main, draw=none, fill=none] (u_name) at (0,-0.4) {$u$}; 

            \draw[red] (-2,0) ellipse (1cm and 2cm);
            \node[main, draw=none, fill=none] (Ru) at (-3,1.8) {$R_u$}; 

            \draw[blue, very thick] (2,0) ellipse (1cm and 2cm);
            \node[main, draw=none, fill=none] (Bu) at (1,1.8) {$B_u$}; 

            \node[main, draw=none, fill=none] (B_u_up) at (1.74,2) {}; 
            \node[main, draw=none, fill=none] (B_u_down) at (1.74,-2) {}; 

            \node[main, scale = 0.5] (u_r) at (-2,-1.2) {};
            \node[main, scale = 0.5] (u_r_temp) at (-1.6,-0.4) {};

            \draw[red] (u) -- (u_r_temp);
            \draw[blue, very thick] (u_r_temp) -- (u_r);

            \node[main, scale = 0.5] (u_b) at (2,1.2) {};
            \node[main, scale = 0.5] (u_b_temp) at (1.6,0.4) {};

            \draw[blue, very thick] (u) -- (u_b_temp);
            \draw[red] (u_b_temp) -- (u_b);


            \node[main, scale = 0.75] (v) at (7, 0) {};
            \node[main, draw=none, fill=none] (v_name) at (7,-0.4) {$v$}; 

            \draw[red] (5,0) ellipse (1cm and 2cm);
            \node[main, draw=none, fill=none] (Ru) at (6,1.8) {$R_v$}; 

            \draw[blue, very thick] (9,0) ellipse (1cm and 2cm);
            \node[main, draw=none, fill=none] (Bu) at (10,1.8) {$B_v$}; 

            \node[main, draw=none, fill=none] (R_v_up) at (5.26,2) {}; 
            \node[main, draw=none, fill=none] (R_v_down) at (5.26,-2) {};

            \node[main, scale = 0.5] (v_b_2) at (9,1.2) {};
            \node[main, scale = 0.5] (v_b_2_temp) at (8.6,0.4) {};

            \draw[blue, very thick] (v) -- (v_b_2_temp);
            \draw[red] (v_b_2_temp) -- (v_b_2);

            \node[main, scale = 0.5] (v_r_2) at (5,-1.2) {};
            \node[main, scale = 0.5] (v_r_2_temp) at (5.6,-0.4) {};

            \draw[red] (v) -- (v_r_2_temp);
            \draw[blue, very thick] (v_r_2_temp) -- (v_r_2);


            \draw[red, dashed] (u_r) to[out=-30,in=-150] (v_r_2);

           \draw[blue, very thick, dashed] (v_b_2) to[out=-210,in=-330]  (u_b);
                
        \end{tikzpicture}
        \caption{
        Ways in which an edge can be introduced between $H_u$ and $H_v$ upon ``sprinkling" the edges.
        The ``sprinkled" edges correspond to the dashed edges. If a ``sprinkled" edge is introduced between $R_u$ and $R_v$, then a $red$ edge is introduced between $H_u$ and $H_v$ in $F$, this corresponds to the $red$ dashed edge. Similarly, if a "sprinkled" edge is introduced between $B_u$ and $B_v$, then a $blue$ edge is introduced between $H_u$ and $H_v$ in $F$, this corresponds to the $blue$ dashed edge.}
        \label{Hourglass: colored graph}
    \end{figure}
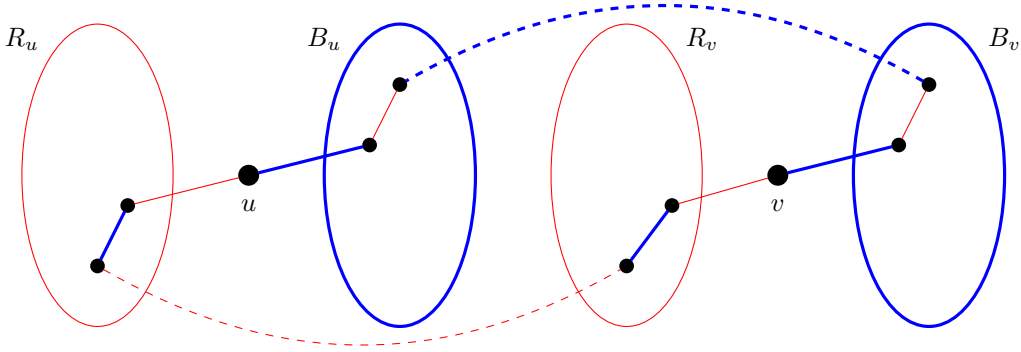

Let $H_u, H_v \in \mathcal{H}_{(c_r,c_b)}$ and let $A \coloneqq A_{uv}^{col}$, for some $col \in \{red, blue\}$. 

From above we know that $|A| \geq \frac{n^{4/3}}{4 \lambda_n^4}$. We shall assume
this holds with equality, by simply considering a subset of the pairs. 
We want to determine the probability that among the ``sprinkled" edges, we introduce an edge of the right color between a pair of vertices in $A$, and hence an edge in $F$. Note that, such edges may already exist; however, due to the color requirement, this may decrease the probability of introducing such an edge. 
Despite this, we will show that upon 
``sprinkling" edges with probability $\sim M\lambda_n/n^{4/3}$ the probability we 
introduce an edge 
$H_u H_v$ in $F$ of color $col$ is at least $\frac{M}{8\lambda_n^3}$.

First, we need to consider the pairs of vertices in $A$ for which there exists an 
edge in $G_1$; denote this set by $A_{exists}$.
It follows by Lemma~\ref{lma: info between hourglasses} and Lemma~\ref{lma: prob change} that 
$|A_{exists}|$ is stochastically dominated by the binomial distribution with 
probability $p = \frac{2(1 - \lambda_nn^{-1/3})}{n}$ and $\frac{n^{4/3}}{4\lambda_n^2}$ 
trials. 

    So, $\mathbb{E}[|A_{exists}|] \leq \frac{n^{4/3}}{4\lambda_n^2} \frac{2(1 - \lambda_nn^{-1/3})}{n} = \frac{ n^{1/3}}{2\lambda_n^4} (1 - o(1))$. 
    Thus, by the \hyperref[Chernoff Bound]{Binomial Chernoff Bound}, we have
    \begin{equation} \label{eq: hourglass free pairs chernoff}
        \mathbb{P}\left[ |A_{exists}| - \frac{ n^{1/3}}{2\lambda_n^4} (1 - o(1)) > \frac{n^{1/6}}{\lambda_n}\right] \leq 2e^{-\frac{2 \lambda_n^2}{3} (1 + o(1))} = o(1) \,.
    \end{equation}
    Let $A_{free} \coloneqq A \setminus A_{exists}$ be the set of such pairs of vertices which are ``free". By (\ref{eq: hourglass free pairs chernoff}) and as $|A| = \frac{n^{4/3}}{4\lambda_n^4}$ we have w.h.p. that,
    \begin{equation} \label{eq: size of Y}
        |A_{free}| \geq \frac{ n^{4/3}}{4\lambda_n^4} (1 - o(1)) - \frac{ n^{1/3}}{2\lambda_n^4} (1 - o(1)) - \frac{n^{1/6}}{\lambda_n} = \frac{n^{4/3}}{4 \lambda_n^4}(1 - o (1)) \,,
    \end{equation}
    as $\lambda_n^3 / n^{7/6} \xrightarrow[]{} 0$ since $\lambda_n \ll n^{1/3}$. 
    
    We have shown 
    $|A_{exists}| > \frac{ n^{1/3}}{2\lambda_n^4} (1 - o(1)) - \frac{n^{1/6}}{\lambda_n}$ 
    with probability exponentially small in $\lambda_n$. As there are 
    $\Theta\left({\lambda_n^3 \choose 2}\right) = \Theta(\lambda_n^6)$ 
    pairs of hourglasses in $\mathcal{H}_{(c_r,c_b)}$, it follows by the union bound that in $G_1$ w.h.p. (\ref{eq: size of Y}) holds 
    for every pair of hourglasses in $\mathcal{H}_{(c_r, c_b)}$. Therefore, in what follows we shall assume before sprinkling (\ref{eq: size of Y}) holds for every pair of hourglasses. So $E^{col}_{uv}$ holds if at least one of the ``sprinkled" edges joins a pair in $A_{free}$ and is of the right color. Thus,
    \begin{equation} \label{eq: edge prob in F}
        \mathbb{P}[E^{col}_{uv}] 
        \geq 1 - \left( 1 - \frac{M\lambda_n}{2n^{4/3}} \right)^{ \frac{n^{4/3}}{4 \lambda_n^4}(1 - o (1))}
        = \frac{M}{8\lambda_n^3}(1 + o(1)) > \frac{M}{9\lambda_n^3}\,.
    \end{equation}
    
    We recall the construction of $F$. The vertex set of $F$ consists of the hourglasses in $\mathcal{H}_{(c_r, c_b)}$. Let $N = \frac{1}{4} \phi (M/2 - 1)^3 \lambda_n^3$ denote the number of vertices in $F$. For each pair of hourglasses $H_u, H_v \in \mathcal{H}_{(c_r, c_b)}$ we introduce a $red$ edge between $H_u$ and $H_v$ in $F$ if $E^{red}_{uv}$ occurs; similarly, we introduce a $blue$ edge between $H_u$ and $H_v$ in $F$ if $E^{blue}_{uv}$ occurs.  It's worth noting that $F$ may contain parallel edges of different colors; particularly, the event of a $red$ edge between $H_u$ and $H_v$ is independent with the event of a $blue$ edge between $H_u$ and $H_v$.
    Thus, we have shown that each colored edge in $F$ is included independently with probability at least $\frac{M}{9\lambda_n^3}$. Furthermore, a simple coupling argument shows that we can take the probability of each edge to be precisely $\frac{M}{9\lambda_n^3}$, as this can only reduce the probability of $F$ containing a specified subgraph. 
    
    Due to the independence of $red$ and $blue$ edges occurring in $F$, its not hard to see that the probability that $F$ contains a subgraph is stochastically dominated the probability $\mathcal{G}(N, \frac{M}{9\lambda_n^{3}}, 2)$ contains such a subgraph, this is handled in the following claim.
    
    \begin{claim}
    	One can remove edges from $F$ so that the resulting graph is distributed precisely as $\mathcal{G}(N, \frac{M}{9\lambda_n^{3}}, 2)$. \\
	
	Consider the following process on $F$, 
    for each unordered pair of vertices $H_u,H_v$ in $F$ define the following random variable,
    \begin{equation} \label{eq: Y F tilde def}
        Y_{uv} = \begin{cases} 
            1 & with \ probability \ 1/2 \\
            0 & otherwise 
        \end{cases} \,.
    \end{equation} 
    Intuitively, the random variable
    $Y_{uv}$ determines which color to be considered between the vertices $H_u, H_v \in F$.

    Formally, we construct the graph $\widetilde{F}$ as follows: for each pair of vertices 
    $H_u,H_v$ with $Y_{uv} = 1$ and $E^{red}_{uv} = 1$ add a $red$ $H_uH_v$ edge in $\widetilde{F}$. Similarly, 
    for each pair of vertices $H_u, H_v$ with $Y_{uv} = 0$ and $E^{blue}_{uv} = 1$ 
    add a $blue$ $H_uH_v$ edge in $\widetilde{F}$. 

    Clearly, $\widetilde{F}$ no longer contains parallel edges. We claim that $\widetilde{F}$ is sampled
    from $\mathcal{G}(N,\frac{M}{9\lambda_n^3}, 2)$. Observe,
    \begin{equation} \label{eq: red edge in F tilde}
        \mathbb{P}[red \ edge \ between \ H_u \ and \ H_v \ in \ \widetilde{F}] = \mathbb{P}[Y_{uv} = 1 \wedge E^{red}_{uv} = 1] = \frac{1}{2} \frac{M}{9\lambda_n^3} \,.
    \end{equation}
    Similarly,
    \begin{equation} \label{eq: blue edge in F tilde}
        \mathbb{P}[blue \ edge \ between \ H_u \ and \ H_v \ in \ \widetilde{F}] = \mathbb{P}[Y_{uv} = 0 \wedge E^{blue}_{uv} = 1] = \frac{1}{2} \frac{M}{9\lambda_n^3} \,.
    \end{equation}
    It's not hard to see from (\ref{eq: red edge in F tilde}) and (\ref{eq: blue edge in F tilde}) 
    that edges exist in $\widetilde{F}$ with probability $\frac{M}{9\lambda_n^3}$ and by 
    (\ref{eq: Y F tilde def}) are colored $red$ or $blue$ uniformly at random. 
    Thus, $\widetilde{F} \sim \mathcal{G}(N, \frac{M}{8\lambda_n^3}, 2)$ as required.

    \end{claim}

    Let $G$ denote the graph obtained from $G_1$ by ``sprinkling" the edges. 
    As $M(M/2 - 1)^3 > 72 / \phi$ and $N = \frac{1}{4} \phi (M/2 - 1)^3 \lambda_n^3$ it follows that 
    $\frac{M}{9\lambda_n^3} > \frac{2}{N}$. Thus by Corollary~\ref{cor: disjoint cycle bicycle}, w.h.p. $F$ contains a 
    proper alternating odd bicycle with disjoint cycles. 
    
    Let $P = H_1, H_2, \dots, H_s$ along with $H_1 H_i$ and $H_jH_s$ for $i < j$ 
    denote a proper alternating odd bicycle in $F$, as seen in Figure~\ref{fig: alternating bicycle in F}. 
    We will show that among the ``sprinkled" edges as well as the hourglasses 
    $H_1, \dots, H_s$ there exists a proper alternating odd bicycle in $G$.
    This can be seen in Figure~\ref{Hourglass: bicycle}.
    
    Recall, edges in $F$ correspond to alternating paths between the centres of the hourglasses starting and ending with the color of the edge. However, as the degree of $H_i$ and $H_j$ is three in the alternating bicycle, the alternating paths corresponding to the edges $H_1H_i$ and $H_{i-1}H_i$ may not be edge disjoint, similarly with respect to $H_j$.
    
    \begin{figure}[h]
    \centering
        \begin{tikzpicture}[node distance={10mm}, main/.style = {draw, circle, fill=black}] 
            \begin{scope}[every node/.style={circle, thin, draw, minimum size=1mm}]
                \node[main,scale = 0.5] (v1) at (0,0) {};
                \node[main, draw=none, fill=none] (v1_name) at (0,-0.4) {$H_1$}; 

                \node[main,scale = 0.5] (v2) at (1.5,0) {};
                \node[main, draw=none, fill=none] (v2') at (2.25,0) {$\cdots$};
                \node[main, draw=none, fill=none] (C1) at (2.25,0.6) {$C_1$};
                \node[main,scale = 0.5] (v3) at (3,0) {};
                \node[main,scale = 0.5] (v4) at (4.5,0) {};
                \node[main, draw=none, fill=none] (vi_name) at (4.5,-0.4) {$H_i$}; 
                \node[main,scale = 0.5] (v5) at (6,0) {};
                \node[main, draw=none, fill=none] (v5') at (6.75,0) {$\cdots$};
                \node[main, draw=none, fill=none] (P) at (6.75,0.4) {$P$};
                \node[main,scale = 0.5] (v6) at (7.5,0) {};
                \node[main,scale = 0.5] (v7) at (9,0) {};
                \node[main, draw=none, fill=none] (vi_name) at (9,-0.4) {$H_j$}; 
                \node[main,scale = 0.5] (v8) at (10.5,0) {};
                \node[main, draw=none, fill=none] (v8') at (11.25,0) {$\cdots$};
                \node[main, draw=none, fill=none] (C2) at (11.25,0.6) {$C_2$};
                \node[main,scale = 0.5] (v9) at (12,0) {};
                \node[main,scale = 0.5] (v10) at (13.5,0) {};
                \node[main, draw=none, fill=none] (vs_name) at (13.5,-0.4) {$H_s$}; 

                \draw[red] (v1) -- (v2);
                \draw[blue, very thick] (v2) -- (v2');
                \draw[red] (v2') -- (v3);
                \draw[blue, very thick] (v3) -- (v4);

                \draw[red] (v4) -- (v5);
                \draw[blue, very thick] (v5) -- (v5');
                \draw[red] (v5') -- (v6);
                \draw[blue, very thick] (v6) -- (v7);
                
                \draw[red] (v8) -- (v7);
                \draw[blue, very thick] (v8') -- (v8);
                \draw[red] (v9) -- (v8');
                \draw[blue, very thick] (v10) -- (v9);

                \draw[blue, very thick] (v1) to[out=60,in=120] (v4);
                \draw[red] (v7) to[out=60,in=120] (v10);
                
            \end{scope}
    
        \end{tikzpicture}
        \caption{The proper alternating odd bicycle in $F$ under the assumption that $c_F(H_1 H_2) = red$ and $s$ is odd.}
        \label{fig: alternating bicycle in F}
    \end{figure}
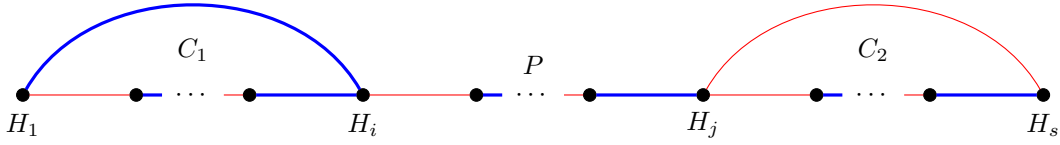

    \begin{figure}[h]
    \centering
        \begin{tikzpicture}[node distance={15mm}, main/.style = {draw, circle, thick, fill=black} ]
            

            \node[main, scale = 0.75] (H1) at (0, 0) {};
            \node[main, draw=none, fill=none] (u_name) at (-0.4,0) {$1$}; 

            \draw[red] (0,1) ellipse (1cm and 0.5cm);

            \draw[blue, very thick] (0,-1) ellipse (1cm and 0.5cm);
            
             \node[main, draw = none, fill = none] (H1_b1) at (-1.035, -0.9) {};
             \node[main, draw = none, fill = none] (H1_b2) at (1.035, -0.9) {};

             \draw[blue, very thick] (H1) -- (H1_b1); 
             \draw[blue, very thick] (H1) -- (H1_b2);
             
             \node[main, draw = none, fill = none] (H1_r1) at (-1.035, 0.9) {};
             \node[main, draw = none, fill = none] (H1_r2) at (1.035, 0.9) {};

             \draw[red] (H1) -- (H1_r1); 
             \draw[red] (H1) -- (H1_r2);
             
             
             \node[main, scale = 0.75] (H2) at (2.5, 0) {};
             \node[main, draw=none, fill=none] (u_name) at (2.1,0) {$2$}; 

             \draw[red] (2.5,1) ellipse (1cm and 0.5cm);

             \draw[blue, very thick] (2.5,-1) ellipse (1cm and 0.5cm);
            
              \node[main, draw = none, fill = none] (H2_b1) at (1.465, -0.9) {};
              \node[main, draw = none, fill = none] (H2_b2) at (3.535, -0.9) {};

              \draw[blue, very thick] (H2) -- (H2_b1); 
              \draw[blue, very thick] (H2) -- (H2_b2);
             
              \node[main, draw = none, fill = none] (H2_r1) at (1.465, 0.9) {};
              \node[main, draw = none, fill = none] (H2_r2) at (3.535, 0.9) {};

              \draw[red] (H2) -- (H2_r1); 
              \draw[red] (H2) -- (H2_r2);
              
             
             \node[main, scale = 0.75] (H3) at (5, 0) {};
             \node[main, draw=none, fill=none] (u_name) at (4.6,0) {$i$}; 

             \draw[red] (5,1) ellipse (1cm and 0.5cm);

             \draw[blue, very thick] (5,-1) ellipse (1cm and 0.5cm);
            
              \node[main, draw = none, fill = none] (H3_b1) at (3.965, -0.9) {};
              \node[main, draw = none, fill = none] (H3_b2) at (6.035, -0.9) {};

              \draw[blue, very thick] (H3) -- (H3_b1); 
              \draw[blue, very thick] (H3) -- (H3_b2);
             
              \node[main, draw = none, fill = none] (H3_r1) at (3.965, 0.9) {};
              \node[main, draw = none, fill = none] (H3_r2) at (6.035, 0.9) {};

              \draw[red] (H3) -- (H3_r1); 
              \draw[red] (H3) -- (H3_r2);

             
             \node[main, scale = 0.75] (H4) at (7.5, 0) {};
             \node[main, draw=none, fill=none] (u_name) at (7.1,0) {$j$}; 

             \draw[red] (7.5,1) ellipse (1cm and 0.5cm);

             \draw[blue, very thick] (7.5,-1) ellipse (1cm and 0.5cm);
            
              \node[main, draw = none, fill = none] (H4_b1) at (6.465, -0.9) {};
              \node[main, draw = none, fill = none] (H4_b2) at (8.535, -0.9) {};

              \draw[blue, very thick] (H4) -- (H4_b1); 
              \draw[blue, very thick] (H4) -- (H4_b2);
             
              \node[main, draw = none, fill = none] (H4_r1) at (6.465, 0.9) {};
              \node[main, draw = none, fill = none] (H4_r2) at (8.535, 0.9) {};

              \draw[red] (H4) -- (H4_r1); 
              \draw[red] (H4) -- (H4_r2);

             
             \node[main, scale = 0.75] (H5) at (10, 0) {};
             \node[main, draw=none, fill=none] (u_name) at (9.4,0) {$j+1$}; 

             \draw[red] (10,1) ellipse (1cm and 0.5cm);

             \draw[blue, very thick] (10,-1) ellipse (1cm and 0.5cm);
            
              \node[main, draw = none, fill = none] (H5_b1) at (8.965, -0.9) {};
              \node[main, draw = none, fill = none] (H5_b2) at (11.035, -0.9) {};

              \draw[blue, very thick] (H5) -- (H5_b1); 
              \draw[blue, very thick] (H5) -- (H5_b2);
             
              \node[main, draw = none, fill = none] (H5_r1) at (8.965, 0.9) {};
              \node[main, draw = none, fill = none] (H5_r2) at (11.035, 0.9) {};

              \draw[red] (H5) -- (H5_r1); 
              \draw[red] (H5) -- (H5_r2);
              
             
             \node[main, scale = 0.75] (H6) at (12.5, 0) {};
             \node[main, draw=none, fill=none] (u_name) at (12.1,0) {$s$}; 

             \draw[red] (12.5,1) ellipse (1cm and 0.5cm);

             \draw[blue, very thick] (12.5,-1) ellipse (1cm and 0.5cm);
            
              \node[main, draw = none, fill = none] (H6_b1) at (11.465, -0.9) {};
              \node[main, draw = none, fill = none] (H6_b2) at (13.535, -0.9) {};

              \draw[blue, very thick] (H6) -- (H6_b1); 
              \draw[blue, very thick] (H6) -- (H6_b2);
             
              \node[main, draw = none, fill = none] (H6_r1) at (11.465, 0.9) {};
              \node[main, draw = none, fill = none] (H6_r2) at (13.535, 0.9) {};

              \draw[red] (H6) -- (H6_r1); 
              \draw[red] (H6) -- (H6_r2);

              \node[main, draw=none, fill=none] (u_name) at (6.25,0) {$\cdots$}; 
              
              \node[main, draw=none, fill=none] (u_name) at (3.75,0) {$\cdots$}; 
              \node[main, draw=none, fill=none] (u_name) at (11.25,0) {$\cdots$}; 
              
              \node[main, scale = 0.25] (H1_r_1) at (0,1.3) {};
              \node[main, draw=none, fill=none, scale = 0.75] (H1_r_1_name) at (0.3,1.2) {$x^{red}_1$}; 
              \node[main, scale = 0.25] (H2_r_1) at (2.5, 1.3) {};
              \node[main, draw=none, fill=none, scale = 0.75] (H2_r_1_name) at (2.8,1.2) {$x^{red}_2$}; 
              
              \node[main, scale = 0.25] (H2_b_1) at (2.5,-1.3) {};
              \node[main, draw=none, fill=none,scale = 0.75] (H2_b_1_name) at (2.1,-1.2) {$x^{blue}_2$}; 
              \node[main, draw = none, fill = none, scale = 0.75] (H2_break) at (3.75, -1.8) {$\cdots$};
              \node[main, scale = 0.25] (H3_b_1) at (5, -1.3) {};
              \node[main, draw=none, fill=none, scale=0.75] (H3_b_1_name) at (4.6,-1.2) {$x^{blue}_i$}; 
              
              \node[main, scale = 0.25] (H3_r_1) at (5,1.3) {};
              \node[main, draw=none, fill=none, scale=0.75] (H3_r_1_name) at (5.3,1.2) {$x^{red}_i$}; 
              \node[main, draw=none, fill=none, scale = 0.75] (H3_break) at (6.25, 1.8) {$\cdots$};
            	
              \node[main, scale = 0.25] (H4_b_1) at (7.5, -1.3) {};
              \node[main, draw=none, fill=none, scale=0.75] (H4_b_1_name) at (7.1,-1.2) {$x^{blue}_j$}; 
              \node[main, draw=none, fill=none, scale = 0.75] (H4_break) at (6.25, -1.8) {$\cdots$};
              
              \node[main, scale = 0.25] (H4_r_1) at (7.5,1.3) {};
              \node[main, draw=none, fill=none, scale=0.75] (H4_r_1_name) at (7.8,1.2) {$x^{red}_j$}; 
              \node[main, draw=none, fill=none, scale = 0.75] (H6_break) at (11.25, -1.8) {$\cdots$};
              \node[main, scale=0.25] (H5_r_1) at (10,1.3) {};
              \node[main, draw=none, fill=none, scale=0.75] (H5_r_1_name) at (10.35,1.2) {$x^{red}_{j+1}$};
              \node[main, scale=0.25] (H5_b_1) at (10, -1.3) {};
              \node[main, draw=none, fill=none, scale=0.75] (H5_b_1_name) at (9.6,-1.2) {$x^{blue}_{j+1}$};
              \node[main, scale=0.25] (H6_b_1) at (12.5, -1.3) {};
              \node[main, draw=none, fill=none, scale=0.75] (H6_b_1_name) at (12.1,-1.2) {$x^{blue}_s$};
              
              \node[main, scale = 0.25] (H1_b_1) at (0,-1.3) {};
              \node[main, draw=none, fill=none, scale=0.75] (H1_b_1_name) at (-0.4,-1.2) {$x^{blue}_1$};
              
              \node[main, scale = 0.25] (H3_b_2) at (5.65, -1.1) {};
              \node[main, draw=none, fill=none, scale=0.75] (H3_b_2_name) at (5.4,-1.1) {$z$};
              
              \node[main, scale = 0.25] (H6_r_1) at (12.5, 1.3) {};
              \node[main, draw=none, fill=none, scale=0.75] (H6_r_1_name) at (12.8,1.2) {$x^{red}_s$};
              \node[main, scale = 0.25] (H4_r_2) at (6.85, 1.1) {};
              \node[main, draw=none, fill=none, scale=0.75] (H1_b_1_name) at (7.05,1.1) {$w$};
              
              \draw[blue, very thick] (H1_r_1) to[out=45,in=135] (H2_r_1);
              \draw[blue, very thick] (H2_b_1) to[out=-45,in=180] (H2_break);
              \draw[blue, very thick] (H2_break) to[out=0,in=-135] (H3_b_1);
              \draw[red] (H3_r_1) to[out=45,in=180] (H3_break);
              \draw[blue, very thick] (H4_break) to[out=0,in=-135] (H4_b_1);
              \draw[red] (H4_r_1) to[out=45, in=135] (H5_r_1);
 	     \draw[blue, very thick] (H5_b_1) to[out=-45, in=180] (H6_break);
         \draw[blue, very thick] (H6_break) to[out=0, in=-135] (H6_b_1);
	     
	     \draw[red] (H1_b_1) to[out=-45, in=-135] (H3_b_2);
	     \draw[blue, very thick] (H4_r_2) to [out=45, in=135] (H6_r_1);

	     \node[main, draw=none, fill=none, scale=0.5,rotate=90] (H1_r_temp) at (0,0.8) {$\cdots$};
	     \draw[red] (H1) -- (H1_r_temp);
	     \draw[red] (H1_r_temp) -- (H1_r_1);
	     
	     \node[main, draw=none, fill=none, scale=0.5,rotate=90] (H2_r_temp) at (2.5,0.8) {$\cdots$};
	     \draw[red] (H2) -- (H2_r_temp);
	     \draw[red] (H2_r_temp) -- (H2_r_1);
	     
	     \node[main, draw=none, fill=none, scale=0.5,rotate=90] (H1_b_temp) at (0,-0.8) {$\cdots$};
	     \draw[blue, very thick] (H1) -- (H1_b_temp);
	     \draw[blue, very thick] (H1_b_temp) -- (H1_b_1);
	     
	     \node[main, draw=none, fill=none, scale=0.5,rotate=90] (H2_b_temp) at (2.5,-0.8) {$\cdots$};
	     \draw[blue, very thick] (H2) -- (H2_b_temp);
	     \draw[red] (H2_b_temp) -- (H2_b_1);
	     
	     \node[main, draw=none, fill=none, scale=0.5,rotate=90] (H3_r_temp) at (5,0.8) {$\cdots$};
	     \draw[red] (H3) -- (H3_r_temp);
	     \draw[blue, very thick] (H3_r_temp) -- (H3_r_1);
	     
	     \node[main, draw=none, fill=none, scale=0.5,rotate=90] (H3_b_temp_1) at (5,-0.8) {$\cdots$};
	     \draw[blue, very thick] (H3) -- (H3_b_temp_1);
	     \draw[red] (H3_b_temp_1) -- (H3_b_1);
	     
	     \node[main, draw=none, fill=none, scale=0.5,rotate=-57] (H3_b_temp_2) at (5.45,-0.8) {$\cdots$};
	     \draw[blue, very thick] (H3) -- (H3_b_temp_2);
	     \draw[blue, very thick] (H3_b_temp_2) -- (H3_b_2);
	     
	     \node[main, draw=none, fill=none, scale=0.5,rotate=90] (H4_r_temp) at (7.5,0.8) {$\cdots$};
	     \draw[red] (H4) -- (H4_r_temp);
	     \draw[blue, very thick] (H4_r_temp) -- (H4_r_1);
	     
	     \node[main, draw=none, fill=none, scale=0.5,rotate=90] (H4_b_temp) at (7.5,-0.8) {$\cdots$};
	     \draw[blue, very thick] (H4) -- (H4_b_temp);
	     \draw[red] (H4_b_temp) -- (H4_b_1);
	     
	     \node[main, draw=none, fill=none, scale=0.5,rotate=123] (H4_r_temp_2) at (7.05,0.8) {$\cdots$};
	     \draw[red] (H4) -- (H4_r_temp_2);
	     \draw[red] (H4_r_temp_2) -- (H4_r_2);
	     
	     \node[main, draw=none, fill=none, scale=0.5,rotate=90] (H5_r_temp) at (10,0.8) {$\cdots$};
	     \draw[red] (H5) -- (H5_r_temp);
	     \draw[blue, very thick] (H5_r_temp) -- (H5_r_1);
	     
	     \node[main, draw=none, fill=none, scale=0.5,rotate=90] (H5_b_temp) at (10,-0.8) {$\cdots$};
	     \draw[blue, very thick] (H5) -- (H5_b_temp);
	     \draw[red] (H5_b_temp) -- (H5_b_1);
	     
	     \node[main, draw=none, fill=none, scale=0.5,rotate=90] (H6_r_temp) at (12.5,0.8) {$\cdots$};
	     \draw[red] (H6) -- (H6_r_temp);
	     \draw[red] (H6_r_temp) -- (H6_r_1);
	     
	     \node[main, draw=none, fill=none, scale=0.5,rotate=90] (H6_b_temp) at (12.5,-0.8) {$\cdots$};
	     \draw[blue, very thick] (H6) -- (H6_b_temp);
	     \draw[red] (H6_b_temp) -- (H6_b_1);
                
        \end{tikzpicture}
        \caption{A depiction of the structure of the hourglasses in $G$ and the ``sprinkled" edges from the proper alternating odd bicycle in $F$ given in 
        Figure~\ref{fig: alternating bicycle in F}. }
        \label{Hourglass: bicycle}
    \end{figure}
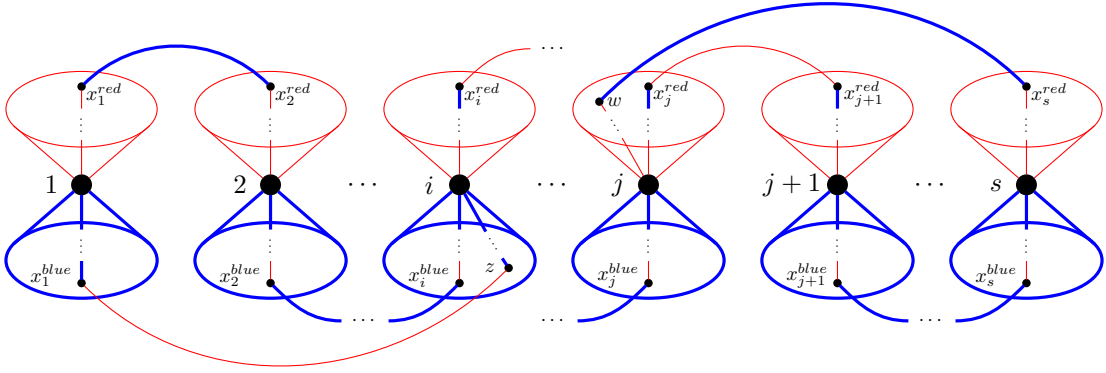

    Without loss of generality, assume the edge $H_1H_2$ is $red$. 
    By construction of $F$, there exists edges \\
    $x^{red}_1 x^{red}_2, x^{blue}_2 x^{blue}_3, \dots, x^{blue}_{i-1} x^{blue}_i$ among the ``sprinkled" edges which form the subpath $H_1, \dots, H_i$ of $P$ in $F$. That is, for $col \in \{red, blue\}$ $x^{col}_l$ is in the component of $H_l$ corresponding to the color $col$. Similarly, there exists a ``sprinkled" edge $x^{blue}_1 z$ for some $z$ in the $blue$ component $B_i$ of $H_i$. These edges are illustrated in Figure~\ref{Hourglass: bicycle}.

    For each $l = 1, \dots, i - 1$, the structure of the hourglass $H_l$, induces a natural alternating 
    path in $H_l$ from $x^{blue}_l$ to $x^{red}_l$ going through the centre vertex $l$, 
    call this path $P_l$. Similarly there exists an alternating path $P_i$ from $x^{blue}_i$ to $i$ in the $blue$ component $B_i$ of $H_i$. 
    It's not hard to see that the paths $P_l$, the path $P_i$, and the 
    ``sprinkled" edges form an alternating path from $x_1^{blue}$ to $i$.

    Moreover, there exists an alternating path $P_z$ from $z$ to $i$ in the $blue$ component $B_i$ of $H_i$. 
    Let $z' \in P_i \cap P_z$ be farthest from $i$, as seen in Figure~\ref{Hourglass: i vertex}. The vertex $z'$ will be one of the degree three vertices in the alternating bicycle in $G$.

    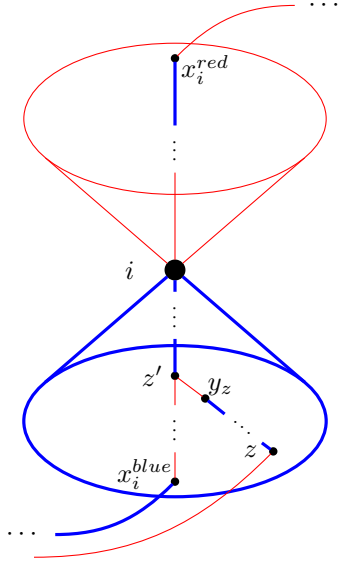
\begin{figure}[h]
    \centering
        \begin{tikzpicture}[node distance={15mm}, main/.style = {draw, circle, thick, fill=black} ]
             \node[main, scale = 0.75] (i) at (0, 0) {};
             \node[main, draw=none, fill=none] (i_name) at (-0.6,0) {$i$}; 

             \draw[red] (0,2) ellipse (2cm and 1cm);

             \draw[blue, very thick] (0,-2) ellipse (2cm and 1cm);
            
              \node[main, draw = none, fill = none] (i_b1) at (1.85, -1.6) {};
              \node[main, draw = none, fill = none] (i_b2) at (-1.85, -1.6) {};

              \draw[blue, very thick] (i) -- (i_b1); 
              \draw[blue, very thick] (i) -- (i_b2);
             
              \node[main, draw = none, fill = none] (i_r1) at (1.85, 1.6) {};
              \node[main, draw = none, fill = none] (i_r2) at (-1.85, 1.6) {};

              \draw[red] (i) -- (i_r1); 
              \draw[red] (i) -- (i_r2);
              
              \node[main, scale = 0.25] (i_b_1) at (0, -2.8) {};
              \node[main, draw=none, fill=none, scale=1] (i_b_1_name) at (-0.4,-2.7) {$x^{blue}_i$}; 
              
              \node[main, scale = 0.25] (i_r_1) at (0,2.8) {};
              \node[main, draw=none, fill=none, scale=1] (i_r_1_name) at (0.4,2.65) {$x^{red}_i$}; 
              \node[main, draw=none, fill=none, scale = 1] (i_break) at (2, 3.5) {$\cdots$};

              \node[main, scale = 0.25] (i_b_2) at (1.3, -2.4) {};
              \node[main, draw=none, fill=none, scale=1] (i_b_2_name) at (1,-2.4) {$z$};
              
              \draw[red] (i_r_1) to[out=45,in=180] (i_break);

              \node[main, draw=none, fill=none, scale = 0.75] (break_1) at (-2, -3.8) {};

              \node[main, draw = none, fill = none, scale = 1] (i_break_2) at (-2, -3.5) {$\cdots$};
              \draw[blue, very thick] (i_break_2) to[out=0,in=-135] (i_b_1);
	     
	     \draw[red] (break_1) to[out=0, in=-135] (i_b_2);

	     \node[main, draw=none, fill=none, scale=0.75,rotate=90] (i_r_temp) at (0,1.6) {$\cdots$};
	     \draw[red] (i) -- (i_r_temp);
	     \draw[blue, very thick] (i_r_temp) -- (i_r_1);

         \node[main, scale = 0.25] (z_prime) at (0, -1.4) {};
         \node[main, draw=none, fill=none, scale=1] (z_prime_name) at (-0.3,-1.4) {$z'$};
	     
	     \node[main, draw=none, fill=none, scale=0.75,rotate=90] (i_b_temp_1) at (0,-0.6) {$\cdots$};
	     \draw[blue, very thick] (i) -- (i_b_temp_1);
	     \draw[blue, very thick] (i_b_temp_1) -- (z_prime);
         \node[main, draw=none, fill=none, scale=0.75,rotate=90] (z_prime_temp_1) at (0,-2.1) {$\cdots$};
         \draw[red] (z_prime) -- (z_prime_temp_1);
         \draw[red] (z_prime_temp_1) -- (i_b_1);
	     
         \node[main, scale = 0.25] (y_z) at (0.4, -1.7) {};
         \node[main, draw=none, fill=none, scale=1] (y_z_name) at (0.6,-1.55) {$y_z$};
	     \node[main, draw=none, fill=none, scale=0.75,rotate=-40] (z_prime_temp_2) at (0.9,-2.1) {$\cdots$};
	     \draw[red] (z_prime) -- (y_z);
         \draw[blue, very thick] (y_z) -- (z_prime_temp_2);
	     \draw[blue, very thick] (z_prime_temp_2) -- (i_b_2);

        \end{tikzpicture}
        \caption{A blow up of the hourglass $H_i$ in $G$, as well as the definitions of $z'$ and $y_z$.}
        \label{Hourglass: i vertex}
    \end{figure}
    
    As $i \neq j$, it follows that $H_i \neq H_j$ and thus by symmetry the same structure holds with respect to the cycle $H_j, \dots, H_s, H_j$. That is, one can similarly define alternating paths $P_l$ for $l = j+1, \dots, s$ as well as vertices $w'$ and $y_w$ in $H_j$ corresponding to $z'$ and $y_z$ in $H_i$. 
    
    It's not hard to see that going through the intermediate hourglasses $H_{i+1}, \dots, H_{j-1}$ in a similar manner we obtain an alternating path from $i$ to $j$ and thus an alternating path from $y_z$ to $y_w$. Adding the edges 
    $y_z z'$ and $y_w w'$ gives a proper alternating odd bicycle as required.
\end{proof}

\section{Connectivity of the Solution Space} \label{sec: connectivity of sol space}

In this section, we consider a nice application of the technique used in the proof of 
Lemma~\ref{lma: improper alternating bicycle}. In particular, we consider $\mathcal{G}(n,p,2)$ in the 
subcritical regime, that is $p = \frac{2 - \epsilon}{n}$ for some constant $\epsilon > 0$.
Given $(G,c)$ sampled from $\mathcal{G}(n,p = (2-\epsilon)/n,2)$ we shall show that for any two vertex 
colorings $\rho$ and $\tau$ adapted to $c$, we can construct a sequence of colorings adapted to 
$c$ which transform $\rho$ into $\tau$ such that for any two adjacent colorings in the sequence, they differ on $\mathcal{O}(\log{n})$ vertices. Particularly, we prove the following Theorem,

\newpage

\begin{theorem*}[\ref{theorem: connected sol space}]
	Let $\epsilon > 0$ be a constant. For $(G,c)$ sampled from $\mathcal{G}(n,p = \frac{2-\epsilon}{n},2)$ 
    w.h.p., for any two vertex colorings $\rho$ and $\tau$ adapted to $c$ there exists a sequence of colorings adapted to $c$, $\rho = \rho_1, \rho_2, \rho_3, \cdots, \rho_l = \tau$ such that for all $i < l$: 
	\begin{equation*}
		\left| \{ v\in V \mid \rho_i(v) \neq \rho_{i+1}(v) \} \right| = \mathcal{O}(\log{n})
	\end{equation*}
\end{theorem*} 

We define a process \hyperref[cluster process]{\textit{Intermediate\_Coloring}} which takes in a 
coloring $\rho$ adapted to $c$ and a vertex $v$ and produces
a coloring $\rho'$ adapted to $c$ such that $\rho'(v) \neq \rho(v)$. 
In particular, when choosing $v$ such that $\rho(v) \neq \tau(v)$ we will show that the coloring $\rho'$ 
agrees with $\rho$ on all but $\mathcal{O}(\log{n})$ vertices and the distance from 
$\rho'$ to $\tau$ is strictly smaller than the distance from $\rho$ to $\tau$. Note that by distance, 
we mean the number of vertices on which they disagree. Formally, $\rho'$ will satisfy the following two properties: 
\begin{equation} \label{eq: inter agrees on}
	\left| \{ v\in V \mid \rho(v) \neq \rho'(v) \} \right| = \mathcal{O}(\log{n}) \,,
\end{equation}
and 
\begin{equation} \label{eq: converge cond}
	\left| \{ v\in V \mid \rho'(v) \neq \tau(v) \} \right| < \left| \{ v \in V \mid \rho(v) \neq \tau(v) \} \right| \,.
\end{equation}

Therefore, as one may guess, we can form our sequence of colorings adapted to $c$ by repeatedly 
applying \hyperref[cluster process]{\textit{Intermediate\_Coloring}} as follows: 
\begin{itemize}
	\item $\rho_0 := p$,
	\item While there exists $v$ such that $\rho_{i-1}(v) \neq \tau(v)$: $\rho_{i} := \hyperref[cluster process]{\textit{Intermediate\_Coloring}}(\rho_{i-1}, v)$.
\end{itemize}

Properties~(\ref{eq: inter agrees on}) and (\ref{eq: converge cond}) ensure this sequence converges to $\tau$ and sequential colorings in the sequence differ on at most $\mathcal{O}(\log{n})$ vertices as required.

Informally, \hyperref[cluster process]{\textit{Intermediate\_Coloring}}$(\rho, v)$ will 
try to swap the color of $v$. 
Naturally, this may cause violations.
Similarly to the process described in the proof of Lemma~\ref{lma: improper alternating bicycle}, 
we continue to swap the color of vertices to correct these violations, until no more violated edges exist. 

We now formally define \hyperref[cluster process]{\textit{Intermediate\_Coloring}}. 
  
\begin{process}{\textit{Intermediate\_Coloring}$(\rho, v)$:} \label{cluster process}
    \begin{enumerate}
        \item Initialize: $S_0 := \{v\}$.
        \item Initialize: $\rho'_0 := \rho$
        \item While $S_i \neq \emptyset:$ 
            \begin{enumerate}[label=\alph*., ref=\theenumi\alph*]
                \item $\rho_i' = \rho_{i-1}'$ everywhere except for on $S_{i-1}$. For all $x \in S_{i-1}$, $\rho_i'(x) \neq \rho_{i-1}'(x)$, that is swap the color assigned to each vertex in $S_{i-1}$.
                \item $S_i \coloneqq \cup_{x \in S_{i-1}} N_{\rho_{i}'(x)}^{\rho_{i}'}(x)$, that is $S_i$ is the endpoints of the violated edges induced by swapping the colors of $S_{i-1}$.
            \end{enumerate}
        \item Return: $\rho'_i$
    \end{enumerate}
\end{process}

We prove the following three lemmas regarding 
\hyperref[cluster process]{\textit{Intermediate\_Coloring}}($\rho, v$) when $v$ is not frozen, 
which will imply Lemma~\ref{theorem: connected sol space}. 
The first lemma states that, we never ``revisit" a 
vertex, which implies that we only change the color of each vertex at most once.

\begin{lemma} \label{lma: Disjoint S}
    If $v$ is non-frozen then, $S_i \cap S_j = \emptyset$ for all $i \neq j$.
\end{lemma}

\begin{proof}
	Assume for the sake of a contradiction, there exists $i > j$ such that $S_i \cap S_j \neq \emptyset$; 
    furthermore, assume $i$ is minimal, that is $i$ is the first iteration in which a vertex in $S_i$ is 
    also in a previously seen $S_j$. Following a similar argument as in the proof of 
    Lemma~\ref{lma: improper alternating bicycle}, $v$ is in the handle of an alternating odd 
    unicycle. Therefore, by Corollary~\ref{cor: frozen char}, $v$ is frozen contradicting our assumption that $v$ is not frozen. 
\end{proof}

In particular, Lemma~\ref{lma: Disjoint S} directly implies the following:

\begin{corollary} \label{cor: no revisit}
	The vertices whose colors are changed by \hyperref[cluster process]{\textit{Intermediate\_Coloring}}($\rho, v$) for non-frozen $v$
    are precisely the vertices in $\bigcup_i S_i$. Furthermore, let $\rho'$ be the coloring returned by 
    \hyperref[cluster process]{\textit{Intermediate\_Coloring}}($\rho, v$); if $u \in S_i$ then
    $\rho'(u) = \rho_{i+1}'(u)$. 
\end{corollary}

The second lemma states that if we choose $v$ such that $\rho(v) \neq \tau(v)$ then, the vertex coloring returned by 
\hyperref[cluster process]{\textit{Intermediate\_Coloring}}($\rho, v$) 
is closer to the target coloring than the initial is, as stated in our 
desired property~\ref{eq: converge cond}.
In particular, we show that we only change the color of vertices which $\rho$ and $\tau$ 
disagree on, which along with Lemma~\ref{lma: Disjoint S} implies the coloring produced by 
\hyperref[cluster process]{\textit{Intermediate\_Coloring}}$(\rho, v)$ agrees with $\tau$ on more vertices than $\rho$ does. 

\begin{lemma} \label{lma: Intermediate coloring}
	Let $\rho' = $ \hyperref[cluster process]{\textit{Intermediate\_Coloring}}$(\rho, v)$ for $v$ such that $\rho(v) \neq \tau(v)$. 
    For all $u \in V$ such that $\rho'(u) \neq \rho(u)$ we have $\rho'(u) = \tau(u)$.
\end{lemma}

\begin{proof}
	Assume for the sake of a contradiction that there exists $u \in V$ such that $\rho'(u) \neq \rho(u)$ 
    and $\rho'(u) \neq \tau(u)$. By Corollary~\ref{cor: no revisit}, there exists a unique $i$ such that 
    $u \in S_i$. Choose $u$ such that $i$ is minimal, that is for all $j < i$ and for all $v \in S_j$, the 
    statement holds for $v$. Note that clearly $i > 0$.
    
    Without loss of generality, assume that $\rho'(u) = blue$. By Corollary~\ref{cor: no revisit} 
    there exists a $z \in S_{i-1}$ 
    with $\rho'(z) = red$ and $c(zu) = red$. As $\rho'(u) \neq \tau(u)$ we have $\tau(u) = red$. 
    Furthermore, minimality in our choice of $u$ implies that $\tau(z) = \rho'(z) = red$. 
    However, as $\tau(u) = \tau(z) = red = c(zu)$ this contradicts $\tau$ being a vertex coloring 
    adapted to $c$ as required.
\end{proof}

Lastly, the third lemma states that if $v$ is non-frozen, \hyperref[cluster process]{\textit{Intermediate\_Coloring}}$(\rho, v)$ does not make too many changes. Precisely, 
at most a logarithmic number of vertices had their color changed by the process, 
as seen in the desired property~\ref{eq: inter agrees on}.

\begin{lemma} \label{lma: sub linear change}
	If $v$ is non-frozen then for $\rho' = $ \hyperref[cluster process]{\textit{Intermediate\_Coloring}}$(\rho, v)$, w.h.p. $|\{ v \in V \mid \rho'(v) \neq \rho(v) \}| = \mathcal{O}(\log{n})$
\end{lemma}

Before proceeding let us introduce the Graph Branching Process (GBP) run on $\mathcal{G}(n,p/2)$ 
    (see. Chapter 11 in \cite{TheProbabilisticMethod} or \cite{ErdosRenyiRandGraphs}). This process is precisely
    a Breadth First Search (BFS) exploration of a connected component of $\mathcal{G}(n,p/2)$. We maintain a queue 
    of vertices, initialized to contain $v$. We let $Z_t$ denote the vertices added to the queue 
    at time $t$. Let $T_t$ denote the ``used" 
    vertices at time $t$ with $T_0 = \{v\}$. So $Z_t$ denotes the neighbors of the first vertex in the queue 
    within the ``unused" vertices $V \setminus T_{t-1}$ at time $t$.  

    \begin{process}{\textit{Graph\_Branching\_Process}$(v)$} \label{graph branching process}
        \begin{enumerate}
            \item Initialize: Queue $Q$ containing $v$,  $T_0 = \{v\}$, $t = 1$
            \item While $Q \neq \emptyset:$ 
                \begin{enumerate}[label=\alph*., ref=\theenumi\alph*]
                    \item $x := Pop(Q)$
                    \item $Z_t := N(x) \setminus T_{t-1}$
                    \item $T_t := T_{t-1} \cup Z_t$ 
                    \item $t := t + 1$ 
                \end{enumerate}
        \end{enumerate}
    \end{process}

    So,
    \begin{equation*}
        |Z_t| \sim \mathcal{B}in(n - |T_{t-1}|, p/2) \,.
    \end{equation*}

    For ease of notation, if $Q = \emptyset$ at iteration $t$ we let $T_t = T_{t-1}$. 
    Thus, $T_\infty := \lim_{t \rightarrow \infty} T_t$ denotes the set of vertices ``visited" by 
    \hyperref[graph branching process]{\textit{Graph\_Branching\_Process}}($v$). Particularly, its not hard to see
    that $T_\infty = V(C(v))$, where $C(v)$ denotes the connected component containin $v$.

    \begin{remark}(Chapter 11 in \cite{TheProbabilisticMethod} or \cite{ErdosRenyiRandGraphs}) \label{rmk: connected component size}
        If $p < 1/n$ then w.h.p. the largest component of $\mathcal{G}(n,p)$ has size $\mathcal{O}(\log{n})$. 
        In particular, when running \hyperref[graph branching process]{\textit{Graph\_Branching\_Process}}($v$)
        on $\mathcal{G}(n,p)$, w.h.p. $|T_\infty| = \mathcal{O}(\log{n})$.
    \end{remark}

    We wish to compare the number of vertices whose colors change during \\
    \hyperref[cluster process]{\textit{Intermediate\_Coloring}}($\rho, v$) 
    to the number of vertices seen by 
    \hyperref[graph branching process]{\textit{Graph\_Branching\_Process}}($v$). 
    \hyperref[cluster process]{\textit{Intermediate\_Coloring}} resembles a branching process, 
    so we would like to model it with \\ 
    \hyperref[graph branching process]{\textit{Graph\_Branching\_Process}}.
    However, the steps of \hyperref[cluster process]{\textit{Intermediate\_Coloring}} are determined in 
    part by the vertex coloring which is not random. This prevents us from directly modeling 
    \hyperref[cluster process]{\textit{Intermediate\_Coloring}}
    with \hyperref[graph branching process]{\textit{Graph\_Branching\_Process}}.

    To overcome this, we introduce a 
    closely related process called \hyperref[cluster process 2]{\textit{Uncolored\_Process}} 
    which acts like \hyperref[cluster process]{\textit{Intermediate\_Coloring}} but does not use the vertex coloring. 
    \hyperref[cluster process 2]{\textit{Uncolored\_Process}}($v, \rho(v)$)
    starts at a vertex $v$ and searches in a Breadth-First manner for vertices 
    reachable from $v$ by an alternating path beginning with the color 
    which is not $\rho(v)$. As usual, we only consider vertices which are not previously seen. 
    This is handled by maintaining a set 
    $T_i$ which stores the vertices seen from iteration $1$ to $i$.

    \begin{process}{\textit{Uncolored\_Process}$(v, \rho(v)):$} \label{cluster process 2}
        \begin{enumerate}
            \item Initialize: $S'_0 = \{v\}$
            \item Initialize: $T_0 = S'_0$
            \item Initialize: $\kappa_0 = \rho(v)$
            \item While $S'_i \neq \emptyset:$ 
                \begin{enumerate}[label=\alph*., ref=\theenumi\alph*]
                    \item $\kappa_i \neq \kappa_{i-1}$. That is, swap the color.
                    \item \label{Uncolored process next layer} $S'_i \coloneqq \left( \cup_{x \in S'_{i-1}} N_{\kappa_{i}}(x) \right) \setminus T_{i-1}$, that is $S'_i$ denotes the set of vertices incident to $S'_{i-1}$ by an edge of color $\kappa_{i}$
                    \item $T_i := T_{i-1} \cup S'_i$
                \end{enumerate}
        \end{enumerate}
    \end{process}

    Note the following equivalent definition of $\kappa_i$ assuming $\rho(v) = red$:
    \[\kappa_i = \begin{cases} 
      red & i \equiv 0 \mod{2} \\
      blue & i \equiv 1 \mod{2}
    \end{cases} \,.
    \] 

    As we will see in Figure~\ref{fig: issue with uncolored process} below, \hyperref[cluster process 2]{\textit{Uncolored\_Process}}($v, \rho(v)$) 
    does not reach all vertices 
    reachable from $v$ by an alternating path beginning with the color 
    which is not $\rho(v)$; however, this suffices for our analysis. 

    We think of \hyperref[cluster process 2]{\textit{Uncolored\_Process}}$(v, \rho(v))$ 
    being run in parallel with \hyperref[cluster process]{\textit{Intermediate\_Coloring}}($\rho, v$).
    For the remainder of this section, we let $S_i$ denote the sets formed by running \\
    \hyperref[cluster process]{\textit{Intermediate\_Coloring}}$(\rho, v)$ 
    and $S_i'$ denote the sets formed by running \hyperref[cluster process 2]{\textit{Uncolored\_Process}}$(v, \rho(v))$.

    Its not hard to see that the number of vertices visited by \hyperref[cluster process 2]{\textit{Uncolored\_Process}}
    follows the same distribution as the number of vertices visited by 
    \hyperref[graph branching process]{\textit{Graph\_Branching\_Process}}; 
    this is handled more formally in the following lemma.

    \begin{lemma} \label{lma: uncolored process vs connected component}
        $|\bigcup_i S'_i|$ is distributed as the size of the connected component in $\mathcal{G}(n,p/2)$ containing $v$.
    \end{lemma}

    \begin{proof}
        We couple the distribution of $\left|\bigcup_{i} S'_i\right|$ to the size of a connected component in $\mathcal{G}(n,p/2)$ by 
        showing that \hyperref[cluster process 2]{\textit{Uncolored\_Process}}$(v, \rho(v))$ follows precisely 
        the same distribution as \\ \hyperref[graph branching process]{\textit{Graph\_Branching\_Process}}($v$) run on $\mathcal{G}(n,p/2)$. 

        Consider the construction of $S'_i$ from $S'_{i-1}$ in Step~\ref{Uncolored process next layer} of 
        \hyperref[cluster process 2]{\textit{Uncolored\_Process}}$(v, \rho(v))$; That is,
        \begin{equation*}
            S'_i \coloneqq \left( \cup_{x \in S'_{i-1}} N_{\kappa_{i}}(x) \right) \setminus T_{i-1} \,.
        \end{equation*}

        Let $\gamma = |S'_{i-1}|$ and fix an ordering $\sigma$ on $S'_{i-1}$. Let $T_i^0 := T_{i-1}$. 
        For $j = 1, \dots, \gamma$ define recursively:
        \begin{equation*}
            Z_j := N_{\kappa_i}(v_j) \setminus T_{i}^{j-1} \qquad and \qquad T_i^{j} := T_{i}^{j-1} \cup Z_j \,.
        \end{equation*}
        It's not hard to see that,
        \begin{equation*}
            \bigcup_{j=1}^{\gamma} Z_j = S'_i \qquad and \qquad T_i = T_i^{\gamma} \,.
        \end{equation*}
        Furthermore, 
        \begin{equation*}
            |Z_j| \sim \mathcal{B}in(n - |T_{i}^{j-1}|, p/2) \,.
        \end{equation*}

        Thus, $|\bigcup_i S'_i|$ follows the same distribution as $|T_\infty|$ which is the size of the connected component 
        in $\mathcal{G}(n,p/2)$ containing $v$. 
    \end{proof}

    It remains to relate the original process \hyperref[cluster process]{\textit{Intermediate\_Coloring}}
    to \hyperref[cluster process 2]{\textit{Uncolored\_Process}}. 
    At a first glance, it seems natural that 
    \hyperref[cluster process]{\textit{Intermediate\_Coloring}} is contained within 
    \hyperref[cluster process 2]{\textit{Uncolored\_Process}} for any starting vertex $v$.
    However, this is not the case as we will see below. 
    Fortunately, it is true if $v$ is non-frozen; this is handled in 
    Lemma~\ref{lma: uncolored vs colored process 1} and Lemma~\ref{lma: uncolored vs colored process 2}.

    Consider running both 
    \hyperref[cluster process]{\textit{Intermediate\_Coloring}}($\rho, v$) and 
    \hyperref[cluster process 2]{\textit{Uncolored\_Process}}($v, \rho(v)$) on the graph in Figure~\ref{fig: issue with uncolored process}.
    It is not hard to see that $z$ is visited by \hyperref[cluster process]{\textit{Intermediate\_Coloring}}($\rho, v$), but not by 
    \hyperref[cluster process 2]{\textit{Uncolored\_Process}}($v, \rho(v)$). 
    This is because when searching from $y$, 
    \hyperref[cluster process]{\textit{Intermediate\_Coloring}}($\rho, v$)
    looks for $red$ edges, whereas \hyperref[cluster process 2]{\textit{Uncolored\_Process}}($v, \rho(v)$) looks
    for $blue$ edges. However, looking at 
    Figure~\ref{fig: issue with uncolored process} we notice that $v$ is in the handle of 
    an alternating odd unicycle and is thus frozen by Corollary~\ref{cor: frozen char}.

    \begin{figure}[h]
    \centering
        \begin{tikzpicture}[node distance={15mm}, main/.style = {draw, circle}] 
            \begin{scope}[every node/.style={circle, thick, draw, minimum size=3mm}]
                \node[main, pattern=north east lines, pattern color=blue] (u) at (0,0) {}; 
                \node[draw = none] at (u) [name=fake_u,outer sep=5pt,inner sep=5pt]{};
                \node[draw = none] at (1,0) {$S_0 = S'_0$};
                \node[main, draw=none, fill=none] (u_name) at (-0.7,0) {$v$};

                \node[main, fill=red] (s1_2) at (0,-2) {}; 
                \node[draw = none] at (s1_2) [name=fake_s1_2,outer sep=5pt,inner sep=5pt]{};
                \node[main, pattern=north east lines, pattern color=blue] (s1_1) at (-2,-2) {}; 
                \node[draw = none] at (s1_1) [name=fake_s1_1,outer sep=5pt,inner sep=5pt]{};
                \node[draw = none] at (-2, -1.3) {y};
                \node[main, fill=red] (s1_3) at (2,-2) {}; 
                \node[draw = none] at (s1_3) [name=fake_s1_3,outer sep=5pt,inner sep=5pt]{};
                \node[draw = none] at (s1_3) [name=fake_s1_3_2,outer sep=7pt,inner sep=7pt]{};
                \node[draw = none] at (3, -2) {$S_1$};

                \draw[red] (u) -- (s1_1);
                \draw[red] (u) -- (s1_2);
                \draw[red] (u) -- (s1_3);

                \node[main, pattern=north east lines, pattern color=blue] (s2_4) at (0.5,-4) {};   
                \node[draw = none] at (s2_4) [name=fake_s2_4,outer sep=5pt,inner sep=5pt]{};
                \node[main, pattern=north east lines, pattern color=blue] (s2_5) at (1.5,-4) {};    
                \node[draw = none] at (s2_5) [name=fake_s2_5,outer sep=5pt,inner sep=5pt]{};
                \node[main, pattern=north east lines, pattern color=blue] (s2_6) at (2.5,-4) {};    
                \node[draw = none] at (s2_6) [name=fake_s2_6,outer sep=5pt,inner sep=5pt]{};
                \node[draw = none] at (3.5, -4) {$S_2$};

                \draw[blue, very thick] (s1_2) -- (s2_4);
                \draw[blue, very thick] (s1_3) -- (s2_5);
                \draw[blue, very thick] (s1_3) -- (s2_6);

                \node[main, fill=red] (s3_1) at (0,-6) {};   
                \node[draw = none] at (s3_1) [name=fake_s3_1,outer sep=5pt,inner sep=5pt]{};
                \node[draw = none] at (1, -6) {$S_3$};

                \draw[red] (s2_4) -- (s3_1);

                \draw[blue, very thick] (s3_1) -- (s1_1);
                \node[draw = none] at (-3, -2) {$S_4$};

                \node[main] (s2_prime) at (-5,-4) {};   
                \node[draw = none] at (s2_prime) [name=fake_s2_prime,outer sep=5pt,inner sep=5pt]{};
                \node[draw = none] at (-6, -4) {$S_2'$};

                \node[main, fill=red] (s5_1) at (-3,-4) {};   
                \node[draw = none] at (s5_1) [name=fake_s5_1,outer sep=5pt,inner sep=5pt]{};
                \node[draw = none] at (-3, -3.4) {z};
                \node[draw = none] at (-2, -4) {$S_5$};

                \draw[blue, very thick] (s1_1) -- (s2_prime);
                
                \draw[red] (s1_1) -- (s5_1);

            \end{scope}

            \draw[blue,fill=my_grey,opacity=0.2](fake_u.west) 
                to[closed,curve through={
                (fake_u.north) ..
                (fake_u.east) ..
                }] (fake_u.west);

                \draw[blue,fill=my_grey,opacity=0.2](fake_s1_2.west) 
                to[closed,curve through={
                (fake_s1_2.north) ..
                (fake_s1_3.north) ..
                (fake_s1_3.east) ..
                (fake_s1_3.south) ..
                (fake_s1_2.south) ..
                }] (fake_s1_2.west);

                \draw[blue,fill=my_grey,opacity=0.2](fake_s2_4.west) 
                to[closed,curve through={
                (fake_s2_4.north) ..
                (fake_s2_6.north) ..
                (fake_s2_6.east) ..
                (fake_s2_6.south) ..
                (fake_s2_4.south) ..
                }] (fake_s2_4.west);

                \draw[blue,fill=my_grey,opacity=0.2](fake_s3_1.west) 
                to[closed,curve through={
                (fake_s3_1.north) ..
                (fake_s3_1.east) ..
                }] (fake_s3_1.west);

                \draw[blue,fill=my_grey,opacity=0.2](fake_s1_1.west) 
                to[closed,curve through={
                (fake_s1_1.north) ..
                (fake_s1_1.east) ..
                }] (fake_s1_1.west);

                \draw[blue,fill=my_grey,opacity=0.2](fake_s5_1.west) 
                to[closed,curve through={
                (fake_s5_1.north) ..
                (fake_s5_1.east) ..
                }] (fake_s5_1.west);
    
        \end{tikzpicture}
        \caption{An instance in which we run 
        \hyperref[cluster process]{\textit{Intermediate\_Coloring}}($\rho, v$) and 
        \hyperref[cluster process 2]{\textit{Uncolored\_Process}}($v, \rho(v)$) 
        Here the vertex $z$ is seen only by  
        \hyperref[cluster process]{\textit{Intermediate\_Coloring}}($\rho, v$). 
        Note the vertex coloring shown is $\rho$. 
        } 
        \label{fig: issue with uncolored process}
    \end{figure}
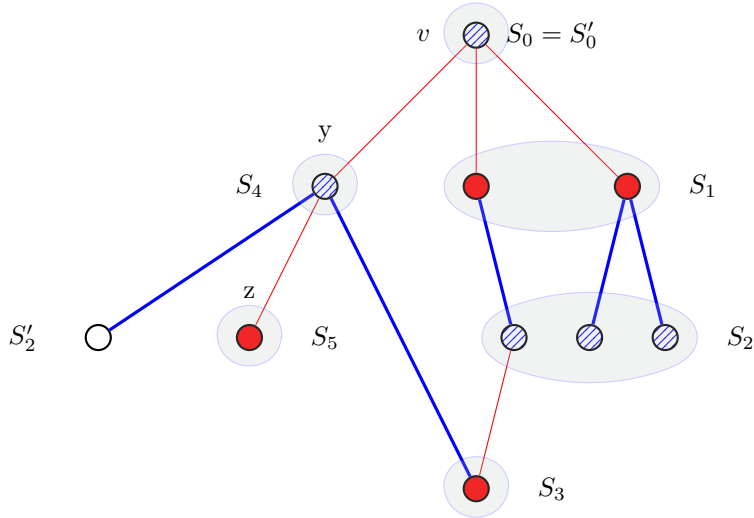

    Suppose $v$ is a non-frozen vertex. 
    The following lemma states: if a vertex $x$ is seen by both 
    \hyperref[cluster process]{\textit{Intermediate\_Coloring}}($\rho, v$) and 
    \hyperref[cluster process 2]{\textit{Uncolored\_Process}}($v, \rho(v)$) then when branching from $x$, both processes consider edges of the same color.

    \begin{lemma} \label{lma: uncolored vs colored process 1}
    	Let $v$ be non-frozen. If $x \in S_i \cap S'_j$, then $i \equiv j \mod{2}$; equivalently, 
        $\kappa_{j+1} = \rho_{i+1}'(x)$.
    \end{lemma}

    The equivalence between $i \equiv j \mod{2}$ and $\kappa_{j+1} = \rho'(x)$ 
    follows from Corollary~\ref{cor: no revisit}.
    Intuitively, if $i \not\equiv j \mod{2}$ then an odd cycle will be formed with $v$ in the handle 
    which by Corollary~\ref{cor: frozen char} implies $v$ is frozen. 

    \begin{proof}
        Assume for the sake of a contradiction that there exists $x \in S_i \cap S'_j$ for some $i, j$ 
        with $\kappa_{j+1} \neq \rho'(x)$; furthermore, choose $i$ to be minimal. 
        Without loss of generality, assume $\rho'(x) = blue$ and $\kappa_{j+1} = red$, 
        as seen in Figure~\ref{fig: two processes lma 1}. 
        
        \begin{figure}[h]
            \centering
            \begin{tikzpicture}[node distance={15mm}, main/.style = {draw, circle}] 
                \begin{scope}[every node/.style={circle, thick, draw, minimum size=3mm}]
                    \node[main, pattern=north east lines, pattern color=blue] (x) at (0,0) {}; 
                    \node[main, draw=none, fill=none] (x_name) at (0,0.3) {$x$}; 

                    \node[main, draw = none] (a) at (-1.5,1.5) {};
                    \node[main, draw = none] (b) at (1.5,1.5) {};
                    \node[main, draw = none] (c) at (-1.5,-1.5) {};
                    \node[main, draw = none] (d) at (1.5,-1.5) {};

                    \draw[red] (a) -- (x);
                    \draw[red] (x) -- (d);
                    \draw[blue, very thick] (b) -- (x);
                    \draw[blue, very thick] (x) -- (c);

                    \node[main, draw = none] (kappa_b) at (2,0.75) {$\kappa_{j} = blue$};
                    \node[main, draw = none] (kappa_a) at (2,-0.75) {$\kappa_{j+1} = red$};

                    \node[main, draw = none] (rho_x) at (-2,-0.75) {$\rho'(x) = blue$};
                \end{scope}

            \end{tikzpicture}
            \caption{The vertex $x \in S_i \cap S'_j$ with $\kappa_{j+1} \neq \rho'(x)$ under the assumption 
            that $\rho'(x) = blue$ and $\kappa_{j+1} = red$.}
            \label{fig: two processes lma 1}
        \end{figure}
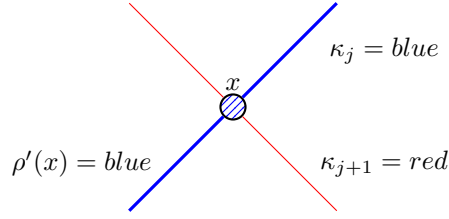
        
        Let $P_r$ be an 
        alternating $v,x$-path ending in $red$ induced by the layers $S_0, S_1, S_2, \dots, S_i$ of 
        \hyperref[cluster process]{\textit{Intermediate\_Coloring}}$(\rho, v)$. Similarly, 
        let $P_b$ be an alternating $v,x$-path ending in $blue$ induced by the layers 
        $S'_0, S'_1, \dots, S'_j$ of 
        \hyperref[cluster process 2]{\textit{Uncolored\_Process}}$(v, \rho(v))$.
        
        Let $y \in V(P_r) \cap V(P_b)$ be closest to $x$ along the path $P_b$. 
        That is the $y,x$-subpath $P_b^{x,y} \subseteq P_b$ intersects $P_r$ only at $y$ and $x$. 
        Let $C$ denote the cycle formed by $P_b^{x,y}$ and the $y,x$-subpath $P_r^{x,y} \subseteq P_r$. 
        Furthermore, let $P^{v,y}_r$ be the $v,y$-subpath of $P_r$. By choice of $y$, $P^{v,y}_r$ intersects $C$ only at $y$. 
        So, $y \in S_{i'} \cap S'_{j'}$ for some $i' < i$ and $j' < j$. So, 
        by minimality of $i$ it follows that $\kappa_{j'+1} = \rho'(y)$ and thus, the two edges in $C$ 
        incident to $y$ are of the same color. Furthermore, by assumption the two edges incident to $x$ in $C$ 
        are of different colors. It follows by the alternating structure of $P_r$ and $P_b$ 
        that $P^{v,y}_r$ and $C$ form an alternating odd unicycle with $v$ on the handle. 
        Thus by Corollary~\ref{cor: frozen char}, $v$ is frozen contradicting our assumption that $v$ is non-frozen.
    \end{proof}

    \begin{lemma} \label{lma: uncolored vs colored process 2}
        If $v$ is non-frozen then for all $i$, $S_i \subseteq \cup_{j=0}^{i} S'_j$.
    \end{lemma}

    \begin{proof}
        Assume for the sake of a contradiction that there exists $i$ such that 
        $S_i \not\subseteq \cup_{j=0}^{i} S'_j$; furthermore take $i$ to be minimal.
        So there exists $x \in S_i$ such that $x \notin \cup_{j=0}^{i} S'_j$. 
        Let $y \in S_{i-1}$ be incident to $x$ with $c(xy) = \rho'(y)$. 
        Minimality of $i$ implies that $y \in S'_j$ for some $j \leq i - 1$. 
        So, $y \in S_{i-1} \cap S'_j$ and thus by Lemma~\ref{lma: uncolored vs colored process 1}, $\kappa_{j+1} = \rho'(y)$.
        As $x \notin \cup_{l = 0}^{j} S'_l$ and $\kappa_{j+1} = \rho'(y)$ it follows that $x$ 
        would have been added to $S'_{j+1}$, contradiction.
    \end{proof}
    
    We now prove Lemma~\ref{lma: sub linear change}. 

    \begin{proof}[Proof of Lemma~\ref{lma: sub linear change}]
        By Lemma~\ref{lma: uncolored vs colored process 2}, 
        Lemma~\ref{lma: uncolored process vs connected component}, and Remark~\ref{rmk: connected component size},
        \begin{equation} \label{eq: uncolored bound}
            \left|\bigcup_{i} S_i\right| \leq \left|\bigcup_{i} S'_i\right| = \mathcal{O}(\log{n})\,,
        \end{equation}
        where the last equality holds w.h.p..
    \end{proof}

We are now ready to prove Theorem~\ref{theorem: connected sol space}.

\begin{proof}[Proof of Theorem~\ref{theorem: connected sol space}]
	We first note that it's not hard to see that the vertex coloring $\rho'$ returned from 
    \hyperref[cluster process]{\textit{Intermediate\_Coloring}} $(\rho, v)$ is a valid vertex 
    coloring adapted to $c$ for any vertex $v$. We therefore consider the following sequence of vertex colorings, all of 
    which are adapted to $c$: $\rho_0 = \rho$ and for all 
    $i \geq 1$ if there exists $v$ such that $\rho_{i-1}(v) \neq \tau(v)$ then 
    $\rho_i = \hyperref[cluster process]{\textit{Intermediate\_Coloring}}(\rho_{i-1}, v)$; otherwise, $\rho_i = \rho_{i-1}$. 
    Therefore, by Lemmas~\ref{lma: Disjoint S}, \ref{lma: Intermediate coloring}, and \ref{lma: sub linear change} 
    it follows that $\{\rho_i\}$ converges to $\tau$ and for each $i \geq 1$,
    \begin{equation*}
        \left| \{ v \in V \mid \rho_i(v) \neq \rho_{i-1}(v) \} \right| = \mathcal{O}(\log{n}) \,,
    \end{equation*}
    as required.
\end{proof}

\subsection*{Acknowledgements}

The author is grateful to Michael Molloy for many helpful discussions and comments as well as overall 
guidance and support. The author is also thankful to Lior Gishboliner for reviewing the paper and 
providing helpful comments. The author was supported by the Ontario Graduate Scholarship (OGS).

\bibliographystyle{alpha} 
\bibliography{refs}

\appendix

\section{Relationship between $\mathcal{G}(n,p,k)$ and $\mathcal{G}(n,m,k)$} \label{appendix: two models}

The purpose of this section is to prove an approximate equivalence of the two models 
$\mathcal{G}(n,p,k)$ and $\mathcal{G}(n,m,k)$, particularly with respect to events which are monotone under edge colorings.

We proceed by relating properties which are monotone under edge colorings to properties that are monotone in the classical sense. 
Particularly, if a property $A$ is monotone under edge colorings and and $\gamma > 0$,
define $Q_A^\gamma$ to be the following property: We say a graph $G$ has $Q_A^\gamma$ 
if and only if a uniformly random edge coloring of $G$ has $A$ with probability at least $1 - \gamma$.

We now claim that $Q_A^\gamma$ is monotone for any property $A$ that is monotone 
under edge colorings and $\gamma > 0$. 
Assume a graph $H$ has $Q_A^\gamma$ and $H \subseteq G$. We wish to show that $G$ 
also has the property $Q_A^\gamma$. 
Consider a uniformly random edge coloring $c_G$ of $G$. 
As $H \subseteq G$, let $c_H$ be the coloring of $H$ induced by $c_G$. 
Therefore, $c_H$ is a uniformly random edge coloring of $H$ and as $H$ has $Q_A^\gamma$, 
with probability at least $1 - \gamma$, $(H, c_H)$ has $A$. 
So, as $A$ is monotone under edge colorings, and $(H, c_H) \subseteq (G, c_G)$ it 
follows that $(G,c_G)$ has $A$ with probability at least $1 - \gamma$. 
Therefore, $G$ has $Q_A$ as required.

We will be particularly interested in the case when $\gamma \coloneqq f(n)$ for a decreasing function $f : \mathbb{R}_{\geq 0} \rightarrow \mathbb{R}_{\geq 0}$
such that $\lim_{x \rightarrow \infty} f(x) = 0$.

By Definition~\ref{rand graph prob model} and Definition~\ref{rand graph edge model} we have the following equivalence relations for properties $A$ which are monotone under edge colorings. In particular, there exists  $f$ satisfying the above properties such that the following equivalences hold,
\begin{equation*}
    \left(\mathcal{G}(n,p,k) \ has \ A \ w.h.p.  \right) \equiv  \left( \mathcal{G}(n,p) \ has \ Q_A^{f(n)} \ w.h.p. \right) \,,
\end{equation*}
\begin{equation*}
    \left( \mathcal{G}(n,m,k) \ has \ A \ w.h.p.  \right)  \equiv \left( \mathcal{G}(n,m) \ has \ Q_A^{f(n)} \ w.h.p. \right) \,.
\end{equation*}

The reasoning behind why such equivalences hold boils down to the following observation: for two events $A$ and $B$, if $A \wedge B$ is true w.h.p. then we must have that $A$ is true w.h.p. and $B$ is true w.h.p. In our setting, if $\mathbb{P}[A \wedge B] \geq 1 - f(n)$ for decreasing $f$ approaching zero then there must exist decreasing functions $f_1$ and $f_2$ both approaching zero such that $\mathbb{P}[A] \geq 1 - f_1(n)$ and $\mathbb{P}[B] \geq 1- f_2(n)$.

We consider the first equivalence, the second follows from the same argument. By Definition~\ref{rand graph prob model}, a graph sampled from $\mathcal{G}(n,p,k)$ has $A$ w.h.p.  is equivalent to saying that with probability $1 - o(1)$ a graph sampled from $\mathcal{G}(n,p)$ along with a uniformly random edge coloring has $A$. Therefore there must exist a function $f$ satisfying the properties above such that w.h.p. a graph sampled from $\mathcal{G}(n,p)$ has $Q_A^{f(n)}$. In other words, w.h.p. a graph $G$ sampled from $\mathcal{G}(n,p)$ has the property that for a uniformly random edge coloring $c$ on $G$, $(G, c)$ has $A$ w.h.p..

Considering the other direction, if we assume $\mathcal{G}(n,p)$ has $Q_A^{f(n)}$ w.h.p. then by Definition~\ref{rand graph prob model}, the probability that $\mathcal{G}(n,p,k)$ has $A$ is at least the probability that for $G$ sampled from $\mathcal{G}(n,p)$ uniformly at random has $Q^{f(n)}_A$ and for a uniformly random edge coloring c of $G$, $(G,c)$ has $A$. Thus, by conditional probability we have,
\begin{align*}
	\mathbb{P}[\mathcal{G}(n,p,k)& \ has \ A ] \\
    &\geq \mathbb{P}_{G \sim \mathcal{G}(n,p)}\left[ G \ has \ Q_A^{f(n)} \wedge (G,c) \ has \ A \ for \ uniformly  \ random \ edge \ coloring \ c \right] \\
	&\geq (1- f(n))(1- o(1)) \\
	&= 1 - o(1) \,.
\end{align*}

Therefore, when $p = \Theta(n^{-1})$ and 
$m = {n \choose 2} p$ we can utilize the standard equivalence 
between $\mathcal{G}(n,p)$ and $\mathcal{G}(n,m)$ for monotone properties. 
In particular, we can utilize the following Theorem from Łuczak \cite{Equiv-rand-graphs} (see Theorem 1.4 in \cite{IntroRandGraphsFrieze}). 

\begin{theorem}[Theorem 1.4 in \cite{IntroRandGraphsFrieze}]
	Let $0 \leq p_0 \leq 1$, $s(n) = n \sqrt{p(1-p)} \rightarrow \infty$, and $\omega(n) \rightarrow \infty$ arbitrarily slowly then the following holds,
	\begin{itemize}
		\item Suppose that $\mathcal{P}$ is a graph property such that $\lim_{n \rightarrow \infty} \mathbb{P}[\mathcal{G}(n,m) \ has \ \mathcal{P}]  = p_0$ for all 
		\begin{equation*}
			m \in \left[ {n \choose 2} p - \omega(n)s(n), {n \choose 2}p + \omega(n)s(n) \right] \,,
		\end{equation*}
		then, we have $\lim_{n \rightarrow \infty } \mathbb{P}[\mathcal{G}(n,p) \ has \ \mathcal{P}] = p_0$.
		\item Let $p_{-} = p - \omega(n)s(n)/n^2$ and $p_{+} = p + \omega(n)s(n)/n^2$. Suppose that $\mathcal{P}$ is a monotone graph property such that $\lim_{n \rightarrow \infty} \mathbb{P}[\mathcal{G}(n, p_{-}) \ has \ \mathcal{P}] = p_0 = \lim_{n \rightarrow \infty} \mathbb{P}[\mathcal{G}(n,p_{+}) \ has \ \mathcal{P}]$, then for $m = \lfloor {n \choose 2} p \rfloor$ we have $\lim_{n \rightarrow \infty} \mathbb{P}[\mathcal{G}(n,m) \ has \ \mathcal{P}] = p_0$.
	\end{itemize}
\end{theorem}

For more on the relationship between $\mathcal{G}(n,p)$ and $\mathcal{G}(n,m)$ we refer the reader to \cite{bolobas-rand-graph}, \cite{IntroRandGraphsFrieze}, and \cite{Equiv-rand-graphs}.

\section{Characterizing Adaptable 2-colorability Proofs} \label{appendix: adapt 2-col proofs}

In this section we prove Lemma~\ref{lma: non-adapt 2-col condition} which we restate below. 

\begin{lemma*}[\ref{lma: non-adapt 2-col condition}]
	\mbox{}
	\begin{enumerate}[label=\alph*.]
    	\item Every proper alternating odd bicycle is not adaptably 2-colorable. 
		\item Every improper alternating bicycle is adaptably 2-colorable.
	\end{enumerate}
\end{lemma*}

\begin{proof}
    We note that b) follows directly from Theorem~\ref{thm: adversary adpt 2 col}.
	
    We now consider a). Let $H$ be an alternating odd bicycle formed from an alternating path $P = v_1 v_2 v_3 \cdots v_s$ and two edges $v_1v_i$ and $v_sv_j$ with $i \leq j$, $i \equiv 1 \mod{2}$, and $j \equiv s \mod{2}$ such that $c(v_1 v_i) \neq c(v_1v_2)$ and $c(v_sv_j) \neq c(v_{s-1}v_s)$. Assume without loss of generality that $c(v_1 v_2)$ is red, which by the alternating structure of $H$ implies the colors of all edges in $H$.
    
    First, we claim that $v_i$ must be assigned red. Assume otherwise, that is assign $v_i$ the color blue. So as $c(v_1 v_i) \neq c(v_1v_2) = red$ we must have that $v_1$ gets assigned red and $v_2$ gets assigned blue. In fact, inductively by the alternating property of $P$, all vertices in $\{v_1, v_2, \dots, v_i\}$ of an odd index are colored $red$ and all vertices of an even index are colored $blue$. Thus, contradicting our assumption that $v_i$ is colored $blue$. So, $v_i$ must be assigned red. 
    
    As $c(v_1v_2) = red$, by the alternating structure of $P$ and as $i$ is odd, $c(v_iv_{i+1}) = red$. Hence as $v_i$ is assigned red, this forces the coloring of all the vertices in $\{v_{i+1}, v_{i+2}, \dots v_{s-1}, v_s\}$. Particularly, all vertices of an odd index are colored $red$ and all vertices of an even index are colored $blue$. Moreover, by the alternating structure of $P$, we have that if $s$ is odd, as $c(v_1v_2) = red$ we must have that $c(v_{s-1}v_s) = blue$; similarly, if $s$ is even we must have that $c(v_{s-1}v_s) = red$. Thus as $c(v_sv_j) \neq c(v_{s-1}v_s)$, and $j \equiv s \mod{2}$, we have that $v_s$ and $v_j$ must receive $c(v_sv_j)$. Hence, $H$ is not adaptably 2-colorable as required.
\end{proof}

\section{Long Alternating Paths} \label{appendix: long alternating paths}
In this section we prove Lemma~\ref{lma: alt path existence sub linear} which we restate.

\begin{lemma*}[\ref{lma: alt path existence sub linear}]
    Let $\epsilon = \epsilon(n)> 0$ be such that $\epsilon \gg n^{-1/3} \log^{1/3}n$, then w.h.p. $\mathcal{G}(n,p= 2(1 + \epsilon)/n, 2)$ contains an alternating path of length at least $\epsilon^2 n / 5$.
\end{lemma*}

Let us now provide the DFS algorithm, which will produce an alternating path. 
As in \cite{longest-path-simple}, let $S$ denote the set of vertices which have been fully 
explored, let $T$ denote the set of unvisited vertices. 
We will use a stack $U$ to track the vertices of our current candidate for such a path. 
Furthermore, for $u \in U$ we will let $\psi(u)$ denote the edge color required to extend the alternating path at $u$.
Therefore, when searching for potential ways to augment the path, we check $T \cap N_{\psi(u)}(u)$. 

Let $\tau$ be an arbitrary ordering on $V$. Our process is,

\begin{process}[DFS]
    \mbox{}
    \begin{enumerate}
        \item Initialize: 
        \begin{enumerate}[label=\alph*., ref=\theenumi\alph*]
            \item $U \gets \emptyset$, $S \gets \emptyset$, $T \gets V$.
        \end{enumerate}
        \item If $U = \emptyset$: \label{DFS line U empty}
        \begin{enumerate}[label=\alph*., ref=\theenumi\alph*]
            \item Let $u \in T$ be minimal according to $\tau$.
            \item Push $u$ to $U$ and let $\psi(u) = red$.
            \item $T \gets T \setminus \{u\}$.
        \end{enumerate}
        \item If $U \neq \emptyset$: \label{DFS line U not empty}
        \begin{enumerate}[label=\alph*., ref=\theenumi\alph*]
            \item Let $u$ be the last vertex added to $U$. \label{DFS line: (u,i)}
            \item If $N_{\psi(u)}(u) \cap T = \emptyset$: \label{DFS line: empty}
            \begin{enumerate}[label=\roman*.]
                \item Pop $u$ off of $U$.
                \item $S \gets S \cup \{u\}$
            \end{enumerate}
            \item If $N_{\psi(u)}(u) \cap T \neq \emptyset$: \label{DFS line: not empty}
            \begin{enumerate}[label=\roman*.]
                \item Let $v \in N_{\psi(u)}(u) \cap T$ be minimal with respect to $\tau$.
                \item Push $v$ to $U$ and let $\psi(v) \in \{red, blue\} \setminus \{\psi(u)\}$.
                \item $T \gets T \setminus \{v\}$.
            \end{enumerate}
        \end{enumerate}
    \end{enumerate}
\end{process}
When referring to a \textit{round} we mean the complete execution of either step~\ref{DFS line U empty} 
or step~\ref{DFS line U not empty}. Note that this is precisely when a vertex is added to $U$ (step~\ref{DFS line U empty} or \ref{DFS line: not empty}) or removed from $U$ (step~\ref{DFS line: empty}).

As in \cite{longest-path-simple} we have the following observations regarding the process, 
\begin{itemize}
    \item During each round, either a vertex is moved from $T$ to $U$ or a vertex is moved from $U$ to $S$.
    \item At every point in time, the vertices in $U$ form an alternating path.
    \item At every point in time, $S, T$, and $U$ are vertex-disjoint.
\end{itemize}

We now consider running the above process on $\mathcal{G}(n,p,2)$. 
Step~\ref{DFS line U not empty} can be implemented as follows: 
let $u$ be the last vertex added to $U$, 
then query the following random variables in order according to $\tau$ on $v \in T$,
\[ X_{uv}^{\psi(u)} = \begin{cases} 
      1 & uv \in E \ \wedge \ c(uv) = \psi(u) \\
      0 & otherwise 
   \end{cases} \,,
\]
stopping once we find a $v \in T$ such that $X_{uv}^{\psi(u)} = 1$. Thus, if we check every vertex in $T$, and all the variables are $0$, it follows that $N_{\psi(u)}(u) \cap T = \emptyset$; 
otherwise, we have found $v \in N_{\psi(u)}(u) \cap T$ minimal with respect to $\tau$. 
Furthermore, although not vital, if we find $X_{uv}^{\psi(u)} = 1$ then we can keep track of 
where $v$ is according to $\tau$, and thus if we consider $u$ again, that is the path died out, we 
can start querying where we left off according to $\tau$. This ensures we don't query the same 
variable multiple times. 

First, we require the following property.

\begin{observation}[\ref{ob: long path vars considered}]
    For each pair of vertices $u, v \in V$ at most one of $X_{uv}^{red}$ and $X_{uv}^{blue}$ is ever revealed. 
\end{observation}

This follows from step~\ref{DFS line U not empty}, as we only consider pairs between $U$ 
and $T$ and a vertex can only be added to $U$ once and thus $\psi(u)$ is set at most one time, 
i.e. $\psi(u)$ does not change at a later time. That is for a vertex $u$ in $U$ we only ever consider $X_{uv}^{\psi(u)}$ for $v \in T$. 

The remaining of the proof follows directly from the work of Krivelevich and Sudakov in \cite{longest-path-simple}; however we state it here for completeness.

We make the following observation regarding the process. 

\begin{observation} \label{ob: num of rand var considered}
    At every point in time, for the current sets $S$ and $T$, we have that for each 
    $s \in S$ and $t \in T$ we have revealed $X_{st}^j$ for some $j \in \{red, blue\}$ and such 
    $X_{st}^j = 0$.
\end{observation}

When referring to $X_{st}^j$ for some color $j$, if the color has no effect we 
write $X_{st}$ as opposed to $X_{st}^j$ for ease of notation. Particularly, we have revealed
$|S||T|$ random variables all of which returned $0$, or false. To see this, consider the point 
in time in which $s$ was added to $S$ and let $S'$ and $T'$ be the corresponding sets at this 
time. 
Therefore, $s$ was added to $S'$ because $X_{st'}^{\psi(s)} = 0$ for each $t' \in T'$. Thus, the result 
follows as $T \subseteq T'$.

Let $X_\gamma$ denote the $\gamma^{th}$ random variable revealed. Therefore, after $t$ random variables have been revealed, if $T \neq \emptyset$ we have,
\begin{equation} \label{eq: long path 1}
    |S \cup U| \geq \sum_{\gamma=1}^t X_\gamma \,,
\end{equation}
This follows, because each positive $X_\gamma$ results in a vertex being moved from $T$ to $U$, and as vertices flow from $T \xrightarrow[]{} U \xrightarrow[]{} S$. 
Also, since vertices are moved from $T$ to $U$ only due to the occurrence of a positive query, except for the first vertex in each connected component, we have that
\begin{equation} \label{eq: long path 2}
    |U| \leq 1 + \sum_{\gamma = 1}^t X_\gamma \,.
\end{equation}

As we are considering a sequence $(X_\gamma)_{\gamma \in \mathbb{N}}$ of independent random variables each of which are $1$ with probability $p/2 = (1 + \epsilon)/n$ and $0$ otherwise, we can use the following concentration result.

\begin{lemma} \label{lma: concentration for long path}
	Let $\epsilon = \epsilon(n) \xrightarrow[]{} 0$ be such that $\epsilon \gg n^{-1/3} \log^{1/3}n$ and $l \geq (\epsilon n \log{n})^{1/2}$ and let $(X_i)_{i = 1}^N$ be a sequence of i.i.d. Bernoulli random variables with parameter $p = \frac{1 + \epsilon}{n}$. For $N_0 = \frac{\epsilon n^2}{2}$,
	\begin{equation*}
		\mathbb{P}\left[\left|\sum_{i=1}^{N_0} X_i - \frac{\epsilon (1 + \epsilon)n}{2}\right| \geq  l\right] = o(1) \,.
	\end{equation*}
\end{lemma}

\begin{proof}
    Clearly, it suffices to prove the statement for $l = (\epsilon n \log{n})^{1/2}$. 
    As $\sum_{i=1}^{N_0} X_i \sim \mathcal{BIN}(N_0, p)$ with expectation $N_0p = \frac{\epsilon (1+\epsilon)n}{2}$ we have by the \hyperref[Chernoff Bound]{Binomial Chernoff Bound} that,
    \begin{equation*}
        \mathbb{P}\left[\left|\sum_{i=1}^{N_0} X_i - \frac{\epsilon (1 + \epsilon)n}{2}\right| \geq  (\epsilon n \log{n})^{1/2}\right] 
        \leq 2 n^{-\frac{2}{3 (1 + \epsilon)}} = o(1) \,.
    \end{equation*}
\end{proof}

We are now ready to prove Lemma~\ref{lma: alt path existence sub linear}. 
Note that by Observation~\ref{ob: long path vars considered}, each random variable $X_\gamma = 1$ with probability $p/2 = (1 + \epsilon)/n$. 
Using the concentration bound in Lemma~\ref{lma: concentration for long path} the proof of Lemma~\ref{lma: alt path existence sub linear} follows the proof of Theorem 1 in \cite{longest-path-simple}, 
which, we provide for completeness. 

\begin{proof}[Proof of Lemma~\ref{lma: alt path existence sub linear}]
    Let $l = (\epsilon n \log{n})^{1/2}$. We claim that the moment after the 
    first $N_0 = \frac{\epsilon n^2}{2}$ random variables are revealed, the set $U$ contains at 
    least $\frac{\epsilon^2 n}{5}$ vertices; call this point in time $\sigma_1$. 
    
    First, we will show that w.h.p. $|S| < n/3$ at time $\sigma_1$. 
    Assume for a contradiction $|S| \geq n/3$. Thus, consider
    a time $\sigma_2 \leq \sigma_1$ at which $|S| = n/3$, which must exist as $S$ is non-decreasing and increases by at 
    most 1 during each round. As the number of random variables revealed at time $\sigma_2$ is upper bounded by $N_0$ it follows by (\ref{eq: long path 2}) and 
    Lemma~\ref{lma: concentration for long path} that w.h.p., 
    $|U| \leq 1 + \sum_{i=1}^{N_0} X_i \leq 1 + \frac{\epsilon(1 + \epsilon)n}{2} + l < n/3$, where 
    the last inequality holds asymptotically by choice of $l$ and small enough $\epsilon$.
    Thus, $|T| = n - |S| - |U| \geq n/3$. By Observation~\ref{ob: num of rand var considered}, all $|S||T| \geq n^2/9 > N_0$ random 
    variables $X_{st}$ with $s \in S$, $t \in T$ have been revealed, 
    contradicting $\sigma_2 \leq \sigma_1$. 

    Thus, we may assume that at time $\sigma_1$, $|S| < n/3$ and $|U| < \epsilon^2 n/5$ (as otherwise we are done). Thus $|T| = n - |S| - |U| > 0$. So by Lemma~\ref{lma: concentration for long path} and (\ref{eq: long path 1}) we have $|S \cup U| \geq \sum_{i=1}^{N_0}X_i \geq \frac{\epsilon (1 + \epsilon)n}{2} - l$. 
    Thus, as $|U| < \epsilon^2n/5$ and $|S \cup U| = |S| + |U|$, $|S| \geq \frac{\epsilon n}{2} + \frac{3 \epsilon^2 n}{10} - l$. 
    By Observation~\ref{ob: num of rand var considered}, all $|S||T| = |S|(n - |S| - |U|) \geq |S|\left(n - |S| - \frac{\epsilon^2 n}{5}\right) $ random variables $X_{st}$ for $s \in S, t\in T$ have been revealed. Thus,
    \begin{align*}
        \frac{\epsilon n^2}{2} &= N_0 \geq |S||T| \geq \left( \frac{\epsilon n}{2} + \frac{3 \epsilon^2 n}{10} - l \right)\left(n - \frac{\epsilon n}{2} - \frac{\epsilon^2 n}{2} + l\right) \\
        &= \frac{\epsilon n^2}{2} + \frac{\epsilon^2 n^2}{20} - \mathcal{O}(\epsilon^3 n^2) - \mathcal{O}(nl) \\
        &> \frac{\epsilon n^2}{2} \,,
    \end{align*}
    where the third inequality holds as for $|S| < \frac{n}{3}$ we have from elementary calculus that \\
    $|S|\left(n - |S| - \frac{\epsilon^2 n}{5}\right)$ is increasing in $|S|$ and therefore minimized at the lower bound of $\frac{\epsilon n}{2} + \frac{3 \epsilon^2 n}{10} - l$ on $|S|$. Moreover, the last inequality follows as $\epsilon \gg n^{-1/3} \log^{1/3}n$, which implies that $\epsilon^2n^2 \gtrsim n (\epsilon n \log{n})^{1/2} = nl$, that is the $\Theta(\epsilon^2 n^2)$ contributes strictly more than the $\mathcal{O}(nl)$ term takes away, for large enough $n$. Thus, we have a contradiction as required.
\end{proof}

\begin{corollary} \label{cor: alt path in edge model}
    Let $\epsilon = \epsilon(n)> 0$ be such that $\epsilon \gg n^{-1/3} \log^{1/3}n$, then w.h.p. $\mathcal{G}(n,m = (1 + \epsilon)n, 2)$ contains an alternating path of length at least $\epsilon^2 n / 5$.
\end{corollary}

Before proceeding, we emphasize the information that was revealed during this process, 
particularly regarding vertices in the alternating path. Consider the stack $U$, i.e. the alternating path.
Let $x$ and $y$ be two non-adjacent vertices in $U$, with $x$ below $y$ in the stack. 
It follows by Observation~\ref{ob: num of rand var considered}, that for $col \neq \psi(x)$ the random variable
$X_{xy}^{col}$ was not revealed. It remains to consider whether $X_{xy}^{\psi(x)}$ was revealed. 

Assume $X_{xy}^{\psi(x)}$ was revealed and equates to $1$ and consider the moment which this occurs. That is,
$x$ is the vertex at the end of the stack and $y \in T$. However, this means that $y$ would be pushed onto $U$ and would 
thus be adjacent to $x$, which is a contradiction. This leads us to the following observation.

\begin{observation*}[\ref{ob: alt path prob change}]
    For two non-adjacent vertices $x$ and $y$ in $U$, with $x$ before $y$, the following hold: 
    \begin{enumerate}[label=\alph*.]
        \item If $X_{xy}^{\psi(x)}$ was revealed then, it must equate to $0$. 
        \item For $col \neq \psi(x)$, $X^{col}_{xy}$ was not revealed. 
    \end{enumerate}
\end{observation*}

\section{Initial Bounds on the Scaling Window} \label{appendix: Critical Window Bounds}

\subsection{Lower Bound of the Scaling Window}

In this section, we prove Theorem~\ref{thm: scaling window lower bound}. 

\begin{theorem*} [\ref{thm: scaling window lower bound}]
    For $\lambda_n \xrightarrow[]{} \infty$ arbitrarily slowly, $\mathcal{G}(n, p = 2(1 - \lambda_n n^{-1/3})/n, 2)$ is w.h.p. adaptably 2-colorable.
\end{theorem*}

First, we require an analytic result regarding sums of unimodal functions (see. Section 4.2 in \cite{lect-notes-unimodal-sums}). Particularly, we call a twice differentiable function $f: \mathbb{R} \xrightarrow[]{} \mathbb{R}$ over a closed interval $I \subseteq \mathbb{R}$ \textit{unimodal} if $f'$ has a unique zero $c \in I$ and $f''(c) < 0$.

\begin{lemma}[\cite{lect-notes-unimodal-sums}] \label{lma: unimodal bound}
    If $f : \mathbb{R} \xrightarrow[]{} \mathbb{R}_{\geq 0}$ is unimodal over $I = [1, \infty)$ then there exists an absolute constant $c > 0$ such that for all $n \geq 1$:
    \begin{equation*}
        \sum_{i= 1}^n f(i) \sim \int_{1}^n f(x) \,d x + c \cdot \max_{x \in [1,n]} f(x) \,.
    \end{equation*}
\end{lemma}

\begin{proof}[Proof of Theorem~\ref{thm: scaling window lower bound}]
    Let $p = 2(1 - \lambda_nn^{-1/3})/n$ for $\lambda_n \xrightarrow[]{} \infty$ with $\lambda_n \ll n^{1/3}$. As per the proof of a) in Theorem~\ref{thm: a sharp threshold}, let $X$ denote the number of improper alternating bicycles and thus,
    \begin{align} \label{eq: scaling window lower bound 1}
        \begin{split}
        \mathbb{E}[X] &\leq \sum_{s = 1}^n n^s s^2 p^{s+1} 2^{-s} = \frac{2(1 - \lambda_n n^{-1/3})}{n} \sum_{s=1}^n s^2 \left( 1 - \lambda_n n^{-1/3} \right)^s \\
        &\sim \frac{2(1 - \lambda_n n^{-1/3})}{n} \sum_{s=1}^n s^2 e^{-s \lambda_n n^{-1/3}} \,.
        \end{split}
    \end{align}
    Elementary calculus shows that $\frac{d}{dx} x^2 e^{-x \lambda_n n^{-1/3}} = x e^{-x \lambda_n n^{-1/3}} \left( 2 - x \lambda_n n^{-1/3} \right)$ and hence, it's not hard to see that $x^2 e^{-x \lambda_n n^{-1/3}}$ is unimodal on $[1, \infty)$ and maximized at $x = 2n^{1/3} / \lambda_n$. Thus by Lemma~\ref{lma: unimodal bound} we have that (\ref{eq: scaling window lower bound 1}) is,
    \begin{align*}
         &= \frac{2(1 - \lambda_n n^{-1/3})}{n} \left( \int_{1}^n x^2 e^{-x \lambda_n n^{-1/3}} \,d x  + \mathcal{O}\left( \frac{n^{2/3}}{\lambda_n^2} \right) \right) \\
         &= \frac{2(1 - \lambda_n n^{-1/3})}{n} \left( \int_{\lambda_n n^{-1/3}}^{\lambda_n n^{2/3}} \left( \frac{u n^{1/3}}{\lambda_n} \right)^2 e^{-u} \frac{n^{1/3}}{\lambda_n} \,d u  + \mathcal{O}\left( \frac{n^{2/3}}{\lambda_n^2} \right) \right) \\
         &= \frac{2(1 - \lambda_n n^{-1/3})}{\lambda_n^3} \left( \int_{\lambda_n n^{-1/3}}^{\lambda_n n^{2/3}} u^2 e^{-u} \,d u  + \mathcal{O}\left( \frac{\lambda_n}{n^{1/3}} \right) \right) \\
         &\sim \frac{2(1 - \lambda_n n^{-1/3})}{\lambda_n^3} \left( \Gamma(3) + o(1)\right) \\
         &= o(1) \,,
    \end{align*}
    where the second equality follows from $u$-substitution with $u = \lambda_n x n^{-1/3}$ and the asymptotic equality follows as $n \xrightarrow[]{} \infty$, $\int_{\lambda_n n^{-1/3}}^{\lambda_n n^{2/3}} u^2 e^{-u} \,d u$ approaches the Gamma function $\Gamma(3)$ and as $\lambda_n \ll n^{1/3}$ we have $\mathcal{O}\left( \frac{\lambda_n}{n^{1/3}} \right) = o(1)$. The final equality follows as $\lambda_n \xrightarrow[]{} \infty$ and $\lambda_n \ll n^{1/3}$.
\end{proof}

We note that similar calculations appear in \cite{2-sat-critical-window}.

\subsection{Upper Bound on the Scaling Window}

The purpose of this section is to prove Theorem~\ref{thm: scaling window upper bound}. 

\begin{theorem*} [\ref{thm: scaling window upper bound}]
    For $\lambda_n \xrightarrow[]{} \infty$ arbitrarily slowly, $\mathcal{G}(n, p = 2(1 + \lambda_n n^{-1/5})/n, 2)$ is w.h.p. adaptably 2-colorable.
\end{theorem*}

\begin{proof}[Proof of Theorem~\ref{thm: scaling window upper bound}]
    We follow the proof of Theorem~\ref{thm: a sharp threshold} (b). \\
    We first consider $\mathcal{G}(n, p = \frac{2(1 + \lambda_n n^{-1/5} / 2)}{n})$. 
    By Lemma~\ref{lma: alt path existence sub linear} there exists w.h.p. an alternating path $P$
    of length $\frac{\lambda_n^2}{20}n^{3/5}$.

    We utilize a ``sprinkling" argument. We will increase $p$ to $2(1 + \lambda_n n^{-1/5})/n$; that is,
    we sprinkle edges with probability $\sim \lambda_n / n^{6/5}$. 
    We will show w.h.p. these sprinkled edges form an alternating odd bicycle along $P$.

    To simplify calculations significantly, we only consider the odd indexed vertices of $P$. Noting that 
    asymptotically this only results in a constant factor loss. 
    We partition the set of odd indices into two sets $I_1$ and $I_2$ of roughly equal size,
    \begin{equation*}
        I_1 = \left\{i \mid i \ odd \wedge i \leq \frac{\lambda_n^2}{40}n^{3/5}\right\} \qquad and \qquad I_2 = \left\{i \mid i \ odd \wedge i > \frac{\lambda_n^2}{40}n^{3/5}\right\} \,.
    \end{equation*}
    So, $|I_1| \sim |I_2| \sim \frac{\lambda_n^2}{80}n^{3/5}$.

    Assume without loss of generality that the first edge in $P$ is colored $red$. Therefore, 
    if upon increasing $p$ we introduce a $blue$ edge between vertices in $I_1$ and a $red$ edge 
    between vertices in $I_2$ then its not hard to see an odd alternating bicycle is formed. 

    Before we consider increasing $p$, let us consider the number of edges already existing between vertices 
    in $I_1$ and $I_2$. We first note that by Observation~\ref{ob: alt path prob change} and Lemma~\ref{lma: prob change}, the information
    revealed during the construction of $P$ would not increase the probability of an edge existing between two vertices in $I_1$ or $I_2$. 
    Therefore, if $B_1$ denotes the number of existing edges between vertices in $I_1$ and $B_2$ denotes the number of existing edges between vertices in $I_2$, 
    it follows that 
    $B_1$ and $B_2$ are stochastically dominated by a binomial distribution with ${\frac{\lambda_n^2}{80}n^{3/5} \choose 2}$ 
    trials and 
    probability $\frac{2(1+\lambda_n n^{-1/5})}{n}$. Therefore for $i \in \{1,2\}$, 
    $\mathbb{E}[B_i] \leq {\frac{\lambda_n^2}{80}n^{3/5} \choose 2} \frac{2(1+\lambda_n n^{-1/5})}{n} \sim \frac{\lambda_n^4 n^{1/5}}{80^2} (1 + o(1))$.
    So by the \hyperref[Chernoff Bound]{Binomial Chernoff Bound},
    \begin{equation*}
        \mathbb{P}[B_i > \mathbb{E}[B_i] + n^{1/5}] \leq 2e^{-\Theta(n^{1/5} / \lambda_n^4)} = o(1) \,.
    \end{equation*}
    Particularly, w.h.p. $B_i = \mathcal{O}(n^{1/5} \lambda_n^4)$. 
    Let $\widetilde{I}_1$ denote the pairs of vertices in $I_1$ for which there does not already exist an edge. 
    Similarly, let $\widetilde{I}_2$ denote the pairs of vertices in $I_2$ for which there does not already exist an edge
    Therefore, w.h.p. for $i \in \{1,2\}$:
    \begin{equation*}
        |\widetilde{I}_i| = {|I_i| \choose 2} - B = {|I_i| \choose 2} (1 - o(1)) \,.
    \end{equation*}
    We now consider sprinkling edges with probability $\sim \frac{\lambda_n}{n^{6/5}}$. 
    Let $X_1$ denote the number of $blue$ edges between pairs of vertices in $\widetilde{I}_1$ which are introduced among the sprinkled edges. Similarly, 
    let $X_2$ denote the number of $red$ edges between pairs of vertices in $\widetilde{I}_2$ which are introduced among the sprinkled edges. 
    Therefore, $X_i$ follows a binomial distribution with $|\widetilde{I}_i|$ trials and probability 
    $\sim \frac{\lambda_n}{2n^{6/5}}$. It suffices to show that w.h.p. $X_1 \geq 1$ and $X_2 \geq 1$,
    \begin{equation*}
        \mathbb{P}[X_i = 0] \sim 
        \left(1 - \frac{\lambda_n}{2n^{6/5}}\right)^{{|I_i| \choose 2} (1 - o(1))} \\
        \sim e^{-\frac{\lambda_n}{2n^{6/5}}\frac{\lambda_n^4 n^{6/5}}{2 \cdot 80^2}(1-o(1))} \\
        = \Theta(e^{-\lambda_n^5}) = o(1) \,.
    \end{equation*}
    So, w.h.p. $X_1 \geq 1$ and $X_2 \geq 1$ and thus an alternating odd bicycle is formed.
\end{proof}

\section{An Expression for $R_{n,p}(k)$} \label{appendix: R(k)}

Here we prove Lemma~\ref{lma: R(k) expression}.

\begin{lemma*}[\ref{lma: R(k) expression}]
    If $\lambda_n \xrightarrow[]{} \infty$, $\lambda_n \ll n^{1/3}$, and $k/n^{2/3} \xrightarrow[]{} 0$ then for $p = \frac{2(1 - \lambda_n n^{-1/3})}{n}$:
    \begin{equation*}
        R_{n,p}(k) = \frac{2}{pn\sqrt{2\pi}k^{3/2}} e^{-\frac{k}{2}(\lambda_n n^{-1/3})^2\left(1 \pm o(1)\right)}
    \end{equation*}
\end{lemma*}

First, we derive the following expression for $R_{n,p}(k)$.

\begin{lemma} \label{lma: R(k) equals}
    For $x \in V$ and $\gamma \in \{red, blue\}$:
    \begin{equation*}
        R_{n,p}(k) \coloneqq R_{n,p}(x, k) = \frac{1}{n} {n \choose k} \left(\frac{kp}{2}\right)^{k-1} \left(1 - \frac{p}{2}\right)^{{k \choose 2} - k + 1 + k(n - k)} 
    \end{equation*}
\end{lemma}

\begin{proof}
     As per \cite{2-sat-critical-window} and Cayley's Theorem, the number of trees on $k$ vertices is $k^{k-2}$. Thus, 
    \begin{equation*}
        \mathbb{P}[\mathcal{G}(k,p/2) \text{ is a tree}] = k^{k-2}\left(\frac{p}{2}\right)^{k-1} \left(1 - \frac{p}{2}\right)^{{k \choose 2} - k + 1} \,.
    \end{equation*}
    So, for $X \in {V \choose k}$ with $x \in X$, 
    \begin{equation*}
        \mathbb{P}[V(\widetilde{G}_i(x)) = X] = \left(1 - \frac{p}{2}\right)^{k(n-k)} \,.
    \end{equation*}
    As there are $k(n-k)$ potential edges between $X$ and $V \setminus X$, and as each possible edge $uv$ with $u \in X$ and $v \notin X$ has precisely one color, namely $\psi(u)$, which if it were assigned, would extend $\widetilde{G}_i(x)$. This also implies that the events $|V(\widetilde{G}_i(x))| = k$ and $\widetilde{G}_i(x)$ is a tree are independent. Thus,
    \begin{align*}
        \mathbb{P}\left[|V(\widetilde{G}_i(x))| = k \ \bigwedge \widetilde{G}_i(x) \text{ is a tree}\right] &= \mathbb{P}[|V(\widetilde{G}_i(x))| = k] \cdot \mathbb{P}[\widetilde{G}_i(x) \text{ is a tree}] \\
        &= \sum_{X \in {V \choose k}, x \in X} \mathbb{P}[V(\widetilde{G}_i(x)) = X ] \cdot \mathbb{P}[\widetilde{G}_i(x) \text{ is a tree}] \\
        &= {n - 1 \choose k - 1} \left(1 - \frac{p}{2}\right)^{k(n-k)} k^{k-2}\left(\frac{p}{2}\right)^{k-1} \left(1 - \frac{p}{2}\right)^{{k \choose 2} - k + 1} \\
        &= \frac{1}{k} {n - 1 \choose k - 1} \left(\frac{kp}{2}\right)^{k-1} \left(1 - \frac{p}{2}\right)^{{k \choose 2} - k + 1 + k(n - k)} \\
        &= \frac{1}{n} {n \choose k} \left(\frac{kp}{2}\right)^{k-1} \left(1 - \frac{p}{2}\right)^{{k \choose 2} - k + 1 + k(n - k)} 
    \end{align*}
\end{proof}

We now consider the following lemma which provides Lemma~\ref{lma: R(k) expression} as a direct corollary.

\begin{lemma} \label{lma: R(k) expression 2}
    For $p = \frac{2(1 + \lambda_n n^{-1/3})}{n}$ we have,
    \begin{equation*}
        R_{n,p}(k) = \frac{2}{pn\sqrt{2\pi}k^{3/2}} e^{-\frac{k}{2}(\lambda_n n^{-1/3})^2 +\frac{k^2}{2n}\lambda_n n^{-1/3} + \mathcal{O}\left(\frac{k}{n}\right) - \frac{k^3}{6n^2} + k\Theta\left(\frac{\lambda_n^3}{n}\right) - \Theta\left(\frac{k^4}{n^3}\right) - \Theta\left(\frac{1}{k}\right)} 
    \end{equation*}
\end{lemma}

\begin{proof}
    We note that the proof of this Lemma is a direct adaptation of the computations done in Section 6 of \cite{2-sat-critical-window}. By Lemma~\ref{lma: R(k) equals}, we have
    \begin{equation*}
        R_{n,p}(k) = \frac{1}{n} {n \choose k} \left(\frac{kp}{2}\right)^{k-1} \left(1 - \frac{p}{2}\right)^{{k \choose 2} - k + 1 + k(n - k)} 
    \end{equation*}
    Now observing that $\frac{n!}{(n-k)!} = n^k \prod_{i = 1}^{k-1} (1 - \frac{i}{n})$ we have,
    \begin{equation} \label{eq: R(k) equation 1}
        \frac{1}{n} {n \choose k} \left(\frac{kp}{2}\right)^{k-1} =  \frac{2}{pnk} \frac{(k/e)^k}{k!} \left(\frac{pne}{2}\right)^k \prod_{i = 0}^{k-1} \left(1 - \frac{i}{n}\right) 
    \end{equation}
    Moreover, by Stirling's Approximation, we have
    \begin{equation*}
        \frac{1}{k} \frac{(k/e)^k}{k!} = \frac{1}{k^{3/2} \left(1 + \frac{1}{12k} + \mathcal{O}\left(\frac{1}{k^2}\right)\right)} = \frac{e^{-1/(12k + \delta_k)}}{k^{3/2}}
    \end{equation*}
    for some $0\leq \delta_k \leq 1$ for each $k$. Thus, \ref{eq: R(k) equation 1} is precisely,
    \begin{equation} \label{eq: R(k) equation 2}
        \frac{2e^{-1/(12k + \delta_k)}}{pn\sqrt{2\pi}k^{3/2}} \left(\frac{pne}{2}\right)^k \prod_{i = 0}^{k-1} \left(1 - \frac{i}{n}\right) \,.
    \end{equation}
    noting that $\frac{pn}{2} = 1 + \lambda_n n^{-1/3}$, $\log{(1 + \lambda_n n^{-1/3})} = \lambda_n n^{-1/3} - \frac{1}{2}(\lambda_n n^{-1/3})^2 + \Theta((\lambda_n n^{-1/3})^3)$, and $\log{(1 - \frac{i}{n})} = -\frac{i}{n} - \frac{i^2}{2n^2} - \Theta\left(\frac{i^3}{n^3}\right)$ we get that \ref{eq: R(k) equation 2} equates to,
    \begin{equation} \label{eq: R(k) equation 3}
        \frac{2}{pn\sqrt{2\pi}k^{3/2}} e^{-\Theta(1/k)} e^{k(1 + \lambda_n n^{-1/3}) - k\frac{1}{2}(\lambda_n n^{-1/3})^2 + k\Theta(\frac{\lambda_n^3}{n})} e^{-\frac{k^2}{2n} + \Theta(\frac{k}{n}) - \frac{k^3}{6n^2} + \Theta(\frac{k^2}{n^2}) - \Theta(\frac{k^4}{n^3})}
    \end{equation}
    Now, we consider 
    \begin{equation} \label{eq: R(k) equation 4}
        \left(1 - \frac{p}{2}\right)^{{k \choose 2} - k + 1 + k(n - k)} \sim e^{-\frac{p}{2} \left(kn - \frac{k^2}{2} - k + 1\right)} = e^{-k(1 + \lambda_n n^{-1/3}) + \frac{k^2}{2n}(1 + \lambda_n n^{-1/3}) + \frac{k}{n}(1 + \lambda_n n^{-1/3}) + \mathcal{O}(\frac{1}{n})} 
    \end{equation}
    Thus, combining \ref{eq: R(k) equation 3} and \ref{eq: R(k) equation 4} we get,
    \begin{align*}
        R_{n,p}(k) &= \frac{2}{pn\sqrt{2\pi}k^{3/2}} e^{-\frac{k}{2}(\lambda_n n^{-1/3})^2 +\frac{k^2}{2n}\lambda_n n^{-1/3} + \mathcal{O}\left(\frac{k}{n}\right) - \frac{k^3}{6n^2} + k\Theta\left(\frac{\lambda_n^3}{n}\right) - \Theta\left(\frac{k^4}{n^3}\right) - \Theta\left(\frac{1}{k}\right)} \\
        &= \frac{2}{pn\sqrt{2\pi}k^{3/2}} e^{-\frac{k}{2}\epsilon^2 +\frac{k^2}{2n}\epsilon + \mathcal{O}\left(\frac{k}{n}\right) - \frac{k^3}{6n^2} + k\Theta\left(\epsilon^3\right) - \Theta\left(\frac{k^4}{n^3}\right) - \Theta\left(\frac{1}{k}\right)}
    \end{align*}
    for $\epsilon = \lambda_n n^{-1/3}$.
\end{proof}

\end{document}